\documentclass[11pt,reqno]{amsart}
\usepackage{amsmath,amsthm,amsfonts,amssymb,mathrsfs,bm,graphicx,stmaryrd}
\usepackage{mathtools}
\usepackage{dsfont}
\usepackage{multicol}
\usepackage[colorlinks=true,linkcolor=blue]{hyperref}
\usepackage[letterpaper,hmargin=1.0in,vmargin=1.0in]{geometry}
\parskip 	\smallskipamount

\newtheorem{theorem}{Theorem}[section]
\newtheorem{lemma}[theorem]{Lemma}
\newtheorem{corollary}[theorem]{Corollary}
\newtheorem{proposition}[theorem]{Proposition}
\newtheorem{observation}[theorem]{Observation}
\newtheorem{remark}[theorem]{Remark}
\newtheorem{maintheorem}{Theorem}
\def\N{\mathbb{N}}
\def\L{\mathbb{L}}
\def\P{\mathbb{P}}
\def\Z{\mathbb{Z}}
\def\R{\mathbb{R}}
\def\E{\mathbb{E}}
\def\Var{{\rm Var}}
\newcommand{\cf}{\mathcal{F}}
\newcommand{\cA}{\mathcal{A}}
\newcommand{\cB}{\mathcal{B}}

\newcommand{\cD}{\mathcal{D}}
\newcommand{\cE}{\mathcal{E}}

\newcommand{\sL}{\mathcal{L}}
\newcommand{\ce}{\mathcal{E}}
\newcommand{\cce}{\mathscr{E}}

\newcommand{\cg}{\mathcal{G}}

\newcommand{\ch}{\mathcal{H}}
\newcommand{\Cov}{\mathrm{Cov}}
\newcommand{\br}{\mathbf{r}}
\newcommand{\bn}{\mathbf{n}}
\newcommand{\si}{\lambda}
\newcommand{\li}{\Lambda}

\newcommand{\don}{\mathds 1}

\DeclareMathOperator*{\argmax}{\arg\! \max}

\begin{document}
\title[Time Correlations in LPP]{Temporal Correlation  in Last Passage Percolation with Flat Initial Condition via Brownian Comparison}

\author{Riddhipratim Basu}
\address{Riddhipratim Basu, International Centre for Theoretical Sciences, Tata Institute of Fundamental Research, Bangalore, India}
\email{rbasu@icts.res.in}
\author{Shirshendu Ganguly}
\address{Shirshendu Ganguly, Department of Statistics, UC Berkeley, Berkeley, CA, USA}
\email{sganguly@berkeley.edu}
\author{Lingfu Zhang}
\address{Lingfu Zhang, Department of Mathematics, Princeton University, Princeton, NJ, USA}
\email{lingfuz@math.princeton.edu}

\date{}

\begin{abstract}
We consider directed last passage percolation on $\Z^2$ with exponential passage times on the vertices. A topic of great interest is the coupling structure of the weights of geodesics as the endpoints are varied spatially and temporally. A particular specialization is when one considers geodesics to points varying in the time direction starting from a given initial data. This paper considers the flat initial condition which corresponds to line-to-point last passage times. Settling a conjecture in \cite{FS16}, we show that for the passage times from the line $x+y=0$ to the points $(r,r)$ and $(n,n)$, denoted $X_{r}$ and $X_{n}$ respectively, as $n\to \infty$ and $\frac{r}{n}$ is small but bounded away from zero, the covariance satisfies $$\mbox{Cov}(X_{r},X_{n})=\Theta\left((\frac{r}{n})^{4/3+o(1)} n^{2/3}\right),$$ thereby establishing $\frac{4}{3}$ as the temporal covariance exponent. This differs from the corresponding exponent for the droplet initial condition recently rigorously established in \cite{FO18,BG18} and requires novel arguments. Key ingredients include the understanding of geodesic geometry and recent advances in quantitative comparison of geodesic weight profiles to Brownian motion using the Brownian Gibbs property. The proof methods are expected to be applicable for a wider class of initial data.
\end{abstract}

\maketitle
\tableofcontents

\section{Introduction}\label{first}
Interface models in one dimension that exhibit Kardar-Parisi-Zhang (KPZ) growth has been a topic of major interest both in statistical physics and probability theory in recent decades, with the large time scaling exponents for height fluctuations and spatial correlation decay predicted in the original work of KPZ \cite{KPZ86}  being verified in only a handful of exactly solvable models.
More recently, there has been interest in understanding the scaling limit of the full space time evolution of such growth models leading to fundamental works such as the construction of the KPZ fixed point \cite{MQR17} and more recently the space time Airy Sheet \cite{DOV18}. While much of these works use remarkable bijections from integrable probability that lead to exact distributional formulae for statistics of interest, in parallel, a more probabilistic approach, often coupled with limited integrable inputs, has proven to be quite fruitful. The present work falls in the latter category.

By now, the joint distribution of the associated height functions at different spatial locations at a given time has been studied to some depth, and going beyond, more recently, significant recent interest has been devoted to understanding the joint distribution of the profile at two (on-scale separated) time points starting from a general initial condition. There has been a number of recent works obtaining exact formulae for the two time joint distribution for a number of models in the KPZ universality class \cite{Joh17, Joh19, JR19,  BL19, L19+}. 
While there have been attempts at asymptotic analysis of these formulae \cite{LD18, johansson2019long},
they are typically quite involved, and it does not appear to be straightforward to extract asymptotic properties of the time evolution of the interface from such information. A natural and fundamental question one can ask about the two time distribution is to evaluate the correlation of the height at a given spatial location. This was considered in \cite{FS16} and predictions about the correlation exponents (when the two time points are close or far) were made using heuristic arguments backed by {experiments \cite{takeuchi2012evidence} and numerical simulations \cite{singha2005persistence, takeuchi2012statistics}}. For the step (droplet or narrow wedge) initial condition this prediction has now been rigorously confirmed using a number of different methods. In an unpublished work \cite{CH14+}, this was established for the model of Brownian last passage percolation using Brownian resampling (this approach has recently been extended to the KPZ equation in \cite{CGH19}). More recently, in two parallel and independent works \cite{FO18,BG18} this was established for the exponential last passage percolation. 

Among the various settings addressed non-rigorously in \cite{FS16}, of fundamental importance is the case of flat initial data which was predicted to behave differently than the droplet case yielding a different temporal correlation exponent when the two time points have large separation. 
The primary contribution of this paper is to establish rigorously the exponent alluded to above. Nonetheless, the arguments are rather robust and are expected to be useful in analyzing a broader class of initial data satisfying certain growth conditions. Study of general initial data has been central to several advances. We would particularly emphasize two separate approaches that are of key importance to this article.  The first one is \cite{MQR17} where the authors relying on Fredholm determinants characterized the Markov kernel describing the evolution of the height function in the well known Totally Asymmetric Exclusion Process (TASEP). The second line of works is for a Brownian model of last passage percolation where Hammond in a series of four papers culminating in \cite{H17b} established strong Brownian regularity properties for the height function started from a rather general class of initial data. Very recently the results were sharpened in \cite{HHJ+} and we shall make use of this recent progress.

The approach in this paper takes inspiration from several of the above works and at a broad level employs a method which is a hybrid of the routes taken in \cite{CH14+}, relying on fine Brownian comparison estimates for the Airy$_2$ process obtained by resampling arguments recently developed in \cite{HHJ+}, and the more geometric arguments implemented by the first two authors in \cite{BG18} to treat the droplet case. The methods have minimal dependence on the specific details of the model and are expected to work for other exactly solvable examples including Brownian last passage percolation (the model under consideration in \cite{CH14+,H17b}).  We shall not elaborate on this further, and instead move towards model definitions and statements of our main results.

\subsection{Model Definition and Main Results:}
We consider directed last passage percolation (LPP) on $\Z^2$ with i.i.d.\ exponential weights on the vertices, i.e., we have a random field $$\omega=\{\omega_{v}: v\in \Z^2\}$$ where $\omega_v$ are i.i.d.\ standard exponential variables. For any up/right path $\gamma$ from $u$ to $v$ where $u\preceq v$ (i.e., $u$ is co-ordinate wise smaller than $v$)  the weight of $\gamma$, denoted $\ell(\gamma)$ is defined by 
$$\ell(\gamma):=\sum_{w\in \gamma\setminus \{v\}} \omega_{w}.$$
For any two points $u$ and $v$ with $u\preceq v$, we shall denote by $T_{u,v}$ the last passage time from $u$ to $v$; i.e., the maximum weight among weights of all directed paths from $u$ to $v$\footnote{Notice that our definition is slightly non-standard as we exclude the last vertex in the weight, but this does not change the asymptotics and will be ignored from now on. We use this definition as it conveniently ensures that the weight of the concatenation of two paths is the sum of the individual weights.}. By $\Gamma_{u,v}$, we shall denote the almost surely unique path that attains the maximum, and this will be called a (point-to-point) polymer or a geodesic. 

Let us now introduce the necessary notations for the line-to-point last passage percolation. For $r\in \Z$, let $\L_{r}$ denote the line 
\begin{equation}\label{linedef}
\{(x,y)\in \Z^2:x+y=2r\}.
\end{equation}
For $v=(v_1,v_2)\in \Z^2$, we shall say that $v\succ \L_{r}$ if $v_1+v_2> 2r$. For $v\succ \L_{r}$, the line-to-point last passage time from $\L_{r}$ to $v$, denoted $X^r_{v}$ is defined as 
$$X^r_{v}:=\max_{u\in \L_{r}} T_{u,v}.$$ 
The almost surely unique path achieving this maximum will be denoted by $\Gamma^{r}_{v}$ and called the line-to-point polymer or geodesic. To avoid notational overhead we shall drop the superscript $r$ in the above notations for the special case $r=0$. Further, for $n\in \N$ and $v=\mathbf{n}$ ($\mathbf{n}$ will denote the point $(n,n)$ throughout) we shall denote $X^{0}_{\mathbf{n}}$ simply by $X_{n}$.

\begin{figure}[htbp!]
\includegraphics[width=.5\textwidth]{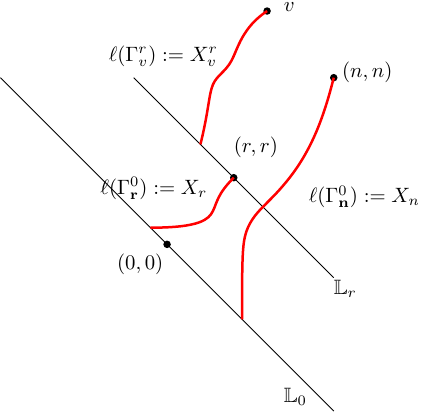} 
\caption{The figure illustrates various definitions and notations for line-to-point LPP that occur throughout the article. $\L_{r}$ denotes the line $\{x+y=2r\}.$ $\Gamma_{v}^{r}$ denotes the geodesic between the vertex $v$ and the line $x+y=2r$ while $\ell(\Gamma_{v}^{r})$ denotes its weight.}
\label{f:defn1}
\end{figure}

It has been of interest to understand the correlation structure of the growing profiles $\{T_{\mathbf{0},v}\}_{v\in \L_{n}}$, $\{X_{v}\}_{v\in \L_{n}}$ as $n$ varies (or indeed for any other generic initial condition, i.e., curve-to-point LPP for some suitable curve). In this context, we consider, for the flat initial data (i.e., line-to-point LPP), the covariance between the last passage times to two points on the main diagonal; specifically between $X_{r}$ and $X_{n}$, denoted $\mbox{Cov}(X_{r},X_{n})$, for $1\ll r \ll n$.

To state our main result, let us first consider the scaling $r=\tau n$. We shall first send $n\to \infty$ and consider the $\tau\to 0$ asymptotics of $\mbox{Cov}(X_{\tau n},X_{n})$. It was shown in \cite{FO18} (this also follows from the results in e.g.\ \cite{MQR17, DOV18}) that 
$$ \rho(\tau):= \lim_{n\to \infty} n^{-2/3}\mbox{Cov}(X_{\tau n},X_{n})$$
exists for all $\tau\in [0,1]$. One naturally conjectures that as $\tau\to 0$, $\rho(\tau)\to 0$ is a power law and we are interested in the exponent. Let $\chi:=\lim_{\tau\to 0}  \frac{\log \rho(\tau)}{\log \tau}$ provided the limit exists. The value of $\chi$ predicted in \cite{FS16, takeuchi2012evidence} is confirmed by our main result.

\begin{maintheorem}
\label{t:main}
In the above set-up, $\chi$ exists and is equal to $\frac{4}{3}$.
\end{maintheorem}

Proof of Theorem \ref{t:main} follows from separate upper and lower bounds to $\mbox{Cov}(X_{r},X_{n})$ when $\frac{r}{n}$ is bounded away from $0$ and $n$ is sufficiently large. We start with the upper bound. 

\begin{theorem}[Upper Bound]
\label{t:upper}
There exist absolute constants $C_1,C_2>0$ such that for any $\delta\in (0,1/2)$ there exists $n_0(\delta) \in \R_+$ with the following property: for any $n, r \in \Z_+$ with $\delta n < r < \frac{n}{2}$ and $n > n_0(\delta)$ we have 
$$\Cov(X_{r},X_{n}) \leq  C_1\left(\frac{r}{n}\right)^{4/3} \exp\left(-C_2\log(r/n)^{5/6}\right) n^{2/3}.$$
\end{theorem}

Next we have the lower bound result with the same exponent. 

\begin{theorem}[Lower Bound]
\label{t:lower}
There exists an absolute constant $C_3>0$ such that for any $\delta\in (0,1/2)$ there exists $n_0(\delta) \in \R_+$ with the following property: for any $n, r \in \Z_+$ with $\delta n < r < \frac{n}{2}$ and $n > n_0(\delta)$ we have 
$$\Cov(X_{r},X_{n}) \geq  C_3\left(\frac{r}{n}\right)^{4/3}n^{2/3}.$$
\end{theorem}
\begin{remark}
Using the scaling of $r=\tau n$ above and as in the literature, these theorems can also be stated in the following way.
With the same $C_1,C_2,C_3>0$, $\delta\in (0,1/2)$ and $n_0(\delta) \in \R_+$, for any $n\in \Z_+$, $n > n_0(\delta)$, and $\delta < \tau < \frac{1}{2}$, such that $n\tau \in \Z_+$, we have 
$$C_3\tau^{4/3} \leq n^{-2/3}\Cov(X_{\tau n},X_{n}) \leq  C_1\tau^{4/3} \exp\left(-C_2\log(\tau)^{5/6}\right),$$
\end{remark}

Clearly Theorem \ref{t:upper} and Theorem \ref{t:lower} together imply Theorem \ref{t:main}. Observe that the upper and lower bounds in  Theorem \ref{t:upper} and Theorem \ref{t:lower} respectively differ by a sub-polynomial (in $\tau=\frac{r}{n}$) factor. As predicted in \cite{FS16, takeuchi2012evidence}, we believe that the lower bound is tight up to a constant, and the $\exp\left(-C_2\log(\tau)^{5/6}\right)$ term in the upper bound is merely an artifact of our proof. This term appears from an estimate (quoted from \cite{HHJ+}) of the Radon-Nikodym derivative of the Airy$_2$ process with respect to Brownian motion. In contrast, for the droplet initial condition, where the covariance scales as $\tau^{2/3}n^{2/3}$, the upper and lower bounds in \cite{BG18,FO18} differ only by a constant factor (actually, in \cite{FO18}, even the constant in front of $\tau^{2/3}$ is evaluated in the $n\to \infty$ limit up to a multiplicative error term that is $1+o(1)$ as $\tau\to 0$). The primary difficulty in treating the case of flat initial data, as compared to the droplet case is that here one needs to deal with Brownian like fluctuations of the line-to-point profile $\{T_{v, \mathbf{n}}\}_{v\in \L_{r}}$ at very short scales. Only upper bound of such fluctuations were enough for the work \cite{BG18} which, in turn, required only the moderate deviation estimates for point-to-point passage times. Here, however, we need finer comparisons. These additional difficulties are tackled by using the aforementioned result from \cite{HHJ+} (also using translation invariance of the Airy$_2$ process) which, however, leads to the non-optimal sub-polynomial factor in the upper bound of Theorem \ref{t:upper}. 

\subsection{Background and Related Results}
Exponential LPP is one of the canonical examples of exactly solvable models in the KPZ universality class. It is classical \cite{Ro81} that $T_{\mathbf{0},\mathbf{n}}\sim 4n$ and $2^{-4/3}n^{-1/3}(T_{\mathbf{0},\mathbf{n}}-4n)$ converges weakly to the GUE Tracy-Widom distribution \cite{Jo99}. Finite dimensional  distributions for the line-to-point profile $\{T_{v, \mathbf{n}}\}_{v\in \L_{0}}$ has also been classically studied \footnote{One usually writes the point-to-line profile $\{T_{\mathbf{0},v}\}_{v\in \L_{n}}$ which has the same law by obvious symmetries of the lattice.}. It is known that the correlation length of this profile is $n^{2/3}$, and after a spatial scaling by the correlation length it converges weakly to the Airy$_2$ process, which is a stationary ergodic process, minus a parabola. More precisely, 
\begin{equation}\label{profile12}
\sL_{n}(x):=2^{-4/3}n^{-1/3}\biggl(T_{(x(2n)^{2/3},-x(2n)^{2/3}), \bn}-4n\biggr) \Rightarrow \cA_{2}(x)-x^2
\end{equation}
in the sense of finite dimensional distributions where $\cA_2$ denotes the Airy$_2$ process on $\R$ \cite{BF08,BP08} (note that the LHS in \eqref{profile12} is defined only for $|x(2n)^{2/3}|\le n$, but this does not cause problems in the $n\to \infty$ limit). Using a tightness result from \cite{Pim17}, it also follows that the weak convergence holds in the topology of uniform convergence on compact sets \cite{FO17}. {For completeness we shall also provide a proof of this fact later in the article (Theorem \ref{t:lpptoairy})}.

One also has an exact solution for the flat initial condition and it turns out that in this case $2^{-2/3}n^{-1/3}(X_{n}-4n)$ converges weakly to the GOE Tracy-Widom distribution \cite{Sas05, BFPS07}.
The profile, $\{X_{v}\}_{v\in \L_{n}}$ after appropriate scaling, converges in this case to the Airy$_1$ process $\cA_1$: 
{$$x \mapsto 2^{-4/3}n^{-1/3}\biggl(X_{(n+x(2n)^{2/3},n-x(2n)^{2/3})}-4n\biggr) \Rightarrow 2^{1/3}\cA_{1}(2^{-2/3}x)$$
in the sense of convergence of finite-dimensional distributions \cite{BFP07, BFPS07, BPS08}.}
To obtain formulae for the weak scaling limits for more generic initial conditions were open until the recent work \cite{MQR17} where formulae for the finite dimensional distributions of such limiting objects were obtained for a general class of initial data. 

Apart from the formulae for finite dimensional distributions, there has been work in understanding the local behavior of the line-to-point profile. One can embed $\cA_{2}(x)$ as the top curve of a non-intersecting line ensemble (Airy line ensemble) which after parabolically adjusting (subtracting $x^2$) exhibits the \emph{Brownian Gibbs property} \cite{CH14}. In recent works of Hammond and co-authors, this and its pre-limiting analogues for the exactly solvable model of Brownian last passage percolation together with one point moderate deviation estimates has been used to great effect in obtaining local Brownian behavior of the Airy$_2$ process \cite{H16, HHJ+}.
A different approach, using comparisons with stationary LPP has been taken in \cite{Pim17}. One-sided Brownian fluctuation estimate using coalescence of geodesics is also obtained in \cite{BG18}. Such estimates have found applications in many geometric questions about last passage percolation \cite{BG18, BGH19, H17a,H17b,H17c}.

As already mentioned, the interest in understanding two-time distributions of the passage time profile started from different initial conditions is rather recent. 
It started from studies of the time correlations, experimentally by Takeuchi and Sano \cite{takeuchi2012evidence}, and numerically by Singha \cite{singha2005persistence} for the step initial condition. The first published mathematical study on this problem was by Ferrari and Spohn \cite{FS16}, who studied the time correlations for exponential LPP started from step, flat or stationary initial conditions. Using a variational problem involving Airy processes they conjectured two asymptotic expansions (in the $\tau\to 0$ and $\tau\to 1$ limits in our notation) for the two time covariance for the step and flat initial conditions and an exact expression for the stationary initial condition. 
Around the same time, for the step initial condition, the exponents conjectured in \cite{FS16, takeuchi2012evidence, singha2005persistence} was rigorously obtained in the related model of Brownian last passage percolation in the unpublished work \cite{CH14+}. This employed the Brownian Gibbs property of the associated line ensemble as mentioned above. For the step and stationary initial conditions, the conjectured expansion of \cite{FS16} was made rigorous in \cite{FO18}, who in particular showed, among other things, that for the correlation coefficient ${\rm Corr}(T_{\mathbf{0},\mathbf{n}}, T_{\mathbf{0},\tau\mathbf{n}}) = \frac{\Cov(T_{\mathbf{0},\mathbf{n}}, T_{\mathbf{0},\tau\mathbf{n}})}{\sqrt{\Var(T_{\mathbf{0},\mathbf{n}})\Var(T_{\mathbf{0},\tau\mathbf{n}})}}$,
\begin{align*}
\lim_{n\to \infty}{\rm Corr}(T_{\mathbf{0},\mathbf{n}},T_{\mathbf{0},\tau\mathbf{n}})=\begin{cases}
\Theta(\tau^{1/3}) &~\text{as}~\tau\to 0\\
1-\Theta((1-\tau)^{2/3})& ~\text{as}~\tau\to 1.
\end{cases}
\end{align*}
Here and throughout the paper, by $\Theta(x)$ we denote a positive quantity whose ratio to $x$ is bounded away from zero and infinity.

A finite $n$  version of the same result was independently obtained in \cite{BG18} for the step initial condition using moderate deviation estimates for point-to-point passage times in exponential LPP together with the understanding of transversal fluctuations and coalescence of geodesics. For the flat initial condition, the $\tau\to 1$ limit also gives the same asymptotics as above and this was also established in \cite{FO18}. It remained to get the $\tau\to 0$ asymptotics for the flat initial condition, which was conjectured to have a different exponent. This is accomplished in Theorem \ref{t:main}. 

On a related, but different, line of recent works, efforts have also been directed to obtain exact formulae for the one-point or multi-point joint distribution of the profile at two (on-scale) separated time points. This originated with Johansson's work in Brownian LPP \cite{Joh17} and has been continued for geometric LPP and discrete polynuclear growth \cite{Joh17, Joh19, JR19} (see also \cite{LD18} for a replica calculation).
In parallel, exact asymptotic formulae for the two time distribution for the height function of TASEP (related to exponential LPP by the standard coupling) with different initial conditions have been obtained, first by Baik and Liu \cite{BL19} on periodic domains, then by Liu \cite{L19+} on $\Z$. 
In principle, all the statistics of the two-time distribution could be obtained from these formulae; however, it does not appear easy to extract the correlation out of the impressive but complicated formulae obtained in these works.

Our approach, in contrast, eschews exact formulae on behalf of more geometric arguments, in principle combining the  approaches taken in \cite{CH14+} and \cite{BG18}. We construct geometric events about optimal paths, their weights and transversal fluctuations to control the correlation between the line-to-point last passage times. In contrast to \cite{BG18}, to rule out the contributions of certain atypical events to the correlation, we require a strong control of the line-to-point profile around its maxima. This is obtained by resorting to strong quantitative estimates of Brownian regularity of the Airy processes based on Brownian Gibbs property and resampling techniques, which has been developed in a series of works by Hammond \cite{H16,H17a,H17b,H17c}. Of particular relevance to us is the recent work \cite{HHJ+} which obtains quantitative comparisons between Brownian motion and Airy$_2$ process, rather than the Brownian bridge comparisons in the earlier work \cite{H16}. 

Even though our results (see Theorem \ref{t:upper} and Theorem \ref{t:lower}) are formulated in a non-asymptotic form, we require $n$ to be sufficiently large depending on a uniform lower bound on $\frac{r}{n}$, since we rely on estimates from \cite{HHJ+} for the Airy$_2$ process. In contrast, the corresponding results in \cite{BG18} only required $r\gg 1$ as certain regularity estimates there were directly proved (using only one point estimates) for the pre-limiting line-to-point profile in exponential LPP, leading to uniformly non-asymptotic results. Here we need stronger control on the line-to-point profile which, so far, is only available in either Brownian LPP or in the Airy limit. Finite $n$ versions of our results without the restriction $r>\delta n$ can presumably be proved in the context of Brownian LPP using our techniques; however we do not explore that question here. 

Finally, it is also worth mentioning that the correlation across two times have recently been investigated beyond the zero temperature case. Using a different Gibbs property from \cite{CH16}, and one point estimates from \cite{CG18b, CG18a}, correlation exponents have recently been established in \cite{CGH19} for the KPZ equation.

\section{Outline of the proof and the key technical ingredients}
\label{iop}

In this section, we present an outline of the proof of Theorem \ref{t:main} as well as a review of the  various technical components involved, and also discuss some possible extensions. See Figure \ref{f:outline} for an illustration of the relevant objects discussed. Recall that our goal is to prove tight bounds on $\mathrm{Cov}(X_{r},X_{n})$. The heuristics behind the exponent $4/3$ is rather natural and was already presented in \cite{FS16}. Let the geodesic $\Gamma_{\mathbf{n}}^{0}$ from $\mathbb{L}_{0}$ to $\mathbf{n}$  intersect $\mathbb{L}_{r}$ at $u_0$.  It is known (see \cite{BSS17++}) that along a journey of length $r$, the transversal fluctuation witnessed by a geodesic is typically $O(r^{2/3})$. Hence one expects that on the event $|u_0-\br|\gg r^{2/3}$ there would likely be no interaction between the geodesics $\Gamma_{\bn}^0$ and $\Gamma^0_{\br}$ (i.e., they pass through disjoint parts of the space) and hence the contribution to covariance between the geodesic weights $X_{r}$ and $X_{n}$ coming from this event will be negligible.

\begin{figure}[htbp!]
\includegraphics[width=.7\textwidth]{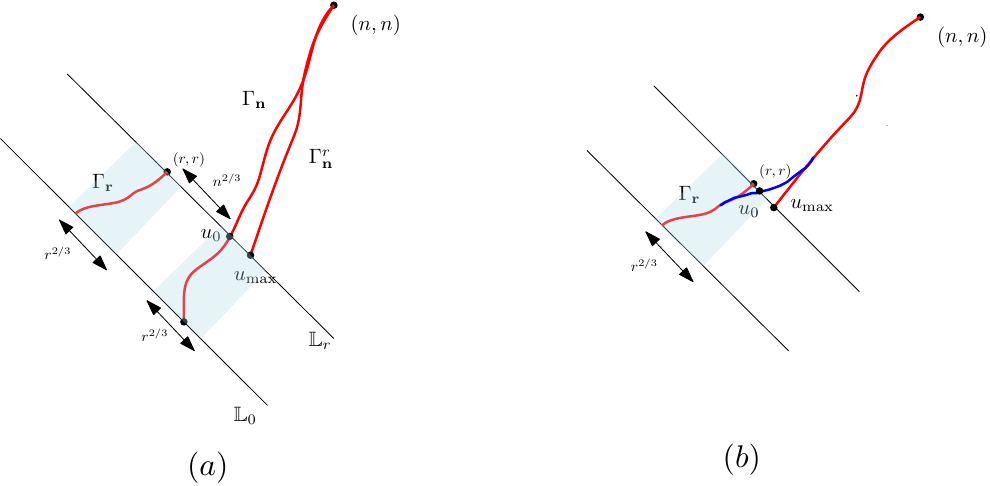} 
\caption{An illustration of the objects featuring in the proof strategies. $\Gamma_{\br}$, $\Gamma_{\bn}$ and $\Gamma^r_{\bn}$ denote the line to point polymers from $\L_0$ to $\br$ and $\bn$ and $\L_r$ to $\bn$ respectively; $u_0$ denotes the intersection of $\Gamma_{\bn}$ with $\L_r$ while $u_{\max}$ denotes the intersection of $\Gamma^r_{\bn}$ with $\L_r.$ (a) emphasizes the typical separation between $(r,r)$ and $u_0$ to be of order $n^{2/3}$ which prevents $\Gamma_{\br}$ and $\Gamma_{\bn}$ from interacting while $(b)$ indicates the rare event $|u_{0}-\br|\approx r^{2/3}$ on which the geodesics would typically coalesce leading to a contribution for the covariance.}
\label{f:outline}
\end{figure}

However, on the event $|u_0-\br|=\Theta(r^{2/3}),$ one expects that the restriction of $\Gamma^0_{\bn}$ between $\L_r$ and $\L_0$ interacts and typically overlaps significantly with $\Gamma^0_{\br}$ as both these paths have transversal fluctuations of the same order as the separation of their endpoints on $\L_r$ (i.e., $u_0$ and $\br$). This causes the covariance conditional on the above mentioned small probability event, to be of the same order as the variance of $X_{r}$ which is $r^{2/3}$.  Now, since $r\ll n$, one also expects that $u_0$ is close to $u_{\max}$ where $u_{\max}$ denotes the starting point of the geodesic $\Gamma_{\mathbf{n}}^{r}$ from $\mathbb{L}_{r}$ to $\mathbf{n}$. Owing to the nature of the polymer weight profile $\{T_{v, \bn}\}_{v\in \L_r}$ which decays parabolically away from $(r,r)$ with local Brownian fluctuations, the location of $u_{\max}$ is roughly uniformly distributed on an interval of size $n^{2/3}$. Hence one expects the probability that $u_{\max}$ (and hence $u_0$) lies within $\Theta(r^{2/3})$ of $\br$ is approximately 
$\Theta((\frac{r}{n})^{2/3})$. This, together with the above heuristic leads to the predicted value of the covariance between $X_r$ and $X_n$ to be $\Theta((\frac{r}{n})^{2/3}r^{2/3})$.  

Let us now provide some further detail on how the argument described above is made rigorous to prove the upper and lower bounds. 

\noindent
\textbf{Upper bound:} The proof of the upper bound is provided in Section \ref{s:upper}. For the upper bound, to make the above heuristic precise one needs two ingredients: 
\begin{itemize}
\item[(i)] $u_{\max}$ is roughly uniformly distributed on an interval of size $n^{2/3}$ centered at $\br$, the probability that $r^{-2/3}|u_{\max}-\br|\approx x$ does not decay with $x$ (as long as $x\leq n^{2/3}r^{-2/3}$).
\item[(ii)] $u_0$ is indeed close to $u_{\max}$. 
\end{itemize}
To establish (ii), we need to know that the line-to-point profile $\{T_{v, \bn}\}_{v\in \L_r}$ is unlikely to be close to its maxima at some point $v$ with $|v-u_{\max}|\gg r^{2/3}$, as such an event could make it possible for the point $v$ to be a potential candidate for the point $u_0$. 
As already alluded to before, it is well known that the line-to-point profile, after suitable scaling, converges to a parabolically adjusted Airy$_2$ process. It is also well known that Airy$_2$ process is locally Brownian, and it is this property that we exploit. Notice that heuristically this should imply for the (pre-limit of) parabolically adjusted Airy$_2$ process:
\begin{itemize}
\item[(a)] The location of the maximum of the weight profile is roughly uniformly distributed on an interval on length $n^{2/3}$, and 
\item[(b)] 
The process around its maxima looks like two-sided Brownian motion conditioned to stay below zero. 
\end{itemize} 
 
To establish precise statements along the lines of (a) and (b) we use the recent advances in \cite{HHJ+} which gives a strong control of Radon-Nikodym derivative of parabolically adjusted Airy$_2$ process with respect to Brownian motion (see Theorem \ref{t.airytail}). The particular consequence of this result which is relevant for our upper bound result is established in Proposition \ref{p:twopeaks} (which is a precise statement encompassing (a) and (b) above). 
To go from (a) and (b) above to (i) and (ii) which feeds into the geometric argument described before, we use the uniformly on compact sets weak convergence of the line-to-point profile in exponential LPP to Airy$_2$ process minus a parabola (Theorem \ref{t:lpptoairy}) and pull back the Brownian regularity results for the Airy$_2$ process to the geodesic weight profile for exponential LPP; see Proposition \ref{prop:bdi} for a precise statement.  

Given the above inputs, the technical geometric argument proving the upper bound involves  decomposing the total covariance  according to the location of $u_{\max}$ (This is illustrated later in Figure \ref{f:jpaths1}). The $j^{th}$ event $A_j$ in this decomposition is that $|u_{\max}-\br| \in  [j^{101}, (j+1)^{101}] r^{2/3}$ ($101$ here is an arbitrary large enough number). The proof then proceeds by proving that the covariance contribution from the above event is (up to a multiplicative sub-polynomial in $\frac{r}{n}$ error) $f(j)r^{4/3}n^{-2/3}$ where $f(j)$ is decaying polynomially in $j$ with a large enough rate (depending on $101$) making it summable. This is formally carried out in Section \ref{covdec} (see in particular Lemmas \ref{l:bjbound} and \ref{l:cjbound}). 

\noindent
\textbf{Lower bound:} The corresponding lower bound is proved in Section \ref{s:lower}. The argument in this section develops upon the ideas present in an earlier work of the first two authors on understanding temporal correlation starting from narrow-wedge initial conditions \cite{BG18}, albeit requiring significant new ingredients. We present below a slightly simplified high level idea. Consider a strip (say $R_{\theta}$) of width $\theta r^{2/3}$ and height $r$ with its shorter sides being segments of the lines $\L_0$ and $\L_r,$ centered at $\mathbf{0}$ and $\br$. We next condition on the entire environment except $R_{\theta}$.  Let $\cE$ denote the set of conditioned environments which satisfy the following  properties:
\begin{itemize}
\item[(i)] The argmax $u_{\max}$ of geodesic weight profile $\{T_{v, \bn}\}_{v\in \L_r}$ lies on $R_{\theta}\cap \L_r$.
\item[(ii)] The profile has parabolic decay away from the maxima.
\item[(iii)] The environment between $\L_0$ and $\L_r$ outside $R_{\theta}$ is depleted on a region of width $O(r^{2/3})$.
\end{itemize}

The argument has two main steps. 

\noindent
\textbf{Step 1} entails showing $\P(\cE)\ge \Theta((\frac{r}{n})^{2/3})$. This follows from noticing that (i) and (ii) above is independent of (iii); the depleted region can be constructed with a uniformly positive probability; and the fact that (i) and (ii) hold with probability proportional to $r^{2/3}/n^{2/3}$ follows from a further Brownian comparison estimate for Airy$_2$ process (Proposition \ref{prop:lbae}) pulled pack to  exponential LPP geodesic weight profile (Proposition \ref{p:maxdecay}).  The Brownian comparison in this section does not rely on the results of \cite{HHJ+}, and instead uses stationarity and strong mixing properties of the Airy$_2$ process.

\noindent
\textbf{Step 2} involves showing that conditioned on any sample point in $\cE,$ both $\Gamma_{\br}$ and the journey of $\Gamma_{\bn}$ from $\L_{r}$ to $\L_{0}$ would typically be constrained within $R_{\theta}$. This causes the two paths to overlap significantly and makes a conditional covariance contribution of the same order as the variance of the weight of the path between $\br$ and $\mathbf{0}$ constrained to stay within $R_{\theta}$. This has variance of the order of $\theta^{-1/2}r^{2/3}$. Thus choosing $\theta$ small enough allows us the variance term to dominate all other correction terms from various approximating statements. In particular this gives a lower bound of $cr^{2/3}$ (for some $\theta$ dependent constant $c>0$) for the conditional covariance of $X_r$ and $X_n$.

Finally a simple consequence of the FKG inequality allows us to lower bound the unconditional covariance with the conditional covariance averaged over $\cE$ to finish the proof of the lower bound.

\subsection{General Boundary Conditions}\label{gbc}
Although the primary focus of this article was to establish the temporal covariance exponent $4/3$ conjectured in \cite{FS16, takeuchi2012evidence} for the flat initial condition, the argument is sufficiently robust to allow for more general initial data. Although the formal pursuing of this direction will be taken up in the future, we provide a brief outline of a possible approach to highlight the generality of the method. Let us define a generic class of initial conditions given by $\pi:=\{\pi(v)\}_{v\in \L_0}$. Then the last passage times of  interest become $$X^{\pi}_r:=\max_{u\in \L_0} (T_{u,\br}+\pi(u)) \text{ and }  X^{\pi}_{n}:=\max_{u\in \L_0} (T_{u,\bn}+\pi(u)).$$ Observe that the flat initial condition considered in Theorem \ref{t:main} corresponds to the case $\pi \equiv 0$ whereas the droplet case considered in \cite{BG18, FO18} corresponds to the case $\pi(\mathbf{0})=0$ and $\pi(v)=-\infty$ for all $v\neq \mathbf{0}$. In \cite{FS16,FO18} another special initial condition, namely the stationary (with slope $0$) initial condition, was also considered where $\{X_{i}\}_{i\in \Z}$ and $\{Y_{i}\}_{i\in \Z}$ are independent sequences of i.i.d.\  $\rm{Exp}(1/2)$ variables and $\pi$ is a two sided random walk profile with $\pi(\mathbf{0})=0$ and increments $X_{i}-Y_{i}$. For this special initial condition (let us denote it by $\pi^*$) it was shown in \cite{FS16} that we have 
$$\lim_{n\to \infty} n^{-2/3}\Cov(X^{\pi^*}_{\tau n},X^{\pi^*}_{n})=\frac{\Var(\xi_{\mathrm{BR}})}{2}\left(1+\tau^{2/3}-(1-\tau)^{2/3}\right)$$ 
where $\Var(\xi_{\mathrm{BR}})$ denotes the variance of the Baik-Rains distribution \cite{BR}. Observe that the covariance above is annealed  i.e.,  the randomness in the initial condition is also averaged over. Simple Taylor expansion shows that in the $\tau\to 0$ limit, the asymptotic covariance is of the order $\tau^{2/3}n^{2/3}$ which matches with the asymptotics of the droplet initial condition \cite{BG18, FO18}. However the root of such an exponent is quite different from that of the droplet case, with the primary contribution to the covariance coming from the fluctuation in the initial condition. In particular,  one does not expect the same exponent to guide the covariance behavior if the initial data is quenched: i.e., the covariance with a typical deterministic initial condition with diffusive growth.  

In fact, as the reader might have already noticed, the proof strategy described above indicates an exponent of $4/3$ for any  deterministic initial profile with sub-diffusive growth. More formally let us define a deterministic initial profile $\pi$ with the following two properties: 
\begin{enumerate}
\item[(i)] $\pi(0)=0$. 
\item[(ii)] For any $v\in \L_0$ we have $|\pi(v)|\leq C|v|^{1/2-s}$ for some $s\in (0,1/2)$ and some $C>0$.  \end{enumerate}

We believe that the arguments in this paper, upon suitable modifications, might be used to show that for any $\pi$ as above and $r,n$ as in the set-up of Theorem \ref{t:main},
$$\Cov(X^{\pi}_r, X^{\pi}_n)=\Theta\left((\frac{r}{n})^{4/3+o(1)}n^{2/3}\right).$$
To see why this is plausible, recall that the overall strategy from the beginning of the section did not explicitly depend on the initial data $\pi$ being identically zero. Rather it only implicitly used the fact that on the event $|u_{\max}-\br|\gg r^{2/3}$ there would likely be no interaction between the geodesics $\Gamma_{\bn}^0$ and $\Gamma^0_{\br}$ (i.e., they pass through disjoint parts of the space) and hence the contribution to covariance between the geodesic weights $X_{r}$ and $X_{n}$ coming from this event will be negligible.  This still continues to hold for $\pi$ as above, since we expect the profile around $u_{\max}$ to behave like a Brownian motion conditioned to stay below zero, and hence the weight deficit between $u_{\max}$ and $\br$ to be of order $\sqrt{|u_{\max}-\br|}$ which dominates $|\pi(u_{\max}-\br)|$ due to the sub-diffusive assumption on the latter. Hence the $\pi$ terms in $X^{\pi}_r=\max_{u\in \L_0} (T_{u,\br}+\pi(u))$ and $X^{\pi}_{n}=\max_{u\in \L_0} (T_{u,\bn}+\pi(u))$ are negligible compared to the fluctuations of the $T_{u,\br}$ and $T_{u,\bn}$ terms, leading to the same behavior as when $\pi\equiv 0$.

In the rest of the discussion, we shall merely try to give some pointers to the interested reader to the relevant estimates appearing later in the article and how one might modify them to construct a complete proof.  Thus the reader might find it beneficial to come back to this discussion on having gone through the rest of the paper.

The argument for the lower bound should go through almost verbatim. One merely needs to take $\tau<s$ in the definition of $\ce_{\mathrm{dec}}$ in \eqref{e:defdec}. This ensures that as before on the event $\ce,$ the geodesic from $\L_0$ to $\bn$ passes close to $R_{\theta}$ ((ii) above ensures that $\pi$-values on the line $\L_0$ is not sufficiently high so that the geodesic will either pass through the depleted region, or have an atypically high transversal fluctuation). 

The upper bound requires some modifications to the statements proved later in the article.  The key property driving the results in the article is the decay of the geodesic weight profile $\{T_{u,\bn}\}_{u\in \L_r}$ around its maxima $u_{\max}$ resembling a two sided Brownian meander below zero. However for our purposes here a rather weak version of the above as stated formally in \eqref{decayneeded} suffices; instead of the expected diffusive decay we are content with a poly-logarithmic decay. This is because stretched exponential tails for geodesic weights make the low probability events appearing in our analysis super-polynomially rare rendering them summable. But to compete with the almost diffusive fluctuation of the initial data in the general setting, this is not enough any more and one has to exploit the full Brownian decay of the weight profile away from its maxima (see in particular the discussion appearing after \eqref{decayneeded}). We shall not expand on this further and the question of more general initial conditions will be taken up in a future project. 

\subsection*{Organisation of the paper}
The rest of the paper is organized as follows with all the remaining sections devoted to proofs. In Section \ref{sec:browncomp} we obtain estimates for the local Brownian properties of the Airy$_2$ process, and the line-to-point profile in exponential LPP. Additional ingredients derived as consequences of well known one point moderate deviation estimates are recorded in Section \ref{s:prelim}.
The proofs of Theorem \ref{t:upper} and \ref{t:lower} are presented in Sections \ref{s:upper} and \ref{s:lower} respectively.
Some of the technical but straightforward or well known arguments omitted from the main body of the article including several Brownian computations and consequences of the convergence of geodesic weight profiles to the parabolically adjusted Airy$_2$ process are presented in Appendix \ref{s:appa} and Appendix \ref{s:appb}. In Appendix \ref{s:appc} we present the technical arguments about estimates on last passage times and transversal fluctuation of paths, adapted from similar arguments in the literature.

\subsection*{Acknowledgements}
The authors thank Patrik Ferrari for bringing this question to their attention and explaining the heuristic from \cite{FS16}. We also thank Alan Hammond and Milind Hegde for useful discussions about the Brownian comparison result in \cite{HHJ+}. We are grateful to two anonymous referees for a careful reading of the manuscript and numerous helpful comments and suggestions. RB is partially supported by a Ramanujan Fellowship (SB/S2/RJN-097/2017) from the Science and Engineering Research Board, an ICTS-Simons Junior Faculty Fellowship, DAE project no. 12-R\&D-TFR-5.10-1100 via ICTS and Infosys Foundation via the Infosys-Chandrasekharan Virtual Centre for Random Geometry of TIFR. SG is partially supported by a Sloan Research Fellowship in Mathematics and NSF Award DMS-1855688. Part of this research was performed at ICTS while SG and LZ were attending the program -Universality in random structures: Interfaces, Matrices, Sandpiles (ICTS/urs2019/01), we acknowledge the hospitality.

\section{Brownian Comparison Estimates for the \texorpdfstring{Airy$_2$}{Airy2} process}   \label{sec:browncomp}
As indicated in Section \ref{iop}, we shall need several estimates for the Airy$_2$ process in particular relating to its locally Brownian nature.  We collect together the required results in this section. A general method proof for such results is via comparison: namely to bound the probability of a certain event under the Airy$_2$ process, we first bound the same  for Brownian motion, and then rely on strong estimates on the Radon-Nikodym derivative of Airy$_2$ process with respect to Brownian motion to obtain the sought estimate. Although there have been several recent breakthroughs in this sphere a particularly useful one for our purpose that we rely on is the following very recent result. 

Recall that (see e.g., \cite{CH14}) $\cA_{2}(\cdot)$ denotes the Airy$_2$ process on $\R$. Let us define the stochastic process $\sL:\R\to \R$ by
$$\sL(x):=\cA_{2}(x)-x^2.$$ For $K\in \R, d>0$, let $\cB^{[K,K+d]}$ denote the law of a Brownian motion with diffusivity $2$ taking value $0$ at $K$ and restricted to $[K,K+d]$. Let $\sL^{[K,K+d]}$ denote the random function on $[K,K+d]$ defined by 
$$\sL^{[K,K+d]}(x):=\sL(x)-\sL(K), ~\forall x\in [K,K+d].$$ 
The next result is the crucial  comparison estimate of the laws of $\sL^{[K,K+d]}$ and $\cB^{[K,K+d]}$. Let $\mathcal{C}\big([K,K+d],\mathbb{R} \big)$ denote the set of all real valued continuous functions defined on $[K,K+d]$ which vanish at $K$. The following result is quoted from \cite{HHJ+}, however we have rephrased the statement slightly. The result in \cite{HHJ+} compares $2^{-1/2}\sL$ to a Brownian motion with diffusivity $1$, and also gives a comparison estimate for the full profile on an interval $[-M,M]$ containing $[K,K+d]$ (see below). The form we state here, tailored to our applications later in the article, is an immediate corollary. 
\begin{theorem}\label{t.airytail}\cite[Theorem 1.1]{HHJ+} There exists an universal constant $G>0$ such that the following holds.
For any fixed $M>0,$ there exists $a_0=a(M)$ such that for all intervals $[K,K+d]\subset [-M,M]$ and for all measurable $A \subset \mathcal{C}\big([K,K+d],\mathbb{R} \big)$ with $0<\cB^{[K,K+d]}(A)=a\le a_0,$
$$\P\left(\sL^{[K,K+d]} \in A\right) \leq a     \exp \Big\{ G {M} \big( \log a^{-1} \big)^{5/6} \Big\}.$$
\end{theorem}
Theorem \ref{t.airytail} is proved in \cite{HHJ+} by first proving a quantitative version in the pre-limiting model of Brownian LPP followed by taking a large system limit.
We also need the following easy corollary that compares Airy$_2$ process directly to a two sided Brownian motion.

\begin{corollary}
\label{c:BrownianAiry}
Let $M>0$ be fixed and let $W$ denote a standard two sided standard Brownian motion on $[-M,M]$ {with diffusivity 2}. Then if for some measurable sequence of sets $A_n \subset C([-M,M],\R),$ $\P[\{W(x)-W(-M):x\in [-M,M]\}\in A_{n}]\to 0$ then $\P[\{\cA_2(x)-\cA_2(-M): x\in [-M,M]\}\in A_n]$ also goes to $0$.
\end{corollary}

\begin{proof}
By Cameron-Martin Theorem, we have that a Brownian motion plus a parabola is absolutely continuous with respect to a Brownian motion (of the same diffusivity) on any compact interval.
Thus
$\P[\{W(x)+x^2-W(-M)-M^2:x\in [-M,M]\}\in A_{n}]\to 0$,
and our conclusion follows by applying Theorem \ref{t.airytail} to events $\{f \in C([-M,M],\R): x\mapsto f(x)+x^2-M^2 \in A_n \}$.
\end{proof}

\subsection{Location and Behavior around Maxima}
Settling Johansson's conjecture, it was  established in \cite{CH14} that almost surely there is a unique maximizer of $\sL$ who also obtained  tail estimates  of the maxima of the kind $\P(|\argmax \sL|\geq s)\leq e^{-cs^3}$. One expects (from e.g.\ Theorem \ref{t.airytail}) that $|\argmax \sL|$ has a uniformly positive and bounded density on any compact interval and $\sL$ around its maxima (re-centered) should behave like a Brownian motion conditioned to stay below $0$. We shall need some of the consequences of such behavior. The first result we need is the following. 

\begin{proposition}
\label{p:twopeaks}
For $M>0$, there exists $C$ depending only on $M$ such that the following holds. For any $\varepsilon\in (0,\frac{1}{2})$ and any interval $I\subseteq [-M,M]$ with $|I|\leq \varepsilon$ we have 
\begin{equation}  \label{eq:lem:bdairy}
\P\left[\max_{x \in I} \sL(x) > \max_{x \in [-2M, 2M]} \sL(x)-\sqrt{\varepsilon}\right] \leq C 
\varepsilon\exp\left(C|\log (\varepsilon)|^{5/6}\right).
\end{equation}
\end{proposition}

\begin{proof}
Let $W: [-2M,2M]\to \R$ denote a two sided Brownian motion with {diffusivity $2$}. Notice that, by Theorem \ref{t.airytail} it suffices to prove that 
there exists a constant $C_1$ depending on $M$ such that for all $I$ and $\varepsilon$ as in the statement of the proposition we have
\begin{equation}  
 \label{bcomp}
\P\left[\max_{x \in I} W(x) > \max_{x \in [-2M, 2M]} W(x)-\sqrt{\varepsilon}\right] \leq C_1 \varepsilon,
\end{equation}
since then by Theorem \ref{t.airytail}, \eqref{eq:lem:bdairy} follows. The standard Brownian  calculation proving \eqref{bcomp} is provided in Appendix \ref{s:appa}.
\end{proof}

The other result that we require concerns lower bounds of such events. As already mentioned, one expects from Theorem \ref{t.airytail} that the probability that the maxima (of $\sL$) is within a particular sub-interval of a compact set is proportional to the length of the interval. We show below that the same remains true, if we also ask in addition that the profile decays nicely around the maxima.  We emphasize that the proof of the lower bound unlike the proof of Proposition \ref{p:twopeaks} relies only on the stationarity of the Airy$_2$ process coupled with one point tail bounds along with known strong mixing properties of the same 
derived in \cite[(5.15)]{prahofer2002scale} and \cite[Proposition 1.13]{corwin2014ergodicity}.

For $\lambda>0$, let $\sL^*_{\lambda}:=\max_{|x|\leq \lambda} \sL(x)$. Further, for $k\in \N$, let $$\sL^*_{\lambda,k}:=\max_{2^{k-1} \lambda \leq |x|\leq 2^{k}\lambda} \sL(x).$$ We have the following result
which although is not stated in the most natural way, as we introduce two parameters $\si < \si'$,
is what we precisely need in later applications (e.g., Proposition \ref{p:maxdecay} below).

\begin{proposition} 
\label{prop:lbae}
For any $\tau> 0$, there exists $\alpha > 0$ such that the following is true. For any $0 < \si < \si' < 1$, 
\begin{equation}
\label{eq:lbae}
\P\left[
\sL^*_{\lambda}=\sL^*_{\lambda'}{\leq \sL(0)+\alpha^{-1}\sqrt{\si'}},  \qquad \sL^*_{\lambda',k}\leq \sL^*_{\lambda'}- \alpha \sqrt{\si'} 2^{k(\frac{1}{2}- \tau)} ~ \forall~k\in \N \right] > \alpha \si.
\end{equation}
\end{proposition}

Before getting in to the details let us try to motivate the statement. Notice that by local Brownian behavior it is expected that the global maximum is equal to $\sL^*_{\lambda}$ with probability proportional to $\lambda$. It is also expected that the decay of the profile around the global maxima should be like that of a (two sided) Brownian meander, and hence the profile at $2^k\lambda'$ distance from the maxima should roughly be $ \sL^*_{\lambda}- \Theta((2^k\lambda')^{1/2})$. What the proposition says is that by asking for a slightly weaker decay ($2^{k(\frac{1}{2}- \tau)}$ instead of $2^{k/2}$) we can show that the decay condition for all scales $k$ along with the prescribed location simultaneously holds with probability proportional to $\lambda$.
We need some preparatory lemmas before proving Proposition \ref{prop:lbae}. The first of these lemmas shows that $\sup_{x\in \R} \sL(x)+\frac{x^2}{2}$ is tight. 

\begin{lemma}
\label{l:goe}
Given any $\varepsilon>0$ (small), there exists $\mathscr{M}>0$ (large) such that  
$$\P\left( \sup_{x\in \R} \sL(x)+ \frac{x^2}{2} \geq \mathscr{M}\right)\leq \varepsilon.$$
\end{lemma} 

\begin{proof}
{We know that $\sup_{x\in \R} 4^{-1/3}\sL(x)$ has the same law as the GOE Tracy-Widom distribution (see e.g. \cite[Corollary 1.3]{Joh03})}. It follows from the standard tail bounds on the GOE Tracy-Widom distribution {(e.g. \cite{ramirez2011beta})}, that $$\sup_{x\in [-1,1]} \sL(x)+ x^2= \sup_{x\in [-1,1]} \cA_2(x),$$ is stochastically dominated by $1+E$ where $E$ has stretched exponential tails. By translation invariance of the Airy$_2$ process it follows that for all $t\in 2\Z \setminus \{0\}$, and for $\mathscr{M} >0$ we have
$$\P\left( \sup_{x\in [t-1,t+1]} \sL(x)+ \frac{x^2}{2} \geq \mathscr{M}\right)= \P\left( \sup_{x\in [-1,1]} \cA_2(x) \geq \mathscr{M}+ \frac{|t-1|^2}{2}\right)\leq e^{-c(\mathscr{M}+t^2)}$$
for some $c>0$. Choosing $\mathscr{M}$ sufficiently large depending on $\varepsilon$ and summing over all $t$ gives the result.
\end{proof}

The next result shows that the maximum of an Airy$_2$ process on an unbounded interval is almost surely infinite.

\begin{lemma}  
\label{lem:lbae2}
Given any $\varepsilon>0$ (small), and $\mathscr{M}>0$ (large), there is $\li=\li(\varepsilon,\mathscr{M})> 0$ (large), such that
\begin{equation}
\P\left[\sup_{x\in (0, \li)} \cA_2(x) <\mathscr{M} \right] < \varepsilon.
\end{equation}
\end{lemma}

\begin{proof}
Let $N \in \Z_+$ be the smallest number such that $\P[\cA_2(0)<\mathscr{M}]^N < \frac{\varepsilon}{2}$.
By the strong mixing condition of the Airy$_2$ process (see \cite[(5.15)]{prahofer2002scale} or \cite[Proposition 1.13]{corwin2014ergodicity}), 
we can inductively find $0=x_1<x_2<\cdots <x_N$, such that for any $2\leq k \leq N$, $$\P\left[\max_{1\leq i \leq k}\cA_2(x_i) < \mathscr{M}\right] < \P\left[\max_{1\leq i \leq k-1}\cA_2(x_i) < \mathscr{M}\right]\P[\cA_2(x_k) < \mathscr{M}]+\frac{\varepsilon}{2N}.$$
Thus using that $\cA_2$ is stationary, we have, by taking $\Lambda = x_N$,
$$\P\left[\sup_{x\in (0, \li)} \cA_2(x) <\mathscr{M} \right]\leq \P\left[\max_{1\leq i \leq N}\cA_2(x_i) < \mathscr{M}\right] < \prod_{i=1}^N\P[\cA_2(x_i) < \mathscr{M}]+\frac{\varepsilon}{2} = \P[\cA_2(0) < \mathscr{M}]^N+\frac{\varepsilon}{2} < \varepsilon,$$
completing the proof.
\end{proof}
Let us explain at this point, our strategy for the proof of Proposition \ref{prop:lbae}. Let $\mathscr{A}=\mathscr{A}(\li,\mathscr{M})$ denote the event 
\begin{equation}\label{keyevent}
\mathscr{A}:=\left\{\sup_{|x|\leq \li}\cA_2(x)\geq 2\mathscr{M}, \sup_{|x|>\Lambda} {\cA_2(x)-\frac{(|x|-\li)^2}{2} \leq \mathscr{M}}\right\}.
\end{equation}
Observe that on $\mathscr{A}$, one gets that $\sup_{x\in \R} \sL(x) =\sup_{|x|<\li} \sL(x)$. Further observe that on $\mathscr{A}$ we also have, $\sup_{|x|\leq \li}\cA_2(x)=\sup_{|x|\leq \li+\sqrt{2\mathscr{M}}} \cA_2(x)$.

Using Lemmas \ref{l:goe}, \ref{lem:lbae2} and translation invariance of $\cA_2$ it follows that one can choose $\mathscr{M}$ and $\Lambda$ such that $\P(\mathscr{A})\geq \frac{1}{2}$. We fix one such choice of $\mathscr{M}$ and $\Lambda$ for the remainder of this proof with $\mathscr{M}$ and $\Lambda$ sufficiently large. The idea now is to show that one can approximately cover the event $\mathscr{A}$ by union of $O(\frac{\Lambda}{\lambda})$ many events each of which has the same probability as the event in the statement of Proposition \ref{prop:lbae}. 

Let $0<\lambda<\lambda'<1$ be fixed and without loss of generality we shall assume that $\frac{\Lambda}{\lambda}$ is an even integer. Divide the interval $[-\li-\si,\li+\si]$ into intervals of size $2\lambda$ each: let $\li_i:=[2i\lambda-\lambda,2i\lambda+\lambda]$ for $|i|\leq \frac{\Lambda}{2\lambda}$. Let $\cce_{i}={\cce_{i}(\alpha)}$ denote the event that the maximum of the process $\cA_2(x)-(x-2i\lambda)^2$ (i.e. $\sL(x-2i\lambda)$ ) on the interval $|x-2i\lambda|\leq \lambda'$ is attained on the interval $|x-2i\lambda|\leq \lambda$. Formally,
$$\cce_{i}:=\left\{\sup_{|x-2i\lambda|\leq \lambda} \biggl(\cA_2(x)-(x-2i\lambda)^2\biggr)=\sup_{|x-2i\lambda|\leq \lambda'} \biggl(\cA_2(x)-(x-2i\lambda)^2\biggr)
\right\}.$$
{For technical reasons we also let $\cce_{i,*}=\cce_{i,*}(\alpha)$ be the event that $\cA_2(2i\si)$ is close to the maximum, (this requirement will be useful for a later application in Section \ref{s:lower}) i.e.,
$$\cce_{i,*}:=\left\{\sup_{|x-2i\lambda|\leq \lambda'} \biggl(\cA_2(x)-(x-2i\lambda)^2\biggr)
{\leq \cA_2(2i\si) + \alpha^{-1}\sqrt{\si'}}
\right\}.$$}
Further, for $k\in \N$, let $\cce_{i,k}=\cce_{i,k}(\alpha)$ be the event capturing the almost diffusive decay condition,
\begin{equation}\label{eventdecomp}
\cce_{i,k}:=\left\{\sup_{|x-2i\lambda|\leq \lambda'} \biggl( \cA_2(x)-(x-2i\lambda)^2 \biggr) \geq \sup_{2^{k-1}\lambda'\leq |x-2i\lambda|\leq 2^{k}\lambda'} \biggl( \cA_2(x)-(x-2i\lambda)^2\biggr)+2^{k(\frac{1}{2} - \tau)} \alpha\sqrt{\lambda'} \right\}.
\end{equation}
Finally let  $$\mathscr{F}_{i}:={\cce_{i,*}\cap} \cce_i \cap \bigcap_{k\in \N} \cce_{i,k}.$$ Notice that  $\mathscr{F}_0$ is exactly the event whose probability needs to be lower bounded in Proposition \ref{prop:lbae}. Further observe that  by translation invariance of the Airy$_2$ process, $\P(\mathscr{F}_i)$ is equal for each $i$. What remains to be shown is that the event $\bigcup_{|i|\leq \frac{\li}{2\si}} \mathscr{F}_{i}$ approximately covers the event $\mathscr{A}$. 
To this end, for $|i|\leq \frac{\Lambda}{2\lambda}$ and $k\in \N$, let $\mathscr{C}_{i,k}$ denote the event 
\begin{equation}\label{loweventdecomp}
\mathscr{C}_{i,k}:=\left\{\sup_{|x-2i\lambda|\leq \lambda}\cA_2(x)=\sup_{|x|\leq \li+\sqrt{2\mathscr{M}}} \cA_2(x) \leq \sup_{2^{k-1}\lambda'\leq |x-2i\lambda|\leq 2^{k}\lambda'} \cA_2(x)+2^{k(\frac{1}{2} - \tau)} \alpha\sqrt{\lambda'} \right\},
\end{equation}
and let $\mathscr{C}_{i,*}$ denote the event
\begin{equation}\label{loweventdecomp2}
\mathscr{C}_{i,*}:=\left\{\sup_{|x-2i\lambda|\leq \lambda}\cA_2(x)=\sup_{|x|\leq \li+\sqrt{2\mathscr{M}}} \cA_2(x) \geq  \cA_2(2i\lambda)+\alpha^{-1}\sqrt{\lambda'} \right\},
\end{equation}
and finally let  {$\mathscr{C}_{i}:=\mathscr{C}_{i,*}\cup \left(\bigcup_{k\leq k_*} \mathscr{C}_{i,k}\right)$} where $k_*$ is the largest integer where $2^{k}\lambda' \leq 8$. 
To see how this would be used, for $|i|\leq\frac{\li}{2\si}$, let $\mathscr{A}_i$ denote the event 
$$\mathscr{A}_i:=\left\{\sup_{|x-2i\lambda|<\lambda} \cA_2(x)=\sup_{|x|<\Lambda + \sqrt{2\mathscr{M}}} \cA_2(x){>}\sup_{|x|>\Lambda} \left[\cA_2(x)-\frac{(|x|-\li)^2}{2}\right]+\mathscr{M}\right\}.$$
Decomposing the event $\mathscr{A}$ according to which interval of the type $|x-2i\lambda| <\lambda$ contains the maxima of $\cA_2$ restricted to $[-\li,\li]$, we easily see $\mathscr{A}\subseteq \cup_{|i|\leq\frac{\li}{2\si}} \mathscr{A}_{i}$.
We will show that $\mathscr{A}_{i}\subseteq \mathscr{F}_i\cup \mathscr{C}_i$ in the proof of Proposition \ref{prop:lbae}, and hence $$\mathscr{A}\subseteq \cup_{|i|\leq\frac{\li}{2\si}} \mathscr{F}_i\cup \mathscr{C}_i.$$
We next show that $\cup_{|i|\leq\frac{\li}{2\si}} \mathscr{C}_i$ has small probability thereby establishing the claim that $\bigcup_{|i|\leq \frac{\li}{2\si}}\mathscr{F}_{i}$ approximately covers the event $\mathscr{A}$.

\begin{lemma}
\label{l:error}
Given $\varepsilon>0$ there exists (small) $\alpha>0$ such that uniformly in $0<\lambda<\lambda'<1$ we have $\P(\cup_{|i|\leq \frac{\Lambda}{2\lambda}}\mathscr{C}_i)\leq \varepsilon$. 
\end{lemma} 

\begin{proof}
The proof will again be based on  Brownian comparison. Let $W$ denote a two sided Brownian motion {with diffusivity $2$}. Let $\mathscr{C}'_i$ denote the same event as $\mathscr{C}_i$ replacing $\cA_2$ by $W$. By Corollary \ref{c:BrownianAiry} it suffices to bound the probability of $\cup_{|i|\leq \frac{\Lambda}{2\lambda}} \mathscr{C}'_i$. Let $\mathscr{C}'$ denote the event 
\begin{multline*}
\mathscr{C}':=\left\{\sup_{|x|\leq \lambda} W(x)= \sup_{|x|\leq \sqrt{2\mathscr{M}}} W(x) \leq \sup_{2^{k-1}\lambda'\leq |x|\leq 2^{k}\lambda'} W(x)+2^{k(\frac{1}{2} - \tau)} \alpha\sqrt{\lambda'}~\text{for some}~k\leq k_* \right\}
\\
{\bigcup \left\{\sup_{|x|\leq \lambda} W(x)= \sup_{|x|\leq \sqrt{2\mathscr{M}}} W(x) \geq W(0) + \alpha^{-1}\sqrt{\si'} \right\}.}
\end{multline*}
By translation invariance, {for each $i$, the law of $W$ and $x\mapsto W(x + 2i\si)-W(2i\si)$ are the same,
and this implies that $\P(\mathscr{C}'_i)\leq \P(\mathscr{C}')$.
}
Therefore it suffices to show that 
$$\lim_{\alpha \to 0} \lambda^{-1}\P(\mathscr{C}')\to 0$$
uniformly in $0<\lambda<\lambda'<1$. The remaining details of the proof involve Brownian computations that are deferred to Appendix \ref{s:appa}.
\end{proof}

We are now ready to prove Proposition \ref{prop:lbae}.

\begin{proof}[Proof of Proposition \ref{prop:lbae}]
We show that $\mathscr{A}_{i}\subseteq \mathscr{F}_i\cup \mathscr{C}_i$.
First, on $\mathscr{A}_i$, since $\mathscr{M}>1$ (by our choice of $\mathscr{M}$ to be large enough) we have
\begin{align*}
\sup_{|x-2i\lambda|\leq \lambda}
\biggl(\cA_2(x)-(x-2i\lambda)^2\biggr)
&\geq
\sup_{|x-2i\lambda|\leq \lambda}
\cA_2(x)-\si^2
\geq
\sup_{\si < |x-2i\lambda|\leq \si'}
\cA_2(x)-\si^2
\\
&\geq
\sup_{\si < |x-2i\lambda|\leq \lambda'} \biggl(\cA_2(x)-(x-2i\lambda)^2\biggr),
\end{align*}
where in the first inequality we used 
$\sup_{|x-2i\lambda|<\lambda} \cA_2(x)=\sup_{|x|<\Lambda + \sqrt{2\mathscr{M}}}\cA_2(x)$ and $\mathscr{M}>1$ and hence since $\lambda'<1,$ $$\sup_{|x-2i\lambda|\leq \lambda}
\cA_2(x)\ge \sup_{\lambda< |x-2i\lambda|\leq \lambda'}
\cA_2(x).$$
Thus the event $\cce_i$ is automatically satisfied.
Second, on $\mathscr{A}_i$, if $\alpha<1$, $\mathscr{M}>1$, and $2^{k}\lambda'\geq 4$ (which is satisfied by $k\ge k_*$), we have
\begin{align*}
   \sup_{|x-2i\lambda|\leq \lambda'}
\biggl(\cA_2(x)-(x-2i\lambda)^2\biggr)
&\geq
\sup_{|x-2i\lambda|\leq \lambda'} \cA_2(x)-1\\
&\geq
\sup_{2^{k-1}\si'\leq|x-2i\lambda|\leq 2^k\si'} \cA_2(x)
-(2^{k-1}\si')^2/2,\,\,\, {\text{ by definition of } \mathscr{A}_i}\\
&\geq
\sup_{2^{k-1}\si'\leq|x-2i\lambda|\leq 2^k\si'} \cA_2(x)-(2^{k-1}\si')^2
+ \sqrt{2^{k}\si'}\\
&\geq
\sup_{2^{k-1}\lambda'\leq |x-2i\lambda|\leq 2^{k}\lambda'} \biggl( \cA_2(x)-(x-2i\lambda)^2\biggr)+2^{k(\frac{1}{2} - \tau)} \alpha\sqrt{\lambda'},
\end{align*}
where in the penultimate inequality we have used the simple fact,  $$-(2^{k-1}\si')^2/2\geq -(2^{k-1}\si')^2
+ \sqrt{2^{k}\si'}
$$ when $2^k\si'\ge 4,$ and the last inequality is rather straightforward.
The above thus shows that the event $\cce_{i,k}$ is  
satisfied as well. 
Thus by definition on $\mathscr{A}_{i}\setminus \mathscr{F}_i$, {either $\cce_{i,*}$ is violated, or $\cce_{i,k}$ is violated for some $k< k_*$.}  It follows that $\mathscr{A}_i\setminus \mathscr{F}_{i}\subseteq \mathscr{C}_{i}$ which immediately implies $\mathscr{A}_{i}\subseteq \mathscr{F}_i\cup \mathscr{C}_i$, as desired. We now have 
$$\frac{1}{2}\leq \P(\mathscr{A}) \leq \sum_{|i|\leq\frac{\li}{2\si}}\P(\mathscr{F}_{i})+\P(\cup_{|i|\leq\frac{\li}{2\si}} \mathscr{C}_{i}).$$
By Lemma \ref{l:error}, we can choose $\alpha$ sufficiently small (uniformly in $0<\lambda<\lambda'<1$) such that $\P(\cup_{|i|\leq\frac{\li}{2\si}} \mathscr{C}_{i})\leq \frac{1}{4}$. By noticing that translation invariance implies that $\P(\mathscr{F}_i)$ is same across all $i$, it follows that $\left(\frac{\li}{\si}+1\right)\P(\mathscr{F}_0)\geq \frac{1}{4}$. Recalling that $\mathscr{F}_0$ is the event in the statement of the proposition, and reducing the value of $\alpha$ if necessary, we get $\P(\mathscr{F}_0)>\alpha \lambda$, completing the proof of the proposition. 
\end{proof}

\subsection{Possible alternate approach}
\label{s:kpz}
As indicated in the introduction, there have been several recent advances in obtaining quantitative Brownian comparison estimates for the geodesic weight profile. We have used the estimates of \cite{HHJ+}, namely Theorem \ref{t.airytail}, as the strong control on the Radon-Nikodym derivatives present there was sufficient for our purpose of obtaining the temporal correlation exponent. However the price we pay is the diverging, albeit subpolynomial correction factor in the upper bound of Theorem \ref{t:upper}. It is possible that using some of the other approaches one might be able to get an improved upper bound. Below we informally report our findings exploring one such possible route following the approach taken in \cite{MQR17} in constructing the KPZ fixed point. Recalling $\sL(x):=\cA_2(x)-x^2,$ using the determinantal formulae for the Airy kernel and the general approach of \cite{MQR17} it seems possible to obtain estimates of the form 
\begin{equation}
\label{e:sz}
\P\left[\max_{x \in I} \sL(x) > \max_{x \in [-2M, 2M]} \sL(x)-\sqrt{\varepsilon},\; \max_{x\in [-2M, 2M]} \sL(x) \in [-M', M']\right] \leq C 
\varepsilon
\end{equation}
for all $I\subseteq [-M,M]$ with $|I|\leq \varepsilon\in (0,\frac{1}{2})$ and $C$ is a constant depending on $M,M'$.

Note that comparing \eqref{e:sz} with Proposition \ref{p:twopeaks}, while we do not need the additional factor $\exp\left(C|\log (\varepsilon)|^{5/6}\right)$ on the right hand side, we require the extra assumption that $\max_{x\in [-2M, 2M]} \sL(x)$ lies in a compact interval. Since to compute correlations one would need to control the contribution even from rare events where $\max_{x\in [-2M, 2M]} \sL(x)$ behaves atypically, the above falls slightly short of being enough for our applications. We do not pursue this further precisely but would like to point the reader to a discussion of a similar flavor set to appear in the forthcoming work of the third author with James Martin and Allan Sly \cite{MSZ}. To conclude this discussion we believe this line of work is of independent interest and whether one can sharpen the above result to be applicable in our setting remains a problem for further research. 
\subsection{Finite $n$ bounds for exponential LPP}
\label{s:finite}
As indicated in Section \ref{iop}, we are working with  non-asymptotic quantities in the exponential LPP models and hence we shall need to extract estimates for certain events involving the line-to-point profile $\sL_n$ (see \eqref{profile12} for the formal definition) from the corresponding results for $\sL$ proved in the previous subsection. This follows by invoking the previously mentioned convergence of $\sL_{n}$ to $\sL$, which we now state formally. 
{We start by extending the definition of $\sL_n$ to all of $\R$ in the following trivial manner: when $|x(2n)^{2/3}|>n$ we let $\sL_n(x)=-\infty$, and we linearly interpolate between points in $(2n)^{-2/3}\Z$.}

\begin{theorem}
\label{t:lpptoairy}
As $n\to \infty,$ the profile $\sL_{n}$ converges weakly to $\sL$ in the topology of uniform convergence on compact sets.
\end{theorem}

Using Theorem \ref{t:lpptoairy} one can now get versions of Propositions \ref{p:twopeaks} and \ref{prop:lbae} for the profile $\sL_{n}(\cdot)$. We start with the pre-limiting version of the former.

\begin{proposition} 
\label{prop:bdi}
{For each $M\geq 1$,there exists a  constant $C>0$ depending only on $M$ with the following property: for any
$0 < \iota < \varepsilon < \frac{1}{18}$, for all $n\geq n_0(\iota, M)$ and for all discrete intervals $I\subset \llbracket-Mn^{2/3},Mn^{2/3}\rrbracket$ with $\iota n^{2/3}\leq  |I| \leq \varepsilon n^{2/3}$,}
$$\P\left[\max_{u\in I} T_{(u,-u),\mathbf{n}} > \max_{u\in \llbracket-2Mn^{2/3}, 2Mn^{2/3}\rrbracket} T_{(u,-u),\mathbf{n}}-\sqrt{\varepsilon}n^{1/3}\right] \leq C\varepsilon \exp\left(C|\log(\varepsilon)|^{5/6}\right).$$
\end{proposition}
For the next result, 
we fix $\tau \in (0,1/2)$ (say $\frac{1}{4},$ for concreteness) and take $\alpha$ as in the statement of Proposition \ref{prop:lbae}, which now is also an absolute constant.
Take $\theta, \delta \in (0,1)$, and $n, r\in \N$, with $\delta n< r <n$, and let $H_0=H_0(n,r)$ be the event where
$$
\max_{|u| < r^{2/3}} T_{(u,-u), \bn} =\max_{|u| < \theta r^{2/3}} T_{(u,-u), \bn} {< T_{\mathbf{0}, \bn} + 2\alpha^{-1}r^{1/3}}.
$$
For each $j\in \Z_+$, let $H_j=H_j(n,r)$ denote the event where
$$
\max_{2^{j-1} r^{2/3} \leq |u| < 2^{j} r^{2/3}} T_{(u,-u), \bn} < \max_{|u| < r^{2/3}} T_{(u,-u), \bn} - 2\alpha\cdot 2^{j\left(\frac{1}{2}-\tau\right)}r^{1/3}.
$$
We define $\cE_{n, r}:=\bigcap_{j\in\Z_{\geq 0}}H_j$; this event ensures that the maxima of the profile $\sL_{n}(\cdot)$ is on an interval of size $\theta r^{2/3}$ around $0$, and the profile decays sufficiently fast away from the maxima. The following pre-limiting analogue of Proposition \ref{prop:lbae} gives a lower bound of this probability.

\begin{proposition}
\label{p:maxdecay}
There exist absolute constants $c_0$ such that for all $\delta, \theta \in (0,1)$, there exists $n_0=n_0(\delta,\theta)$ such that for all $n\geq n_0$ and all $r$ with $\delta n < r <n$ we have $\P(\cE_{n, r})\geq c_0 \theta r^{2/3}n^{-2/3}$.
\end{proposition}

The proofs of the above three results are not central to this paper and hence postponed to Appendix \ref{s:appb}. As already mentioned in the introduction Theorem \ref{t:lpptoairy} has already appeared in the literature (see e.g.\ \cite{FO17}), but we shall provide a proof for completeness. Proposition \ref{prop:bdi} and Proposition \ref{p:maxdecay} are then deduced from Theorem \ref{t:lpptoairy} using standard  Portmanteau type weak convergence statements.

Given the above propositions at our disposal we will not be further needing any other property of the Airy process and the remainder of the arguments will only rely on one point moderate deviation estimates.

\section{Estimates on the Geometry and Weight of Optimal Paths}
\label{s:prelim}

In this section we shall gather some technical results about the geometry of geodesics and weights of optimal paths in various constrained and unconstrained settings. These results will be crucially used throughout for the rest of the paper. All of these results depend only on one point moderate deviation estimates for point-to-line and point-to-point passage times. These types of estimates have recently proved to be useful in a number of different settings and a number of them (and their variants) have already featured in the literature \cite{BSS14, BSS17++, BG18, BGH18}.
However, for the sake of completeness and to remove any dependence on results as yet unpublished, we shall provide complete proofs of all the estimates used in this paper. We shall point out the relevant places in the literature where our arguments are adapted, and in some cases reproduced, from as and when required.

As a word of caution, the proofs appearing in this section are somewhat technical; however skipping those will not hamper the reader's ability to follow the arguments in the rest of the paper. 

We start by recalling the one point estimates and some of their immediate consequences.

\subsection{Moderate Deviation Estimates and Consequences}

Recall that $T_{u,v}$ denotes the last passage time from $u$ and $v$ and $X_{r}$ denote the last passage time from $\L_0$ to $\br$. Let us denote by $T^{(*)}$ and $X^{(*)}$ respectively where the passage times also include the weight of the last vertex. The one point estimates are obtained from the following correspondences (see \cite[Proposition 1.4]{Jo99} and \cite[Proposition 1.3]{BGHK19}):
\begin{itemize}
 \item[(i)] $T^{(*)}_{\mathbf{0},(m,n)}$ has the same law as the largest eigenvalue of $X^*X$ where $X$ is an $(m+1)\times (n+1)$ matrix of i.i.d.\ standard complex Gaussian entries, and 
 \item[(ii)] $2X^{(*)}_r$ has the same law as the largest eigenvalue of $X'X$ where $X$ is a $(2r+2)\times (2r+1)$ matrix of i.i.d.\ real standard Gaussian entries. 
 \end{itemize}
 It is easy to see that ignoring the weight of the last vertex does not change any estimates for, large $m,n$ and $r$ and hence we get the following one point estimates from \cite[Theorem 2]{LR10}. 

\begin{theorem}
\label{t:onepoint}
For each $\psi>1$ There exists $C,c>0$ depending on $\psi$ such that for all $m,n,r\geq 1$ with $\psi^{-1}<\frac{m}{n}< \psi$ and all $x>0$ we have the following:
\begin{enumerate}
\item[(i)] $\P(T_{\mathbf{0}, (m,n)}-(\sqrt{m}+\sqrt{n})^{2} \geq xn^{1/3}) \leq Ce^{-c\min\{x^{3/2},xn^{1/3}\}}$.
\item[(ii)] $\P(T_{\mathbf{0}, (m,n)}-(\sqrt{m}+\sqrt{n})^{2} \leq -xn^{1/3}) \leq Ce^{-cx^3}$.
\item[(iii)] $\P(X_r-4r \geq xr^{1/3}) \leq Ce^{-c\min\{x^{3/2},xr^{1/3}\}}$.
\item[(iv)] $\P(X_r-4r \leq -xr^{1/3}) \leq Ce^{-cx^3}$.
\end{enumerate}
\end{theorem}

Observe that Theorem \ref{t:onepoint} implies that 
\begin{equation}
\label{e:mean}
|\E T_{\mathbf{0}, (m,n)} -(\sqrt{m}+\sqrt{n})^2|\leq C'n^{1/3}
\end{equation}
for some constant $C'$ depending only on $\psi$. Although we shall be primarily caring about passage times between pairs of points with slope bounded away from $0$ and $\infty$, the result in \cite{LR10} also gives some amount of control without the slope condition which will be enough for our purposes: namely we have, for $m\geq n\geq 1$ and for all $x>0$
\begin{equation}
\label{e:steep}
\P(T_{\mathbf{0}, (m,n)}-(\sqrt{m}+\sqrt{n})^{2} \geq xm^{1/2}n^{-1/6}) \leq Ce^{-cx}.
\end{equation}

We shall next quote a result about last passage times across parallelograms. These were proved in \cite{BSS14} for Poissonian LPP (see Proposition 10.1, Proposition 10.5 and Proposition 12.2 in \cite{BSS14}). For completeness, in Appendix \ref{s:appc} we include complete proofs in our setting, i.e., for exponential LPP, using the point-to-point estimates of Theorem \ref{t:onepoint} as the input.

Consider the parallelogram $U$ whose one pair of sides lie on $\L_0$ and $\L_r$ with length $2r^{2/3}$ and midpoints $(mr^{2/3},-mr^{2/3})$ and $\br$ respectively. 
Let $U_1$ (resp.\ $U_2$) denote the intersections of $U$ with the strips $\{0\leq x+y\leq 2r/3\}$ and $\{4r/3\leq x+y\leq 2r\}$ respectively.
For $u,v\in U$, let $T^{U}_{u,v}$ denote the weight of the highest weight path between $u$ and $v$ that does not exit $U$.
Further, for $u=(u_1,u_2),$ let $d(u):=u_1+u_2.$

The next set of results show that it is unlikely that there exists $u,v\in U$ that are well separated in the time direction (i.e., $|d(u)-d(v)|$ is sufficiently large) that wither $T_{u,v}$ is much larger that its expectation or $T_{u,v}^{U}$ is much smaller than $\E T_{u,v}$. 
\begin{theorem}
\label{t:supinf}
For each $\psi<1$, there exists $C,c>0$ depending only on $\psi$ such that for all $|m|<\psi r^{1/3}$ and $U$ as above we have
\begin{enumerate}
\item[(i)] 
for all $x, L>0$ and $r$ sufficiently large depending on $L$,
$$\P\left( \inf_{u,v\in U: d(v)-d(u)\geq \frac{r}{L}}  (T_{u,v}-\E T_{u,v}) \leq -xr^{1/3}\right)\leq Ce^{-cx^3}.$$
\item[(ii)]
for all $x>0$ and $r\geq 1$,
$$\P\left( \sup_{u\in U_1,v\in U_2}  (T_{u,v}-\E T_{u,v}) \geq xr^{1/3}\right)\leq Ce^{-c\min\{x^{3/2},xr^{1/3}\}}.$$
\item[(iii)] 
for all $x, L>0$ and $r$ sufficiently large depending on $L$,
$$\P\left( \inf_{u,v\in U: d(v)-d(u)\geq \frac{r}{L}}  (T_{u,v}^U-\E T_{u,v}) \leq -xr^{1/3}\right)\leq Ce^{-cx}.$$
\end{enumerate}
\end{theorem}
The corresponding results in \cite{BSS14} are stated in a slightly more general form; there all pairs $u,v\in U$ such that the line joining $u$ and $v$ has slope bounded away from $0$ and $\infty$ are considered. The proof we provide can be adapted to also cover this more general case, but the above statement will be sufficient for all our applications. Analogous results have also been subsequently developed under rather general assumptions in the recent work \cite{BGHH} by the first two named authors along with Alan Hammond and Milind Hegde.

Note that by taking $r$ sufficiently large, all pairs $u,v$ considered in the statements above satisfy the slope condition. As mentioned above, the proof of Theorem \ref{t:supinf} is deferred to Appendix \ref{s:appc}. Specifically, part (i) is proved in Appendix \ref{s.thm4.2(i)}, part (ii) is proved in Appendix \ref{s.thm4.2(ii)} and part (iii) is proved in Appendix \ref{s.thm4.2(iii)}.

The next result is similar to Theorem \ref{t:supinf} but it controls the maximum passage time from a line segment to a line. Let $A$ denote the line segment on $\L_r$ with endpoints $(r+r^{2/3},r-r^{2/3})$ and $(r-r^{2/3},r+r^{2/3})$. We have the following result. 

\begin{proposition}
\label{p:l2l}
There exists $C,c>0$ such that for all $r\geq 1$, $x>0$, and $A$ as above we have
$$\P(\sup_{u\in A} X_{u}\geq 4r+xr^{1/3}) \leq Ce^{-c\min\{x^{3/2},xr^{1/3}\}}.$$
\end{proposition}

\begin{proof}
Since $X_{u} + T_{u, 2\br} \leq X_{2r}$ for any $u \in A$, we have
\begin{eqnarray*}
\P\left[\max_{u \in A} X_{u} \geq 4r+xr^{1/3} \right] 
&\leq &
\P\left[X_{2r} - \min_{u \in A} T_{u, 2\mathbf{r}} \geq 4r+xr^{1/3} \right]
\\
&\leq &
\P\left[X_{2r} \geq 8r+\frac{xr^{1/3}}{2} \right]
+
\P\left[\min_{u \in A}T_{u, 2\mathbf{r}} \leq 4r-\frac{xr^{1/3}}{2} \right].
\end{eqnarray*}
Using Theorem \ref{t:onepoint} (iii) for the first term and Theorem \ref{t:supinf} (i) for the second term we see that both the terms are upper bounded by $Ce^{-c\min\{x^{3/2},xr^{1/3}\}}$ for some $C,c>0$, thus completing the proof. 
\end{proof}

As a consequence of Proposition \ref{p:l2l}, we can also control the passage time from a line to a larger line segment. For $\phi>1$, Let $A_{\phi}$ denote the line segment on $\L_r$  given by $|x-y|\leq \phi r^{2/3}$. We have the following result. 

\begin{lemma}
\label{l:l2lbd}
There exists $C,c>0$ such that for all $r\geq 1$, $x>0$, and $A_{\phi}$ as above we have
$$\P(\sup_{u\in A_{\phi}} X_{u}\ge 4r+x  r^{1/3})\leq C\phi e^{-c\min\{x^{3/2},xr^{1/3}\}}.$$
\end{lemma}

\begin{proof}
Divide $A_{\phi}$ into $\phi$ many line segments $A_{j,\phi}$ each of length $r^{2/3}$ each. Write 
$$\sup_{u\in A_{\phi}} X_{u}=\max_{j} \sup_{u\in A_{j,\phi}} X_{u}.$$
The result follows by Proposition \ref{p:l2l}, translation invariance of the underlying field and a union bound. 
\end{proof}

Then next result controls the passage times constrained to be in a thin cylinder. For $r\in \N$ and $\theta>0$ (small) let $R=R_{\theta}$ denote the rectangle 
$$R=R_{\theta}:= \{(u_1,u_2)\in \Z^2: 0\leq u_1+u_2 \leq 2r: |u_1-u_2|\leq \theta r^{2/3}\}.$$
Let $A$ and $B$ denote the sides of $R$ aligned with $\L_r$ and $\L_0$ respectively. For $u\in B$ and $v\in A$, recall that $T_{u,v}^{R}$ denotes the maximum passage time from $u$ to $v$ constrained to not exit $R$. Let us denote $Y=T_{\mathbf{0},\br}^{R}$ and $Y_{*}=\inf_{u\in B,v\in A } T_{u,v}^{R}$. We have the following proposition. 

\begin{proposition}
\label{p:constrained}
There exists $C_1,C_2>0$ such that for all $\theta$ sufficiently small and all $r\geq 1$ we have 
\begin{enumerate}
\item[(i)] $\E[Y_{*}]\geq 4r-C_1\theta^{-1}r^{1/3}$.
\item[(ii)] $\E[(Y_{*}-4r)^{4}], \E[|Y-4r|^{4}]\leq C_2\theta^{-4}r^{4/3}$.
\end{enumerate}
\end{proposition}

\begin{proof}
For $i=0, 1,2,\ldots , \theta^{-3/2}$ (assume without loss of generality that $\theta^{-3/2}$ is an integer) let $A_{i}$ denotes the line segment given by the intersection of the line $x+y=2i\theta^{3/2}r$ with $R$. Let $Z_{i}:=\inf _{u\in A_{i}, v\in A_{i+1}} T_{u,v}^{R}$ and $Z'_{i}:=\sup _{u\in A_{i}, v\in A_{i+1}} T_{u,v}^{R}$. Clearly we have 
$$\sum Z_{i} \leq Y_{*} \leq Y \leq \sum Z'_{i}.$$
Observe that if follows from Theorem \ref{t:supinf} (i) that $\P(Z_{i}\leq 4\theta^{3/2}r-x\theta^{1/2}r^{1/3})\leq e^{-cx^3}$ and hence $\E Z_{i}\geq 4\theta^{3/2}r-C_1\theta^{1/2}r^{1/3}$. Item (i) of the proposition immediately follows. Notice that Theorem \ref{t:supinf} (ii) also implies that 
$\P(Z'_i\geq 4r+xr^{1/3})\leq e^{-c\min\{x^{3/2},xr^{1/3}\}}$. Item (ii) now follows by noticing that $Z_{i}$ (also $Z'_i$) are independent across $i$ and invoking standard concentration inequalities for sums of independent subexponential variables (see e.g.\ \cite[Theorem 2.8.1]{V18}).
\end{proof}

The next lemma controls the second moment of the difference of passage times from nearby points. Recall that for $u=(u_1,u_2)\in \Z^2$. 

\begin{lemma}
\label{l:sideregularity}
Take $m=0$ in the definition of $U$, i.e. $U=\{(u_1, u_2) : 0 \leq u_1+u_2 \leq 2r, |u_1-u_2| \leq 2r^{2/3}\}$.
Let $U'$ be the rectangle 
$\{(x, y) \in U: x+y\geq r\}$. 
Then we have for all $r\geq 1$
\begin{enumerate}
\item[(i)] $\E[\displaystyle{(\sup_{u\in U'} (X_u-2d(u)))^2}]=O(r^{2/3}).$
\item[(ii)] $\E[\displaystyle{(\inf_{u\in U'} (T^{U}_{\mathbf{0},u}-2d(u)))}^2]=O(r^{2/3}).$
\end{enumerate} 
\end{lemma}

\begin{proof}
Observe that the same argument as in the proof of Proposition \ref{p:l2l} implies
$$\P\left(\sup_{u\in U'} (X_u-2d(u)) \geq xr^{1/3}\right)\leq e^{-cx}.$$
On the other hand Theorem \ref{t:supinf}(iii) gives
$$\P\left(\inf_{u\in U'} (T^{U}_{\mathbf{0},u}-2d(u)) \leq -xr^{1/3}\right)\leq e^{-cx}.$$
The proof of the lemma is completed by noting that for each $u\in U'$ we have $X_{u}\geq T^{U}_{\mathbf{0},u}$.
\end{proof}
The next estimate is designed to establish that: (i) paths with large transversal fluctuations are unlikely to be geodesics, and (ii) paths with large transversal fluctuations typically have much smaller weight than the geodesics. Denote $\L_{r,r^{2/3}}:=\{v \in \L_{r}: v=\br+ (v_1,-v_1),\,\, |v_1|\leq r^{2/3}\}$.
\begin{proposition}
\label{l: prep1-tf}
There exist constants $c_1,c_2>0$ such that for all sufficiently large $\phi$, and all $r\geq 1$, the event that there exists a path $\gamma$ from {$\L_0$ to $\L_{r,r^{2/3}}$} which exits the strip $|x-y|\leq \phi r^{2/3}$ and has $\ell(\gamma)\ge 4r-c_1\phi^2 r^{1/3}$, denoted by $\mathrm{LargeTF}{(\phi, r)}$ satisfies $$\P(\mathrm{LargeTF}{(\phi, r)})\leq e^{-c_2\phi^3}.$$ 
\end{proposition}
The proof of Proposition \ref{l: prep1-tf} is accomplished by applying Theorem \ref{t:supinf} (ii) together with a chaining argument which first appeared in \cite{BSS14} to obtain the upper bound for transversal fluctuation in Poissonian LPP. Our proof of Proposition \ref{l: prep1-tf} is a refinement of the same argument and we defer the proof to Appendix \ref{s.ppn4.7}.

Together with Theorem \ref{t:onepoint}, Proposition \ref{l: prep1-tf} immediately gives the following result. Let $U_{\phi}$ denote the strip $|x-y|\leq \phi r^{2/3}$. Let $T_{\mathbf{0},\br}^{U_{\phi}}$ denote the weight of the highest weight path from $\mathbf{0}$ to $\br$ that does not exit $U_{\phi}$ and let $X_{u}^{U_{\phi}}$ denote the weight of the best path with one endpoint on $\L_0$ and the other end point $u=(u_1,u_2)$ on $\L_{r}$ satisfying $|u_1-u_2|\leq r^{2/3}$ that does not exit $U_{\phi}$.

\begin{lemma}
\label{l:transversal}
There exists constants $C,c>0$ such that for all $\phi>0$ and $r\geq 1$, and for $u$ as above we have
\begin{enumerate}
\item[(i)] $\P(T_{\mathbf{0},\br}^{U_{\phi}}\neq T_{\mathbf{0},\br}) \leq Ce^{-c\phi^3}$.
\item[(ii)] $\P(X_{u}^{U_{\phi}}\neq X_u) \leq Ce^{-c\phi^3}$.
\end{enumerate}
\end{lemma}

\subsection{Lower Bounds}
In this subsection we shall show that two events, even though not typical, does hold with probability uniformly bounded away from $0$. Recall the thin rectangle $R=R_{\theta}$ from the previous subsection. We have the following result which shows that the best path from $\mathbf{0}$ to $\br$ constrained to not exit $R$ still can take arbitrarily large values with small but uniformly positive probabilities.

\begin{lemma}
\label{l:uppertail}
For every fixed value of $\theta$, there exists $C,c>0$ such that for every $x>0$ we have for all sufficiently large $r$ (depending on $x$) 
$$\P(T^{R}_{\mathbf{0},\mathbf{r}}>4r+xr^{1/3})\geq Ce^{-cx^{3/2}}.$$
\end{lemma}

\begin{proof}  
Without loss of generality we shall prove this result for $\theta=1$. Clearly it suffices to prove the lemma for $x$ sufficiently large, fix such an $x$. Observe that for any $k\in \N$,  $T^{R}_{\mathbf{0},\mathbf{r}}$ is stochastically larger than the sum of $k$ independent copies of $T^R_{\mathbf{0},\frac{1}{k}\mathbf{r}}$. Now observe that,
$$\P\left(T^R_{\mathbf{0},\frac{1}{k}\mathbf{r}} \geq \frac{4r}{k}+ (\frac{r}{k})^{1/3}\right) \geq \P\left(T_{\mathbf{0},\frac{1}{k}\mathbf{r}} \geq \frac{4r}{k}+ (\frac{r}{k})^{1/3}\right)-\P(T^R_{\mathbf{0},\frac{1}{k}\mathbf{r}}\neq T_{\mathbf{0},\frac{1}{k}\mathbf{r}}).$$
By Lemma \ref{l:transversal}, the second term goes to $0$ as $k\to \infty$ uniformly in all large $r$. Using this, the fact that $r^{-1/3}(T_{\mathbf{0},\mathbf{r}}-4r)$ converges to a scalar multiple of the GUE Tracy-Widom distribution and that the support of the GUE Tracy-Widom distribution is all of $\R$, it follows that there exists $\beta>0$ such that for all $k$ large and all $r$ large (depending on $k$) we have
$$\P\left(T^R_{\mathbf{0},\frac{1}{k}\mathbf{r}} \geq \frac{4r}{k}+ (\frac{r}{k})^{1/3}\right) \geq \beta.$$
This in turn implies that 
$$\P\left(T^R_{\mathbf{0},\mathbf{r}} \geq 4r+ k^{2/3}r^{1/3}\right) \geq \beta^{k}.$$ 
Setting $k=x^{3/2}$ completes the proof. 
\end{proof}

The next result shows that there is a positive probability for an on-scale rectangle (i.e., an $r\times r^{2/3}$ rectangle) to act as a ``barrier", i.e., any path crossing that rectangle will be heavily penalized. Constructing such an event has been useful to localize geodesics in several settings \cite{BSS14, BSS17++, BG18}. We start with a preliminary result which has already appeared before (see e.g.\ \cite[Lemma 8.3]{BSS14}) but we shall provide a proof for completeness. 

For $\Delta>0$ and $r\in \N$, let $R_{\Delta}$ denote the rectangle 
$$R_{\Delta}=\{(u_1,u_2)\in \Z^2: 0\leq u_1+u_2 \leq 2r: |u_1-u_2|\leq \Delta r^{2/3}\}.$$
Let $A$ and $B$ denote the intersection of $R_{\Delta}$ with $\L_0$ and $\L_r$ respectively. We have the following lemma. 

\begin{lemma}
\label{l:barbasic1}
For any $\Delta, M>0$, and for all $r$ sufficiently large (depending on $\Delta, M$) there exists $\beta=\beta(\Delta, M)>0$ such that
$$\P\left(\sup_{u\in A, v\in B} T_{u,v}\leq 4r-Mr^{1/3}\right)\geq \beta.$$
\end{lemma}

\begin{proof}
Clearly it suffices to prove the result for $\Delta, M$ sufficiently large. Let $\mu$ be a small positive constant to be chosen appropriately later depending on $M$. Let us divide the line segments $A$ and $B$ into line segments of length $(\mu r)^{2/3}$. Let these line segments be denoted $A_{i}: i=1,2,\ldots, 2\Delta \mu^{-2/3}$, and $B_{i}: i=1,2,\ldots, 2\Delta \mu^{-2/3}$. Notice that for each $i,j$ the events 
$$\left\{\sup_{u\in A_i, v\in B_j} T_{u,v}\leq 4r-Mr^{1/3}\right\} $$
are decreasing in the vertex weights. If we could show that for all $i,j$ we have 
\begin{equation}
\label{e:ijlb}
\P\left(\sup_{u\in A_i, v\in B_j} T_{u,v}\leq 4r-Mr^{1/3}\right)\geq \beta_{*}
\end{equation}
for some $\beta_*(\Delta,M,\mu)>0$, then the FKG inequality, which implies positive correlation of monotone events on product spaces (a more detailed discussion can be found immediately following Proposition \ref{prop:event}), would imply the conclusion of the lemma for $\beta= (\beta_*)^{4\Delta^2\mu^{-4/3}}$. We shall now prove \eqref{e:ijlb}. 

Without loss of generality, let us assume that $A_{i}$ has endpoints $(ir^{2/3},-ir^{2/3})$ and $((i+\mu^{2/3})r^{2/3},-(i+\mu^{2/3})r^{2/3})$ and $B_{j}$ has endpoints $\br+ (jr^{2/3},-jr^{2/3})$ and $\br+ ((j+\mu^{2/3})r^{2/3},-(j+\mu^{2/3})r^{2/3})$. Let us set $u_i=-\mu \br+ (ir^{2/3},-ir^{2/3})$ and $v_j=(1+\mu)\br+(jr^{2/3},-jr^{2/3}).$
Notice that 
$$\sup_{u\in A_i, v\in B_j} T_{u,v} \leq T_{u_i, v_j} -\inf_{u\in A_i} T_{u_i,u} -\inf_{v\in B_j} T_{v,v_j},$$
and consider the events
$$\cB_1:= \{T_{u_i,v_j} \leq (1+2\mu)4r -3M r^{1/3}\};$$
$$\cB_2:=\{\inf_{u\in A_i} T_{u_i,u} \geq 4\mu r-Mr^{1/3}\};$$
$$\cB_3:=\{\inf_{v\in B_j} T_{v,v_j} \geq 4\mu r-Mr^{1/3}\}.$$

Clearly, on $\cB_1\cap \cB_2\cap \cB_3$, we have 
$$\left\{\sup_{u\in A_i, v\in B_j} T_{u,v}\leq 4r-Mr^{1/3}\right\}.$$
We are left to provide a uniform, across $i,j$, lower bound for $\P(\cB_1\cap \cB_2\cap \cB_3)$. We start with a lower bound of $\P(\cB_1)$. Notice first that by \eqref{e:mean}, we have $\E T_{u_i,v_j}-(1+2\mu)4r=O(r^{1/3})$ and  by \eqref{profile12}, $r^{-1/3}(T_{u_i,v_j}-(1+2\mu)4r)$ weakly converges to the GUE Tracy-Widom distribution upto a shift and scaling as $r\to \infty$ for a fixed $\Delta$ and fixed $i,j$. Since the Tracy-Widom distribution has support on all of $\R$, one expects that $\P(r^{-1/3}(T_{u_i,v_j}-(1+2\mu)4r)\leq -3M)$ is uniformly bounded away from $0$ for large $r$. While the process convergence result Theorem \ref{t:lpptoairy} can be used to make the above precise, instead we rely on recently established probability tail bounds in \cite[Theorem 1.2]{BGHK19} which proves a  bound for point-to-line passage times. Namely, a direct consequence of the latter is that for any $\mu>0$, $M>0,$ there exists $\beta_*=\beta_*(M)>0$, such that for all large  $r$, $$\P\left(r^{-1/3}(X_{(1+2\mu)r}-(1+2\mu)4r)\leq -3M\right)\ge 3\beta_*.$$
Since by definition $X_{(1+2\mu)r}$ stochastically dominates $T_{u_i,v_j}$, it follows that for all $\mu,$ and $i,j$ we have $$\P(\cB_1)\geq 3\beta_{*}$$ for all sufficiently large $r$. Observe next that $\inf_{u\in A} \E T_{u_i,u} \geq 4\mu r- C(\mu r)^{1/3}$ for some $C>0$ ({it follows from Theorem \ref{t:onepoint}} and \eqref{e:mean} in particular) and and hence it follows from Theorem \ref{t:supinf}(i) that $\P(\cB_2^c)\leq \beta_*$ for $\mu$ sufficiently small. An identical reasoning gives $\P(\cB_3^c)\leq \beta_*$ for $\mu$ sufficiently small. This completes the proof of \eqref{e:ijlb} for $\mu$ sufficiently small and an application of the FKG inequality as explained above completes the proof of the lemma. 
\end{proof}

Next we need a slightly stronger variant of Lemma \ref{l:barbasic1}. Fix $L$ sufficiently large (without loss of generality assume $r/4L$ is an integer). For $i=0,2,\ldots, 4L-1$ let $U_{(i)}$ denote the sub-rectangle of $R_{\Delta}$ consisting of points $u=(u_1,u_2)$ with $d(u)=u_1+u_2\in [\frac{ir}{2L}, \frac{(i+1)r}{2L})$. For $0\leq i<j\leq 4L-1$ with $j\geq i+2$ let us denote by $\mathcal{B}_{i,j}$ the event that 
$$\mathcal{B}_{i,j}:= \left\{ T^{R_{\Delta}}_{u,u'}-\E T_{u,u'} \leq -Lr^{1/3}, ~\forall u\in U_{(i)}, u'\in U_{(j)}\right\}.$$  
Let $\mathcal{B}$ denote the event that 
$$\mathcal{B}:=\left\{T^{R_{\Delta}}_{u,u'}-\E T_{u,u'} \leq -Lr^{1/3}, ~\forall u,u' \in  R_{\Delta} ~\text{with}~ |d(u)-d(u')|\geq \frac{r}{L}\right\}.$$ 
Clearly, for any $u,u' \in  R_{\Delta}$ with $d(u')-d(u)\geq \frac{r}{L}$ there exists $j\geq i+2$ such that $u\in U_{(i)}$ and $u'\in U_{(j)}$ and hence
\begin{equation}
\label{e:incl}
\cB \supseteq \bigcap \cB_{i,j}.
\end{equation}
We have the following lemma. 

\begin{lemma}
\label{l:barbasic}
There exists a constant $\rho=\rho(L,\Delta)$ such that the following hold for all sufficiently large $r$ (depending on $L$ and $\Delta$):
\begin{enumerate}
\item[(i)] For all $i,j$ as above $\P(\mathcal{B}_{i,j})\geq \rho$. 
\item[(ii)] We have $\P(\mathcal{B})\geq \rho^{16L^2}$.
\end{enumerate}
\end{lemma}

\begin{proof}
Observe first that each event $\mathcal{B}_{i,j}$ is a decreasing event in the weights of the vertices on $R_{\Delta}$. Hence (ii) follows from (i) by using \eqref{e:incl} and an application of the FKG inequality. For (i), fix $0\leq i<j\leq 4L-1$ with $j\geq i+2$. Let $L_1$ be the line segment given by the intersection of $R_{\Delta}$ with the straight line $\{u\in \Z^2: d(u)=\frac{(i+1)r}{2L}+\frac{r}{100L}\}$ and let $L_2$ be the line segment given by the intersection of $R_{\Delta}$ with the straight line $\{u\in \Z^2: d(u)=\frac{jr}{2L}-\frac{r}{100L}\}$. Let $\mathcal{B}^1_{i,j}$ denote the event that 
$$ \sup_{u\in L_1,u'\in L_2} T_{u,u'}^{R_{\Delta}}-2|d(u')-d(u)| \leq -(L+C^*\sqrt{\Delta})r^{1/3}.$$
Let $\mathcal{B}^2_{i,j}$ denote the event 
$$\sup_{u\in U_{(i)},u'\in L_1} T_{u,u'}^{R_{\Delta}}-2|d(u')-d(u)| \leq  \frac{1}{3}C^*\sqrt{\Delta}r^{1/3}$$
and similarly let 
$\mathcal{B}^3_{i,j}$ denote the event 
$$\sup_{u\in L_2,u'\in U_{(j)}} T_{u,u'}^{R_{\Delta}}-2|d(u')-d(u)| \leq  \frac{1}{3}C^*\sqrt{\Delta}r^{1/3}.$$
Observing that $\E T_{u,u'}\leq 2|d(u)-d(u')|+\frac{1}{3}C^*r^{1/3}$ for all $u,u'$ for $C^*$ sufficiently large (this is a consequence of Theorem \ref{t:onepoint}, for example) it follows that $\mathcal{B}_{i,j}$ contains $\mathcal{B}^1_{i,j}\cap \mathcal{B}^2_{i,j}\cap \mathcal{B}^3_{i,j}$. Notice also that these three events are independent. {It follows from Theorem \ref{t:supinf}(ii)} that for $C^*$ sufficiently large we have 
$\P(\mathcal{B}^2_{i,j}), \P(\mathcal{B}^3_{i,j}) \geq \frac{1}{2}$ for all $r$ sufficiently large. It also follows from Lemma \ref{l:barbasic1} that $$\P(\mathcal{B}^1_{i,j})\geq \rho' (L,C^*,\Delta)>0,$$ for all $r$ sufficiently large. This completes the proof of (i).
\end{proof}

Given the above preparation we are now ready to dive into the proof of Theorem \ref{t:upper}.

\section{Covariance Upper Bound: Proof of Theorem \ref{t:upper}}
\label{s:upper}
Let us first recall the basic set up. Let $\delta\in (0,\frac{1}{2})$ be fixed and let $M$ be a sufficiently large absolute constant chosen appropriately large later with the only criteria being that \eqref{e:mchoice} is satisfied. It is useful to emphasize that the largeness of $M$ does not depend on $\delta$. Let $n\geq n_0(\delta,M)$ be a sufficiently large positive integer and let a positive integer $r\in (\delta n,\frac{n}{2})$ be also fixed. All constants appearing in the proofs of this section will be independent of $\delta$ but will depend on $M$ unless explicitly mentioned otherwise. As $M$ is an absolute constant, constants depending on $M$ will also be referred to as absolute constants. Recall that $X_n$ (resp.\ $X_r$) denote the line-to-point last passage time from $\L_0$ to $\mathbf{n}$ (resp.\ $\mathbf{r}$). Let $\cf_{r}$ denote the $\sigma$-algebra generated by the random variables $\omega_{v}$ for the set of vertices $v$ on or above the line $\L_{r}$, i.e., $v=(v_1,v_2)$ with $v_1+v_2\geq2r$. We shall also need to consider the line-to-point profile from $\L_r$ to $\mathbf{n}$, i.e., $\{T_{u,\mathbf{n}}:u\in \L_r\}$. Observe that this profile is measurable with respect to the $\sigma$-algebra $\cf_r$. 

\noindent
Recall that $\Gamma_{\bn}$ denotes the line-to-point geodesic from $\L_0$ to $\bn$ achieving the weight $X_n$. Let $u_0$ denote the unique point at which $\Gamma_{\bn}$ intersects $\L_{r}$. As indicated in Section \ref{iop},  our strategy is to show that the events $|u_0-\br|\gg r^{2/3}$ have insignificant contribution to the covariance as in that case, the paths $\Gamma_{\bn}$ and $\Gamma_{\br}$ are likely to pass through disjoint regions. However, notice that the point $u_0$ is not measurable with respect to $\cf_r$, and hence we shall use the following proxy for $u_0$. Let 
\begin{equation}\label{argmax1}
u_{\max}:=\arg\max \{T_{u,\bn}: u\in \L_r\};
\end{equation}
i.e., $u_{\max}$ is the starting point of the line-to-point geodesic $\Gamma_{\bn}^r$ from $\L_r$ to $\bn$. Notice that $u_{\max}$ is measurable with respect to $\cf_r$ and one would expect that if $r\ll n$, $u_{\max}$ and $u_0$ are unlikely to be too far from each other, making the former a natural candidate for the proxy for $u_0$. 
See Figure \ref{f:jpaths} for an illustration.

\begin{figure}[htbp!]
\includegraphics[width=.5\textwidth]{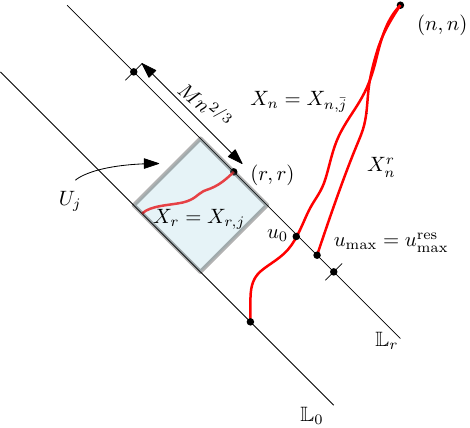} 
\caption{Illustration of relevant notations:  $U_j$ is a rectangle defined in \eqref{recdef20} with width $\log^{10}(j+2)r^{2/3}$.  $u_{\max}^{{\rm{res}}}$ denotes the location at which the geodesic weight to $\bn$ gets maximized in the interval around $(r,r)$ of size $2Mn^{2/3}.$ 
Recalling that $X_n=\ell(\Gamma^{0}_{\br})$, the notation $X_{r,j}$ is used to denote the maximal weight of a path between $\br$ and $\L_0$ restricted to stay within $U_j$ and similarly $X_{n,\bar j}$ denotes the maximal weight path between $\bn$ and $\L_0$ that does not intersect $U_j.$}
\label{f:jpaths}
\end{figure}
\noindent
We shall consider several events depending on the location of $u_{\max}$. The first step is to restrict within a large compact region. As it is unlikely that $|u_{\max}-\br|\gg n^{2/3}$, we make the following definition.
For each $x \in \R_+$, let 
\begin{equation}\label{finitelinesegment}
\L_{r,x}:=\{v \in \L_{r}: v=\br+ (v_1,-v_1),\,\, |v_1|\leq x\}
\end{equation}
denote the line segment centered at $\br$ of length $2x$. Let $M$ be a large absolute constant to be chosen appropriately later. 
Let $u^{\rm{res}}_{\max}$ denote the maxima of the restricted profile $\{T_{v,\bn}: v\in \L_{r,Mn^{2/3}}\}$ and choose $j_0$ such that  
\begin{equation}\label{indexchoice}
j_0^{101}r^{2/3}=Mn^{2/3};
\end{equation}
and without loss of generality let us assume that $j_0$ is an integer. 
(The number $101$ is not anything special, and just an arbitrary choice to ensure that probabilities of certain events, that we decompose our covariance into, are small and summable (see Lemma \ref{lem:bdac} (ii)). This gets applied in the proofs of Lemma \ref{l:bjbound} and \ref{l:cjbound}.) For $j=0,1,2,\ldots, j_0-1$, let $S_j$ denote the subset of $\L_r$ (union of two line segments) consisting of all points $u=\br+(u_1,-u_1)$ such that $|u_1|\in [\frac{1}{2}j^{101}r^{2/3}, \frac{1}{2}(j+1)^{101}r^{2/3})$. 
For $j=0,1,\ldots , j_0-1$, let  $$A_j:=\{u^{\rm{res}}_{\max}\in S_j\}$$  and $B_j \subset A_j$ denote the event where in addition,

\begin{equation}\label{decayneeded}
\max_{u \in \L_{r,(\log(j+2))^{10} r^{2/3}}} T_{u, \bn} < T_{u^{\rm{res}}_{\max}, \bn} - 1000(\log(j+2))^2r^{1/3},
\end{equation}
and finally let us set $C_j:=A_j \setminus B_j$. 
Note that Brownian decay away from the maxima, should imply that the above inequality should typically hold with the factor $1000(\log(j+2))^2r^{1/3}$ replaced by $j^{50.5}r^{1/3}$ (and can be replaced by $j^{50.5-s}r^{1/3}$ for $s>0,$ to ensure a high probability event). However for our purposes the above logarithmic factor suffices. Nonetheless as outlined in Section \ref{gbc}, to handle general initial data, such a crude definition is no longer useful and one has to exploit the full diffusive decay property.    
\begin{figure}[htbp!]
\includegraphics[width=.5\textwidth]{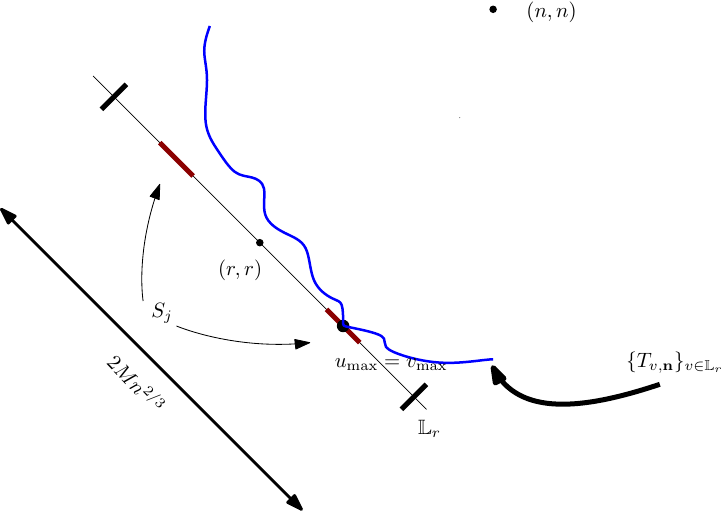} 
\caption{Illustration of the event $A_j$ which says that the maximum of the restricted profile $\{T_{v,\mathbf{n}}\}$ when $v=\mathbf{r}+(u,-u)$ and $u$ varies in the interval $|u|\leq Mn^{2/3}$ occurs at the point $u^{\rm{res}}_{\max}$ in the interval $S_j.$}
\label{f:jpaths1}
\end{figure}

We point out again that the events $A_j, B_j, C_j$ are all measurable with respect to the $\sigma$-algebra $\cf_r$. Our first order of business is to obtain estimates for probabilities of these events using Proposition \ref{prop:bdi}.

\begin{lemma}  
\label{lem:bdac}
There exists an absolute constant $C$ such that in the above set-up we have
\begin{enumerate}
\item[(i)] $\P[A_j] \leq C(j+1)^{100}r^{2/3}n^{-2/3}\exp\left( C\left(\log(n/r)\right)^{5/6} \right)$, for each $j \in \{0,1,,\ldots, j_0-1\}$.
\item[(ii)] $\sum_{j=0}^{j_0-1}\P[C_j] \leq C r^{2/3}n^{-2/3}\exp\left( C\left(\log(n/r)\right)^{5/6} \right)$.
\end{enumerate}
\end{lemma}
\begin{figure}[htbp!]
\includegraphics[width=.7\textwidth]{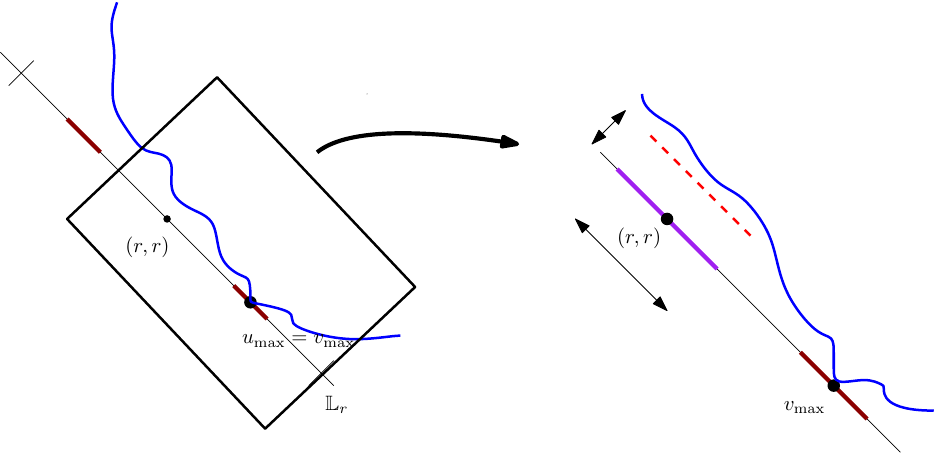} 
\caption{The figure depicts the event $B_j$ which asks that in addition to $A_j,$ the profile $\{T_{v,\mathbf{n}}\}$ has a characteristic ``almost" diffusive  drop characterized by a Brownian motion around its maxima $u^{\rm{res}}_{\max}$ and in particular takes significantly smaller values around $v=\mathbf{r}$.}
\label{f:jpaths2}
\end{figure}

\begin{proof}
These estimates are applications of Proposition \ref{prop:bdi}. Here we are working on the profile from $\L_r$ to $\bn$, which has the same law as the profile from $\L_0$ to $\bn-\br$. We can assume that $\frac{r}{n}$ is small enough so that $2(\log (j_0+1))^{10}r^{2/3}(n-r)^{-2/3} < \frac{1}{18}$, since otherwise the statement holds by taking $C$ large enough.

For (i) we divide $A_j$ into events $$A_{j,i,\pm}:=\{u^{\rm{res}}_{\max}=\br+(u_1,-u_1): u_1 \in \pm[\frac{1}{2}ir^{2/3}, \frac{1}{2}(i+1)r^{2/3})\},$$
for $i=j^{101}, \cdots, (j+1)^{101}-1$.
For each of these events, we apply Proposition \ref{prop:bdi} with the discrete intervals $$I_{i,\pm}=[\frac{1}{2}ir^{2/3}, \frac{1}{2}(i+1)r^{2/3}) \cap \Z,\;\;\varepsilon = \frac{1}{2}r^{2/3}(n-r)^{-2/3},$$ and $\iota = \frac{1}{2}\delta^{2/3}$
(note that the assumption above ensures that $0 < \iota < \varepsilon < \frac{1}{18}$).
Then for each of $i$ we get, using Proposition \ref{prop:bdi} ({recall that $C$ appearing in Proposition \ref{prop:bdi} is independent of $\iota, \varepsilon$ and depends only on $M$}), that $$\P(A_{j,i,\pm})\le C'r^{2/3}(n-r)^{-2/3}\exp\left(C'|\log(r^{2/3}(n-r)^{-2/3})|^{5/6}\right),$$
for some absolute constant $C'$, {when $n=n(\delta,M)$ is large enough}.
Our first estimate follows by taking $C=100C'$ and summing over all $i=j^{101}, \cdots, (j+1)^{101}-1$.

\medskip

For (ii), note that all these $C_j$ are mutually disjoint, and $\bigcup_{j=0}^{j_0-1} C_j$ is contained in the event:
\begin{equation}  \label{eq:lem:bdac:pf1}
\max_{u \in \L_{r,(\log (j_0+1))^{10} r^{2/3}}} T_{u, \bn} \geq T_{u^{\rm{res}}_{\max}, \bn} - 1000(\log(j_0+1))^2r^{1/3}.
\end{equation}
We apply Proposition \ref{prop:bdi} with the following choices:  $$I=[-(\log (j_0+1))^{10}r^{2/3}, (\log (j_0+1))^{10}r^{2/3}] \cap \Z,\,\,
\varepsilon = 2(\log (j_0+1))^{10}r^{2/3}(n-r)^{-2/3}, \text{ and } \iota = \delta^{2/3}.$$
{Note that $j_0\geq (2^{2/3} M)^{1/101}$ and hence by choosing $M$ to be a sufficiently large absolute constant, we have 
\begin{equation}
\label{e:mchoice}
\sqrt{\varepsilon} >1000(\log(j_0+1))^2r^{1/3}n^{-1/3};
\end{equation}
and the assumption above ensures that $0 < \iota < \varepsilon < \frac{1}{18}$}.
Then when $n$ is large enough, we get the desired bound for the probability of the event \eqref{eq:lem:bdac:pf1} and in turn the required upper bound for $\sum_{j=0}^{j_0-1}\P[C_j]$.
\end{proof}

\subsection{Decomposition of the Covariance}\label{covdec}
To control the covariance $\Cov (X_r,X_n)$ our strategy is to first condition on $\cf_r$, then decompose the conditional covariance depending on which of the events $A_j$ occurs on the conditioned environment (observe that for any configuration of weights measurable with respect to $\cf_r$ exactly one of $A_j$, for $j=0,1,\ldots, j_0-1$ holds) and finally average over the conditioned environment. For notational convenience we shall denote conditional covariance given any $\sigma$-algebra $\cg$ by $\Cov_{\cg}(\cdot, \cdot)$. We shall use the following elementary fact: if random variables $Y$ and $Z$ are such that $Y$ is independent of some $\sigma$-algebra $\cg$ then $$\Cov(Y,Z)=\E[YZ]-\E Z\E Y = \E[\E[YZ\mid \cg]-\E[Z\mid \cg]\E[Y\mid \cg]]=\E[\Cov_{\cg}(Y, Z)].$$
Using this and the fact that $X_r$ is independent of the $\sigma$-algebra $\cf_r$, we get
\begin{eqnarray}
\label{e:decomp}
\Cov (X_r,X_n)&=& \E[\Cov_{\cf_r}((X_r,X_n))]= \E \left[\sum_{j=0}^{j_0-1} \left[\don_{A_j}\Cov_{\cf_r}((X_r,X_n)) \right]\right]\nonumber\\
&=&\sum_{j=0}^{j_0-1} \E[\don_{B_j}\Cov_{\cf_r}((X_r,X_n))]+ \sum_{j=0}^{j_0-1} \E[\don_{C_j}\Cov_{\cf_r}((X_r,X_n))].
\end{eqnarray}
To prove Theorem \ref{t:upper} we need to control each of the summands in the RHS above. This is done in the next two lemmas. We start with upper bounding $\E[\don_{B_j} \Cov_{\cf_r}((X_r,X_n))]$.

\begin{lemma}
\label{l:bjbound}
In the above set-up, there exists an absolute constant $C$ such that for all $j=0,1,2,\ldots, j_0-1$ we have 
$$\E[\don_{B_j} \Cov_{\cf_r}((X_r,X_n))] \leq  C(j+1)^{-1000}r^{2/3}\P[B_j]
\biggl((\log(j+2))^2-\log(\P[B_j]) + \log(n/r)\biggr)^2.$$
\end{lemma}

The next lemma upper bounds the terms $\E[\don_{C_j} \Cov_{\cf_r}((X_r,X_n))]$.

\begin{lemma}
\label{l:cjbound}
In the above set-up, there exists an absolute constant $C$ such that for all $j=0,1,2,\ldots, j_0-1$ we have 
$$\E[\don_{C_j} \Cov_{\cf_r}((X_r,X_n))] \leq Cr^{2/3}\P[C_j] \biggl(-\log(\P[C_j])+\log(n/r)\biggr)^{2}.$$
\end{lemma}

We postpone the proofs of Lemma \ref{l:bjbound} and Lemma \ref{l:cjbound} until the next subsection and show first how these estimates can be used to complete the proof of Theorem \ref{t:upper}. The arguments consist of a sequence of playing around with the above bounds and might be slightly difficult to parse at first read.

\begin{proof}[Proof of Theorem \ref{t:upper}]
In this proof, $C$ shall denote an absolute constant whose value might change from line to line.
In light of \eqref{e:decomp}, we first control $\sum_{j=0}^{j_0-1} \E[\don_{B_j}\Cov_{\cf_r}((X_r,X_n))]$.
Observe that by Lemma \ref{lem:bdac} (i) we have $$\P[B_j] \leq \P[A_j] \leq  C(j+1)^{100}r^{2/3}n^{-2/3}\exp\left( C\left(\log(n/r)\right)^{5/6} \right),$$ for all $j<j_0$. Using Lemma \ref{l:bjbound},
\begin{align} \label{e:b0}
\sum_{j=0}^{j_0-1} \E[\don_{B_j}\Cov_{\cf_r}((X_r,X_n))]
\leq
Cr^{2/3}\sum_{j=0}^{j_0-1}\P[B_j](j+1)^{-1000} \left(\log^4(j+2) + \log^2(n/r)+ |\log(\P[B_j])|^2\right).
\end{align}

Splitting the RHS above by separating the first two terms from the third term in the sum, and using the upper bound on $\P[B_j],$ the first sum is bounded by
\begin{align*}
&Cr^{4/3}n^{-2/3}\exp\left( C\left(\log(n/r)\right)^{5/6} \right)\sum_{j=0}^{j_0-1}
\left((\log(j+2))^4 + (\log(n/r))^2\right)(j+1)^{-900}
\\
&\leq Cr^{4/3}n^{-2/3}\exp\left( C\left(\log(n/r)\right)^{5/6} \right).
\end{align*}
For the second sum, note that $|\log(x)|^2x$ increases for $x \in (0, e^{-2})$, and decreases for $x \in (e^{-2}, 1)$.
Thus if $$C'(j+1)^{100}r^{2/3}n^{-2/3}\exp\left( C'\left(\log(n/r)\right)^{5/6} \right) > e^{-2},$$ we have
\begin{align*}
&(j+1)^{-1000}r^{2/3}|\log(\P[B_j])|^2\P[B_j]
\\
&\leq
(j+1)^{-1000}r^{2/3}\cdot
2^2 e^{-2}
\leq 4 C'(j+1)^{-900}r^{4/3}n^{-2/3}\exp\left( C'\left(\log(n/r)\right)^{5/6} \right), 
\end{align*}
and otherwise we have
\begin{align*}
(j+1)^{-1000}r^{2/3}|\log(\P[B_j])|^2\P[B_j]
\leq
\left|\log\left(C'(j+1)^{100}r^{2/3}n^{-2/3}\exp\left( C'\left(\log(n/r)\right)^{5/6}\right)\right)\right|^2\\
\times
C'(j+1)^{-900}r^{4/3}n^{-2/3}\exp\left( C'\left(\log(n/r)\right)^{5/6} \right),
\\
\leq
C'(j+1)^{-900}r^{4/3}n^{-2/3}\exp\left( C'\left(\log(n/r)\right)^{5/6}\right)
\left|\frac{2}{3}\log(n/r) + \log(C')\right|^2.
\end{align*}
In each case, by summing over $j=0,1,\ldots, j_0-1$,
and using the value of $j_0$ (see \eqref{indexchoice}) we can bound the second sum on the right hand side of \eqref{e:b0}  by $$C'''r^{4/3}n^{-2/3}\exp\left( C'''\left(\log(n/r)\right)^{5/6} \right).$$
In conclusion, we have
\begin{equation}
\label{e:b}
\sum_{j=0}^{j_0-1} \E[\don_{B_j}\Cov_{\cf_r}((X_r,X_n))] \leq Cr^{4/3}n^{-2/3}\exp\left( C\left(\log(n/r)\right)^{5/6} \right),
\end{equation}
for some absolute constant $C$.
Having bounded $\sum_{j=0}^{j_0-1} \E[\don_{B_j}\Cov_{\cf_r}((X_r,X_n))]$, we next control $\sum_{j=0}^{j_0-1} \E[\don_{C_j}\Cov_{\cf_r}((X_r,X_n))]$.
By Lemma \ref{l:cjbound} it is bounded by
$$
C\sum_{j=0}^{j_0-1}r^{2/3}\P[C_j] \biggl(-\log(\P[C_j])+\log(n/r) + 2\biggr)^{2}
$$
for some absolute constant $C$.
Note that $x\left(-\log(x) + \log(n/r) + 2\right)^{2}$ is concave for $x \in (0, 1)$.
By Jensen's inequality, this sum is bounded by
$$
Cr^{2/3}\left(\sum_{j=0}^{j_0-1}\P[C_j]\right) \left(-\log\left( \frac{\sum_{j=0}^{j_0-1}\P[C_j]}{j_0} \right)+\log(n/r) + 2\right)^{2}.
$$
Using Lemma \ref{lem:bdac} (ii), {
we have $\sum_{j=0}^{j_0-1}\P[C_j]\leq Cr^{2/3}n^{-2/3}
\exp(C\left(\log(n/r)\right)^{5/6})$.
Noting that $x\left(-\log(x)+\log(j_0)+\log(n/r)+2\right)^{2}$ is increasing for $x \in (0, 1)$, 
if $$C'r^{2/3}n^{-2/3}\exp(C'\left(\log(n/r)\right)^{5/6}) > 1,$$
\begin{align*}
\left(\sum_{j=0}^{j_0-1}\P[C_j]\right) \left(-\log\left( \frac{\sum_{j=0}^{j_0-1}\P[C_j]}{j_0} \right)+\log(n/r) + 2\right)^{2}
\leq
\left(\log(j_0)+\log(n/r) + 2\right)^{2}
\\
\leq
C'r^{2/3}n^{-2/3}
\exp(C'\left(\log(n/r)\right)^{5/6})\left(\log(j_0)+\log(n/r) + 2\right)^{2};
\end{align*}
and otherwise
\begin{align*}
&\left(\sum_{j=0}^{j_0-1}\P[C_j]\right) \left(-\log\left( \frac{\sum_{j=0}^{j_0-1}\P[C_j]}{j_0} \right)+\log(n/r) + 2\right)^{2}
\\
&\leq
C'r^{2/3}n^{-2/3}
\exp(C'\left(\log(n/r)\right)^{5/6})
\left(-\log\left( \frac{C'r^{2/3}n^{-2/3}
\exp(C'\left(\log(n/r)\right)^{5/6})}{j_0} \right)+\log(n/r)+2\right)^{2}.
\end{align*}
In any case, we have
\begin{align}
\label{e:c2}
&\sum_{j=0}^{j_0-1} \E[\don_{C_j}\Cov_{\cf_r}((X_r,X_n))]
\leq
C'C''r^{4/3}n^{-2/3}
\exp\left(C'\left(\log(n/r)\right)^{5/6}\right)
\\
\nonumber
&\times
\left(\left|\log\left( C'r^{2/3}n^{-2/3}
\exp(C'\left(\log(n/r)\right)^{5/6}) \right)\right|+ \log(j_0)+\log(n/r)+2\right)^{2},
\\
\nonumber
&\leq
C'C''r^{4/3}n^{-2/3}
\exp\left(C'\left(\log(n/r)\right)^{5/6}\right)
\left(|C'+\log(C')| + \frac{1}{101}\log(M) +\frac{5}{3}\log(n/r)+2\right)^{2},
\\
\nonumber
&\leq
Cr^{4/3}n^{-2/3}\exp\left( C\left(\log(n/r)\right)^{5/6} \right),
\end{align}
for some absolute constant $C$.} Using \eqref{e:b} and \eqref{e:c2} together with \eqref{e:decomp} gives the desired result.  
\end{proof}

It remains to prove Lemma \ref{l:bjbound} and Lemma \ref{l:cjbound}. We need some preparatory work for the proofs of these results. Our primary objective (for Lemma \ref{l:bjbound}) is to show that on $B_j$ (for $j$ large) covariance is small. As indicated in Section \ref{iop}, the main idea here to approximate $X_n$ and $X_r$ by some other variables which are independent, and then control the error incurred in the approximation. We start by proving some required estimates about the geometry of the optimal paths that are consequences of the results developed in Section \ref{s:prelim}.

\subsection{Approximation and Concentration Estimates}
Consider the set-up in Lemma \ref{l:bjbound}. For $j=0,1,2,\ldots, j_0-1$, let $U_j$ denote the rectangle (see Figure \ref{f:jpaths} for an illustration), 
\begin{equation}\label{recdef20}
U_j:=\biggl\{u=(u_1,u_2)\in \Z^2: 0\leq u_1+u_2 < 2r, |u_1-u_2|\leq (\log(j+2))^{10} r^{2/3}\biggr\}.
\end{equation}
For each ${v\succ \L_0}$ (recall the notation from below \eqref{linedef}), let $X_{v,j}$ denote the maximum passage time among all paths between $\L_0$ to $v$ that are completely contained in the rectangle $U_j$. Similarly, let $X_{v,\bar{j}}$ denote the maximum passage time among all paths from $\L_0$ to $v$ that do not intersect $U_j$. 
Also following our usual convention, we use $X_{r,j}$ to denote $X_{\br, j}$ and $X_{n,\bar{j}}$ to denote $X_{\bn, \bar{j}}$.
Our objective is to prove first that with high probability $X_{r,j}=X_r$ and  that with high conditional probability $X_{n,\bar{j}}=X_n$ where the conditioning is done on $\cf_r$ such that $\don_{B_j}=1$. These are the following two lemmas.

\begin{lemma}
\label{l:transversal1}
The exists $C,c>0$ such that for all $j<j_0$, we have 
$$\P[X_{r,j}\neq X_r]\leq C\exp\left(-c(\log(j+2))^{10}\right).$$
\end{lemma}

This is an immediate consequence of Lemma \ref{l:transversal}.

\begin{lemma}
\label{l:transversal2}
The exists $C,c>0$ such that for all $j<j_0$, we have 
$$\don_{B_j}\P[X_{n,\bar{j}}\neq X_n\mid \cf_{r}]\leq C\exp\left(-c(\log(j+2))^2\right).$$
\end{lemma}

\begin{proof}
We assume that $j > j_1$ for some (large) absolute constant $j_1$, since otherwise the conclusion follows by taking $C$ large enough.
Recall that $u_0=(x_0,y_0)$ denotes the unique point at which $\Gamma_{\bn}$ intersects $\L_r$.
The event $X_{n,\bar{j}}\neq X_n$ is equivalent to $\Gamma_{u_0}\cap U_j \neq \emptyset$,
and implies either
\begin{itemize}
\item $|x_0-y_0| < 2(\log(j+2))^{10} r^{2/3}$ or,
\item $\Gamma_{\br+u_*}\bigcap U_j \neq \emptyset$  for  $u_*= \pm(\lceil (\log(j+2))^{10} r^{2/3} \rceil,-\lceil (\log(j+2))^{10} r^{2/3} \rceil ).$
\end{itemize}
Under the event $B_j$, if $|x_0-y_0| < 2(\log(j+2))^{10} r^{2/3}$, we must have
$$
\max_{u \in \L_{r,(\log(j+2))^{10} r^{2/3}}} X_{u} > X_{u^{\rm{res}}_{\max}} + 1000(\log(j+2))^2r^{1/3},
$$
since on the event $B_j$,  if $|x_0-y_0| < 2(\log(j+2))^{10} r^{2/3}$, the polymer $\Gamma_{(x_{0},y_0)}$ must make for the deficit of weights between $\Gamma_{(x_0,y_0),\bn}$ and $\Gamma_{u^{\rm{res}}_{\max},\bn}$ induced by the event $B_j.$
By Proposition \ref{p:l2l}, if $j_1$ is taken large enough, for any $i \in \Z$ we have,
$$
\P\left[\max_{u \in (i,-i)+ \L_{r,r^{2/3}}}X_{u} > X_{u^{\rm{res}}_{\max}} + 1000(\log(j+2))^2r^{1/3}\right] \leq C'\exp\left(-c'(\log(j+2))^2\right),
$$
for some constants $C', c' > 0$.
Thus by dividing $\L_{r,(\log(j+2))^{10} r^{2/3}}$ into translations of $\L_{r,r^{2/3}}$, we have

\begin{eqnarray*}
\P[|x_0-y_0| < 2(\log(j+2))^{10} r^{2/3}] &\leq & \P\left[\max_{u \in \L_{r,(\log(j+2))^{10} r^{2/3}}} X_{u} > X_{u^{\rm{res}}_{\max}} + 1000(\log(j+2))^2r^{1/3}\right]\\
&\leq & 3C'(\log(j+2))^{10} \exp\left(-c'(\log(j+2))^2\right).
\end{eqnarray*}
Recalling \eqref{recdef20} note that by translation invariance of the underlying field  Lemma \ref{l:transversal1} also implies that for each choice of $u_*$ we get that the probability of $\Gamma_{\br+u_*}\bigcap U_j \neq \emptyset$ is bounded by $C''\exp\left(-c''(\log(j+2))^{10}\right)$, for some constants $C'', c'' > 0$.

Thus we have
$$\don_{B_j}\P[X_{n,\bar{j}}\neq X_n\mid \cf_{r}]\leq 3C'(\log(j+2))^{10}\exp\left(-c'(\log(j+2))^{2}\right)
+2C''\exp\left(-c''(\log(j+2))^{10}\right),
$$
which implies our conclusion.
\end{proof}

We shall also need some concentration estimates for $X_{r,j}$ and $X_{n,\bar{j}}$. 
We start with the following estimates for $X_{r,j}$.
\begin{lemma}
\label{l:rjconc}
In the above set-up, there exists $C,c>0$ such that for all $j<j_0$ and for all $x>0$ sufficiently large we have the following:
\begin{enumerate}
    \item[(i)] $\P[X_r-X_{r,j}\geq xr^{1/3}] \leq C\exp(-cx)$.
    \item[(ii)] $\P[|X_{r,j}-\E X_{r,j}| \geq xr^{1/3}] \leq C\exp(-cx)$.
\end{enumerate}
\end{lemma}

\begin{proof}
Let $X'$ denote the maximum passage time from $\mathbf{0}$ to $\br$ that are contained in $U_0$.
By Theorem \ref{t:onepoint} and Theorem \ref{t:supinf}, when $x$ is large enough we have
\begin{align*}
\P[X_{r,0} - 4r < -xr^{1/3}] &\leq \P[X' - 4r < -xr^{1/3}]
\leq
C'\exp(-c'x^{3/2}), \text{ and,}\\ 
\P[X_r - 4r > xr^{1/3}] &\leq C'\exp(-c'\min\{x^{3/2},xr^{1/3}\}).
\end{align*}
Here $c', C'>0$ are absolute constants.
Thus, since $X_{r,0}\leq X_{r,j}\leq X_r$, 
we must have
\begin{align*}
\P[X_r-X_{r,j}\geq xr^{1/3}] \leq
\P[X_{r,0} - 4r < -\frac{1}{2}xr^{1/3}] + \P[X_r - 4r > \frac{1}{2}xr^{1/3}]
\leq Ce^{-cx},
\end{align*}
and (i) follows. For (ii) notice that
\begin{align*}
\P[|X_{r,j}-4r|\geq xr^{1/3}] \leq
\P[X_{r,0} - 4r < -xr^{1/3}] + \P[X_r - 4r > xr^{1/3}]\leq Ce^{-cx}.
\end{align*}
This implies that $|\E X_{r,j}-4r|\leq Cr^{1/3}$ and (ii) also follows.
\end{proof}

We need one more result for the proofs of Lemmas \ref{l:bjbound} and \ref{l:cjbound}. Recalling that $X_n^r$ denotes the line to point last passage time from $\L_r$ to $\bn$,
 we next seek to obtain  concentration estimates for $X_{n}-X_{n,\bar{j}}$,
$X_n^r-X_{n,\bar{j}}$,
and $X_{n}-X_n^{r}$,
conditional on $\cf_r$ on the event $A_j$. 
Unlike Lemma \ref{l:rjconc}, there are extra suboptimal powers of $\frac{n}{r}$ on the right hand side of these estimates which we have not attempted to remove.
These however shall not introduce any difficulties as they will only be used in the proofs of Lemma \ref{l:bjbound} and \ref{l:cjbound} to control the contribution to the covariances from the tails, by integrating the effect of the deep end of the tails where the exponential terms would be small enough to kill the polynomial pre-factors. 
\begin{lemma}
\label{l:njconc}
In the above set-up, there exists $C,c>0$ such that for all $j<j_0$ and for all $x>0$ sufficiently large we have the following:
\begin{enumerate}
    \item[(i)] $\P[\don_{B_j} (X_n-X_{n,\bar{j}}) \geq xr^{1/3} ] \leq Cn^{2/3}r^{-2/3}\exp(-cx)$.
    \item[(ii)] $\P\left[\don_{B_j}\left|X_{n,\bar{j}}-\E \left[X_{n,\bar{j}}\mid \cf_r\right] \right| \geq xr^{1/3}\right] \leq Cn^{4/3}r^{-4/3}\exp(-cx)$.
    \item[(iii)] $\P\left[|X_{n}-X_n^{r}- 4r | \geq xr^{1/3}\right] \leq Cn^{2/3}r^{-2/3}\exp(-cx)$.
\end{enumerate}
\end{lemma}

\begin{proof}
We first state our proof strategy.
We will bound the upper and lower tails of
$$
X_{n}-X_n^{r}- 4r, \;\; \don_{B_j}(X_{n,\bar{j}}-X_n^r- 4r),
$$
conditioned on $\cf_r$.
In Step 1 below, we bound lower tails, and the upper tails would be bounded in Steps 2 and 3.
In Step 2 we bound the upper tails conditioned on some ``good'' events, where the point to line cost $X_n^r$ is not too small, and the profile from $\L_r$ to $\bn$ has parabolic decay when moving away from the center.
In Step 3 we will control the probability where these events fail.
In Step 4 we will show how we get our conclusions from the obtained tail estimates.

Throughout the proof we take $C, c>0$ to be some large and small absolute constants, and we assume that $x$ is sufficiently large enough.

\noindent
\textbf{Step 1.}
We first show that,
\begin{equation}    \label{e:njconc4}
\don_{B_j}\P\left[X_{n, \bar{j}}-X^r_{n} - 4r < - xr^{1/3}\mid \cf_r\right] \leq C \exp(-cx),
\end{equation}
and
\begin{equation}    \label{e:njconc5}
\P\left[X_{n}-X^r_{n} - 4r \leq - xr^{1/3}\mid \cf_r\right] \leq C \exp(-cx).
\end{equation}

Recall from \eqref{argmax1}, that $u_{\max}$ is the starting point of the line-to-point geodesic $\Gamma^r_{\bn}.$
Then $$X_{u_{\max}, \bar{j}} \leq X_{n, \bar{j}}-X^r_{n} \text{ and }X_{u_{\max}} \leq X_{n}-X^r_{n}.$$
Thus
\begin{align*}
\P\left[X_{n}-X^r_{n} - 4r \leq - xr^{1/3}\mid \cf_r\right]
\leq
\P\left[X_{u_{\max}} - 4r \leq - xr^{1/3}\mid \cf_r\right]\\
=
\P\left[X_{r} - 4r \leq - xr^{1/3}\right]
\leq
\P\left[T_{\mathbf{0}, \br} - 4r \leq - xr^{1/3}\right]
\leq C\exp(-cx)
\end{align*}
where the equality is due to translation invariance, 
and the last inequality is by Theorem \ref{t:onepoint}.
Also, if we let $X'$ denote the maximum passage time among all paths between $\L_0$ and $u_{\max}$ that do not exit $U_j+u_{\max}-\br$, then
\begin{align*}
\don_{B_j}\P\left[X_{n, \bar{j}}-X^r_{n} - 4r \leq - xr^{1/3}\mid \cf_r\right]
&\leq
\don_{B_j}\P\left[X_{u_{\max}, \bar{j}} - 4r \leq - xr^{1/3}\mid \cf_r\right]
\\
&\leq
\don_{B_j}\P\left[X' - 4r \leq - xr^{1/3}\mid \cf_r\right],
\end{align*}
since on $B_j$, $u_{\max} \not\in \L_{r,(\log(j+1))^{10}r^{2/3}}$.
By translation invariance, this probability equals
$$
\don_{B_j}\P\left[X_{r, j} - 4r \leq - xr^{1/3}\right]
\leq
\P\left[X_{r, 0} - 4r \leq - xr^{1/3}\right]
\leq C\exp(-cx),
$$
and the last inequality is proved at the beginning of the proof of Lemma \ref{l:rjconc}.  

We next move on to the upper tails. \\
\noindent
\textbf{Step 2.} As indicated in the proof strategy, we first prove the following upper tail estimate  conditioned on the good event $E_x\cap F_x,$  where $E_x$ is the event
$$
T_{\br+(u_1,-u_1),\bn} \leq 4(n-r) - \frac{u_1^2}{10n},\; \forall u_1\in \Z, |u_1|>xn^{2/3},
$$
and $F_x$ is the event
$$
X^r_n > 4(n-r) - xn^{1/3}.
$$
We will show 
\begin{equation}    \label{e:njconc6}
\don_{E_x\cap F_x}\P\left[X_{n}-X^r_{n} - 4r \geq xr^{1/3}\mid \cf_r\right] \leq Cn^{2/3}r^{-2/3} \exp(-cx),
\end{equation}
Recall that $u_0$ is the intersection of $\Gamma_{\bn}$ with $\L_r$, and that $u_0 = \br+(u_{01},-u_{01})$ for some $u_{01} \in \Z$.
Now assume the event $E_x\cap F_x$, and $X_{n}-X^r_{n} - 4r \geq xr^{1/3}$. We consider two cases:
\begin{itemize}
\item 
If $|u_{01}| < xn^{2/3}$,
since $X_{u_0} \geq X_{n}-X^r_{n} \geq 4r+xr^{1/3}$ we have
$$
\max_{u_1\in \Z, |u_1|\leq xn^{2/3}} X_{\br+(u_1,-u_1)} \geq 4r+xr^{1/3}.
$$
\item
If $|u_{01}| \geq xn^{2/3}$,
we have
$$
X_{u_0} \geq X^{r}_n+4r+xr^{1/3}-T_{u_0,\bn} > 4n-xn^{1/3}-\left(4(n-r)-\frac{u_{01}^2}{10n}\right)
=4r-xn^{1/3}+\frac{u_{01}^2}{10n},
$$
and hence,
$
\displaystyle{\max_{u_1\in\Z, |u_1|> xn^{2/3}} X_{\br+(u_1,-u_1)}-\frac{u_{1}^2}{10n} \geq 4r-xn^{1/3}.}
$
\end{itemize}
In conclusion, we have 
\begin{align*}
\don_{E_x\cap F_x}\P\left[X_{n}-X^r_{n} - 4r \geq xr^{1/3}\mid \cf_r\right]
\leq
\P\left[\max_{u_1\in \Z, |u_1|\leq xn^{2/3}} X_{\br+(u_1,-u_1)} \geq 4r+xr^{1/3}
\right]\\
+
\P\left[\max_{u_1\in\Z, |u_1|> xn^{2/3}} X_{\br+(u_1,-u_1)}-\frac{u_{1}^2}{10n} \geq 4r-xn^{1/3}
\right] .
\end{align*}
To estimate the right hand side, we apply Proposition \ref{p:l2l} to $\L_{r,r^{2/3}}+(i,-i)$ for $i \in 2\lfloor r^{2/3} \rfloor\Z$ to get the bound
$$
C'\lceil x\rceil n^{2/3}r^{-2/3}\exp(-c'x) + \sum_{i=1}^{\infty} C'\lceil x\rceil n^{2/3}r^{-2/3}\exp\left(-c'(i^2x^2/10-x)n^{1/3}r^{-1/3}\right)
$$
for constants $C', c'$. When $x$ is large enough this implies \eqref{e:njconc6}.

\noindent
\textbf{Step 3.}
Having proved an upper tail estimate conditioned on $E_x\cap F_x$, we next show that,
\begin{equation}\label{e:njconc7}
\P[E_x\cap F_x] \geq 1 - C\exp(-cx^2).
\end{equation}
By Theorem \ref{t:onepoint}, we have $$\P[F_x^c] \leq \P[T_{\br, \bn} \leq 4(n-r) - xn^{1/3}]\leq C'\exp(-c'x^3),$$ for some $c', C'>0$.
To bound the probability of $E_x^c$ notice that on this event, there is some $i \in 2\lfloor n^{2/3}\rfloor \Z$, $|i| > (x-1)n^{2/3}$ (where the upper bound is due to the geometric constraints of directed paths), such that
$$
\max_{u \in \L_{r, n^{2/3}}+(i,-i)} T_{u,\bn} > 4(n-r) - \frac{(|i|+n^{2/3})^2}{10n}.
$$
By Theorem \ref{t:onepoint} and Theorem \ref{t:supinf}, there are $C'', c'' > 0$ such that
$$
\P\left[\max_{u \in \L_{r, n^{2/3}}+(i,-i)} T_{u,\bn} > 4(n-r) - \frac{(|i|+n^{2/3})^2}{10n}\right] \leq C''\exp\left(-c''(|i|-n^{2/3})^2n^{-4/3}\right)
$$
for $(x-1)n^{2/3} < |i| < 0.99(n-r)$; and
\begin{align*}
&\P\left[\max_{u \in \L_{r, n^{2/3}}+(i,-i)} T_{u,\bn} > 4(n-r) - \frac{(|i|+n^{2/3})^2}{10n}\right]
\\
&\leq
\P\left[\max_{u \in \L_{r, n^{2/3}}+(i,-i)} T_{u,(\lfloor 1.01n\rfloor, \lfloor 1.01n\rfloor)} > 4(n-r) - \frac{(|i|+n^{2/3})^2}{10n}\right]
\leq
C'\exp\left(-c''(|i|-n^{2/3})^2n^{-4/3}\right)
\end{align*}
for $0.99(n-r) < |i| < n-r+n^{2/3}$.
By summing over all $i \in 2\lfloor n^{2/3}\rfloor \Z$, $(x-1)n^{2/3} < |i| < n-r+n^{2/3}$, we get the desired bound which implies \eqref{e:njconc7}.

\noindent
\textbf{Step 4.}
Finally, we deduce our conclusions from the above tail estimates.
Adding up the estimates \eqref{e:njconc5}, \eqref{e:njconc6} and \eqref{e:njconc7} immediately imply conclusion (iii).
For (ii), since $X_n^r$ is measurable with respect to $\cf_r$, we have
\begin{align}    \label{e:njconc11}
&\P\left[\don_{E_x\cap F_x \cap B_j}\left|X_{n,\bar{j}}-\E \left[X_{n,\bar{j}}\mid \cf_r\right] \right| \geq xr^{1/3}\right]
\\
\nonumber
&=
\P\left[\don_{E_x\cap F_x \cap B_j}\left|X_{n,\bar{j}}-X_n^{r}- 4r - \don_{E_x\cap F_x \cap B_j}\E \left[X_{n,\bar{j}}-X_n^{r}- 4r\mid \cf_r\right] \right| \geq xr^{1/3}\right]
\\
\nonumber
&\leq
\P\left[\don_{ B_j}\left|X_{n,\bar{j}}-X_n^{r}- 4r  \right| \geq xr^{1/3} - \don_{E_x\cap F_x \cap B_j}\left|\E \left[X_{n,\bar{j}}-X_n^{r}- 4r\mid \cf_r\right]\right|\right].
\end{align}
The next obvious step is to  control the tail of $\left|X_{n,\bar{j}}-X_n^{r}- 4r  \right|$, as well as bound the conditional expectation $\don_{E_x\cap F_x \cap B_j}\left|\E \left[X_{n,\bar{j}}-X_n^{r}- 4r\mid \cf_r\right]\right|$.
Since $X_{n,\bar{j}} \leq X_n$, \eqref{e:njconc6} implies
\begin{equation}    \label{e:njconc9}
\don_{E_x\cap F_x}\P\left[X_{n,\bar{j}}-X^r_{n} - 4r \geq xr^{1/3}\mid \cf_r\right] \leq Cn^{2/3}r^{-2/3} \exp(-cx).
\end{equation}
By \eqref{e:njconc4}, \eqref{e:njconc7}, and \eqref{e:njconc9} we have
\begin{equation}    \label{e:njconc3}
\P\left[\don_{B_j}|X_{n,\bar{j}}-X_n^{r}- 4r | \geq xr^{1/3}\right] \leq Cn^{2/3}r^{-2/3}\exp(-cx).
\end{equation}
Before proceeding further notice that by triangle inequality and simple union bound, conclusion (iii) and \eqref{e:njconc3} imply conclusion (i). Next, by \eqref{e:njconc4} and \eqref{e:njconc9} we have
\begin{align}    \label{e:njconc10}
&\don_{E_x\cap F_x \cap B_j}\left|\E\left[X_{n,\bar{j}}-X^r_{n} - 4r\mid \cf_r\right]\right|
\leq
\don_{E_x\cap F_x \cap B_j}\E\left[|X_{n,\bar{j}}-X^r_{n} - 4r| \mid \cf_r\right],
\\
\nonumber
& \leq 
r^{1/3}\int_{\R_+} \min\left\{1, Cn^{2/3}r^{-2/3} \exp(-cx) \right\}
+ C\exp(-cx) dx,\\
\nonumber
&=
c^{-1}\left(1+\log\left(Cn^{2/3}r^{-2/3}\right)\right)r^{1/3} + Cc^{-1}r^{1/3}.
\end{align}
Then by \eqref{e:njconc10}, and replacing $x$ by $x-(c^{-1}\left(1+\log\left(Cn^{2/3}r^{-2/3}\right)\right) + Cc^{-1})$ in \eqref{e:njconc3}, the last line of \eqref{e:njconc11} is bounded by
\begin{align*}
&Cn^{2/3}r^{-2/3}\exp(-cx)\exp\left(\left(1+\log\left(Cn^{2/3}r^{-2/3}\right)\right) + C\right)
\\
&\leq C^2n^{4/3}r^{-4/3}\exp(-cx)\exp\left(1 + C\right),
\end{align*}
yielding
$$
\P\left[\don_{B_j}\left|X_{n,\bar{j}}-\E \left[X_{n,\bar{j}}\mid \cf_r\right] \right| \geq xr^{1/3}\right]
\leq 
C^2n^{4/3}r^{-4/3}\exp(-cx)\exp\left(1 + C\right)+ \P[E_x\cap F_x],
$$
which with \eqref{e:njconc7} implies conclusion (ii).
\end{proof}

\subsection{Proofs of Lemma \ref{l:bjbound} and Lemma \ref{l:cjbound}}
We are finally ready to complete the proofs of the lemmas  using the estimates obtained in the previous subsection. As in some of the earlier proofs, the arguments rely on some manipulation of the already obtained estimates and are slightly technical.
\begin{proof}[Proof of Lemma \ref{l:bjbound}]
Recall the set-up of Lemma \ref{l:bjbound}. Let $j<j_0$ be fixed. Writing $$X_r=X_{r,j}+(X_{r}-X_{r,j}), \text{and, } X_{n}=X_{n,\bar{j}}+(X_{n}-X_{n,\bar{j}}),$$ we get 
$$\Cov_{\cf_r}((X_r,X_n))=\Cov_{\cf_r}(X_{r}-X_{r,j},X_{n}-X_{n,\bar{j}})+\Cov_{\cf_r}(X_{r,j},X_{n}-X_{n,\bar{j}})+\Cov_{\cf_r}(X_{r}-X_{r,j},X_{n,\bar{j}}) $$
where we have used the fact that by definition $X_{r,j}$ and $X_{n,\bar{j}}$ are conditionally (and also unconditionally) independent as they are functions of weights of disjoint sets of vertices. On $B_j$ we shall bound expectations of each of the tree terms in the RHS above separately. 

\noindent
For the first term observe that $X_{r}-X_{r,j}$ and $X_{n}-X_{n,\bar{j}}$ are both nonnegative and hence we get 
\begin{equation}
\label{e:1st}
\E[\don_{B_j}\Cov_{\cf_r}(X_{r}-X_{r,j},X_{n}-X_{n,\bar{j}})] \leq \E[\don_{B_j}(X_{r}-X_{r,j})(X_{n}-X_{n,\bar{j}})].
\end{equation}
For the second term above notice that $\Cov_{\cf_r}(X_{r,j},X_{n}-X_{n,\bar{j}})=\E[(X_{r,j}-\E X_{r,j})(X_{n}-X_{n,\bar{j}})\mid \cf_r]$ and hence we get 
\begin{equation}
\label{e:2nd}
\E[\don_{B_j}\Cov_{\cf_r}(X_{r,j},X_{n}-X_{n,\bar{j}})] \leq \E[\don_{B_j}|X_{r,j}-\E X_{r,j}||X_{n}-X_{n,\bar{j}}|].
\end{equation}
For the third term above, arguing as in the previous case we obtain 
\begin{equation}
\label{e:3rd}
\E[\don_{B_j}\Cov_{\cf_r}(X_{r}-X_{r,j},X_{n,\bar{j}})] \leq \E\left[\don_{B_j}|X_{r}-X_{r,j}|\left|X_{n,\bar{j}}-\E \left[X_{n,\bar{j}}\mid \cf_r\right]\right|\right].
\end{equation}
We can write these expectations as
$$
\int_{\R_+}
r^{2/3}\P[\don_{B_j}(X_{r}-X_{r,j})(X_{n}-X_{n,\bar{j}}) \geq xr^{2/3}] dx,
$$
$$
\int_{\R_+}
r^{2/3}\P[\don_{B_j}|X_{r,j}-\E X_{r,j}||X_{n}-X_{n,\bar{j}}| \geq xr^{2/3}] dx,
$$
$$
\int_{\R_+}
r^{2/3}\P[\don_{B_j}|X_{r}-X_{r,j}|\left|X_{n,\bar{j}}-\E \left[X_{n,\bar{j}}\mid \cf_r\right]\right| \geq xr^{2/3}] dx,
$$
respectively.
It remains to bound the tail probabilities.
Take $c', C'>0$ be small and large enough constants, respectively.
By Lemma \ref{l:rjconc} (i) and Lemma \ref{l:njconc} (i), for $x$ large enough we have
\begin{align*}
\P[\don_{B_j}(X_{r}-X_{r,j})(X_{n}-X_{n,\bar{j}}) \geq xr^{2/3}]
&\leq
\P[X_{r}-X_{r,j} \geq x^{1/2}r^{1/3}]
+
\P[\don_{B_j}(X_{n}-X_{n,\bar{j}}) \geq x^{1/2}r^{1/3}]\\
&\leq C'n^{2/3}r^{-2/3}\exp(-c'x^{1/2}).
\end{align*}
By Lemma \ref{l:rjconc} (ii) and Lemma \ref{l:njconc} (i), for $x$ large enough
\begin{align*}
&\P[\don_{B_j}|X_{r,j}-\E X_{r,j}||X_{n}-X_{n,\bar{j}}| \geq xr^{2/3}]\\
&\leq
\P[|X_{r,j}-\E X_{r,j}| \geq x^{1/2}r^{1/3}]
+
\P[\don_{B_j}|X_{n}-X_{n,\bar{j}}| \geq x^{1/2}r^{1/3}]
\leq C'n^{2/3}r^{-2/3}\exp(-c'x^{1/2}).
\end{align*}
By Lemma \ref{l:rjconc} (i), and since $X_{r}-X_{r,j} \geq 0$, we have $0 \leq \E(X_{r}-X_{r,j}) \leq C''r^{1/3}$, for some constant $C''$.
Then by Lemma \ref{l:rjconc} (i) and Lemma \ref{l:njconc} (ii), for $x$ large enough we have,
\begin{align*}
&\P[\don_{B_j}|X_{r}-X_{r,j}|\left|X_{n,\bar{j}}-\E \left[X_{n,\bar{j}}\mid \cf_r\right]\right| \geq xr^{2/3}]\\
&\leq
\P[|X_{r}-X_{r,j}| \geq x^{1/2}r^{1/3}]
+
\P[\don_{B_j}\left|X_{n,\bar{j}}-\E \left[X_{n,\bar{j}}\mid \cf_r\right]\right| \geq x^{1/2}r^{1/3}] \\
&\leq C'n^{4/3}r^{-4/3}\exp(-c'x^{1/2}).
\end{align*}
We also have that, by Lemma \ref{l:transversal1} and Lemma \ref{l:transversal2},
\begin{align*}
\P[\don_{B_j}(X_{r}-X_{r,j})(X_{n}-X_{n,\bar{j}}) \geq xr^{2/3}]
\leq
\P[B_j, X_{r,j}\neq X_r]
\leq
C'\exp\left(-c'(\log(j+2))^{2}\right)
\P[B_j],
\end{align*}
\begin{align*}
\P[\don_{B_j}|X_{r,j}-\E X_{r,j}||X_{n}-X_{n,\bar{j}}| \geq xr^{2/3}]
\leq
\P[B_j, X_{n}\neq X_{n,\bar{j}}]
\leq
C'\exp\left(-c'(\log(j+2))^{2}\right)
\P[B_j],
\end{align*}
\begin{align*}
\P[\don_{B_j}|X_{r}-X_{r,j}|\left|X_{n,\bar{j}}-\E \left[X_{n,\bar{j}}\mid \cf_r\right]\right| \geq xr^{2/3}]
\leq
\P[B_j, X_{r,j}\neq X_r]
\leq
C'\exp\left(-c'(\log(j+2))^{2}\right)
\P[B_j].
\end{align*}

Take $x_c \in \R_+$ such that
$C'n^{4/3}r^{-4/3}\exp(-c'x_c^{1/2})=C' \exp\left(-c'(\log(j+2))^{2}\right)\P[B_j]$; i.e. 
\begin{equation}  \label{e:xc}
x_c = \left((\log(j+2))^{2}-\frac{1}{c'}\log(\P[B_j]) + \frac{4}{3c'}\log(n/r)\right)^2,
\end{equation}
which can be large enough if $c'$ is small enough.
Thus each of \eqref{e:1st}, \eqref{e:2nd}, and \eqref{e:3rd} is bounded by
\begin{align*}
C'\exp\left(-c'(\log(j+2))^{2}\right)\P[B_j]x_c r^{2/3} + \int_{x>x_c}
C'n^{4/3}r^{-2/3}\exp(-c'x^{1/2}) dx
\\
\leq
C'\exp\left(-c'(\log(j+2))^{2}\right)\P[B_j]x_c r^{2/3} + 
C'n^{4/3}r^{-2/3}x_c
\sum_{i=1}^{\infty}\exp(-c'i^{1/2}x_c^{1/2}).
\end{align*}
Since $c'x_c^{1/2} \geq c'(\log(j+2))^{2} \geq c'(\log(2))^{2}$, it is bounded away from zero, so we further bound the above by
\begin{align*}
C'\exp\left(-c'(\log(j+2))^{2}\right)\P[B_j]x_c r^{2/3} + 
C''n^{4/3}r^{-2/3}\exp(-c'x_c^{1/2})x_c \\
\leq
2C''\exp\left(-c'(\log(j+2))^{2}\right)\P[B_j]x_c r^{2/3}
\end{align*}
for some large $C'' > 0$.
By plugging \eqref{e:xc} into this, we can bound each of \eqref{e:1st}, \eqref{e:2nd}, and \eqref{e:3rd} by
$$
2C''\left(\frac{4}{3c'}\right)^4\exp\left(-c'(\log(j+2))^{2}\right)r^{2/3}\P[B_j]\left((\log(j+2))^{2}-\log(\P[B_j]) + \log(n/r)\right)^2,
$$
and our conclusion follows.
\end{proof}

\medskip

\begin{proof}[Proof of Lemma \ref{l:cjbound}]
Recall the set-up of Lemma \ref{l:cjbound}. Let $j<j_0$ be fixed. Arguing as in the proof of Lemma \ref{l:bjbound} we get that 
$$\E[\don_{C_j}\Cov_{\cf_r}(X_r,X_n)]=\E[\don_{C_j}\Cov_{\cf_r}(X_r,X_n-X_n^r)]\leq \E[\don_{C_j}|X_r-\E X_r||X_n-X_n^r-4r|].$$
Take $c', C'>0$ be small and large enough constants, respectively.
For $x>0$ large enough, by Theorem \ref{t:onepoint},
$$
\P[|X_r-\E X_r| \geq xr^{1/3}] \leq C'\exp(-c'x^{3/2}),
$$
and by Lemma \ref{l:njconc} (iii), we have
$$
\P[|X_n-X_n^r-4r| \geq xr^{1/3}] \leq C'n^{2/3}r^{-2/3}\exp(-c'x).
$$
Then for $x$ large enough we have
\begin{align*}
&\P[\don_{C_j}|X_r-\E X_r||X_n-X_n^r-4r| \geq xr^{2/3}]\\
&\leq
\P[|X_r-\E X_r| \geq x^{1/2}r^{1/3}]
+
\P[|X_n-X_n^r-4r| \geq x^{1/2}r^{1/3}]
\leq
2C'n^{2/3}r^{-2/3}\exp(-c'x^{1/2}).
\end{align*}
We also have that
\begin{equation*}
\P[\don_{C_j}|X_r-\E X_r||X_n-X_n^r-4r| \geq xr^{2/3}]
\leq
\P[C_j].
\end{equation*}
Take $x_c \in \R_+$ such that $2C'n^{2/3}r^{-2/3}\exp(-c'x_c^{1/2}) = \P[C_j]$; i.e. $$x_c = \left(-\frac{\log(\P[C_j]) - \log(2C') - \frac{2}{3}\log(n/r)}{c'}\right)^2,$$ which is large enough when $c'$ is small enough.
Then we have
\begin{align*}
&\E[\don_{C_j}|X_r-\E X_r||X_n-X_n^r-4r|]
=
\int_{\R_+} r^{2/3}\P[\don_{C_j}|X_r-\E X_r||X_n-X_n^r-4r| > xr^{2/3}]dx,
\\
&\leq
r^{2/3}\P[C_j]x_c
+
\int_{x>x_c} 2C'n^{2/3}\exp(-c'x^{1/2})dx
\leq
r^{2/3}\P[C_j]x_c
+
2C'n^{2/3}x_c\sum_{i=1}^{\infty}\exp(-c'i^{1/2}x_c^{1/2}).
\end{align*}
Since $c'x_c^{1/2}\geq \log(2C')$ is bounded away from zero, the above is bounded by
\begin{align*}
r^{2/3}\P[C_j]x_c
+
2C'n^{2/3}x_c\exp(-c'x_c^{1/2})&=
2r^{2/3}\P[C_j]x_c
\\
&=
2r^{2/3}\P[C_j]\left(-\frac{\log(\P[C_j]) - \log(2C') - \frac{2}{3}\log(n/r)}{c'}\right)^2,
\end{align*}
and our conclusion follows.
\end{proof}

We end the main body of the paper by providing the proof of Theorem \ref{t:lower}.
\section{Covariance Lower Bound: Proof of Theorem \ref{t:lower}} \label{s:lower}
We shall work with the same set-up as in the previous section; $\delta\in (0,\frac{1}{2})$ and $\delta n< r< \frac{n}{2}$ sufficiently large will be fixed, and all constants will be independent of $\delta$. We shall use the same notation from the previous section wherever applicable. {As mentioned in Section \ref{iop}, the argument for the lower bound develops upon the ideas in the proof in the narrow wedge case \cite[Theorem 2 (i)]{BG18}. However our set up is substantially more complicated and significant new ingredients are required including a variant of Proposition \ref{p:maxdecay}}. Recall that in the previous section we conditioned on the $\sigma$-algebra $\cf_r$ generated by the configuration above the line $\L_r$. For Theorem \ref{t:lower} we shall need to {condition on a slightly larger} $\sigma$-algebra. For $\theta>0$, let the rectangle $R_{\theta}\subseteq \Z^2$ be defined as (see Figure \ref{f:defn}):
$$R_{\theta}:=\{u=(u_1,u_2)\in \Z^2: 0\leq u_1+u_2 < 2r,\; |u_1-u_2| \leq \theta r^{2/3}\}.$$
Let $\cg_{r,\theta}$ denote the $\sigma$ algebra generated by $\omega_{\theta}=\{\omega_{v}:v\in  \Z^2 \setminus R_{\theta}\}$. Clearly, $\cg_{r,\theta} \supseteq \cf_r$. Observe that for any event $\cE$ measurable with respect to $\cg_{r,\theta}$ there exists a subset $\cE'$ of $\R_{+}^{\Z^2\setminus R_{\theta}}$ (i.e., the configuration space restricted to the co-ordinates $\Z^2\setminus R_{\theta}$) such that $\cE=\cE' \times \R_{+}^{R_{\theta}}$. Often, in such instances, by an $\omega_{\theta}\in \cE$ we shall refer to its projection to $\cE'$. The following proposition will be our main technical tool in proving Theorem \ref{t:lower}. 

\begin{proposition}
\label{prop:event}
There exist absolute positive constants $\beta, \theta>0$ sufficiently small such that for any $\delta\in  (0,\frac{1}{2})$ and $n, r \in \Z_+$ with $\delta n< r <\frac{n}{2}$ and $n>n_0(\delta,\theta)$, there exists an event $\cE$ measurable with respect to $\cg_{r,\theta}$ with $\P(\cE)\geq \beta r^{2/3}n^{-2/3}$, and the following property: for all weight configuration $\omega_{\theta}\in \cE$ we have 
\begin{equation*}  
\label{eq:prop:event}
\Cov(X_r, X_n \mid \omega_{\theta}) > r^{2/3}.
\end{equation*}
\end{proposition}

Postponing the proof of Proposition \ref{prop:event} for the moment let us first show how this implies Theorem \ref{t:lower}. At this point we shall rely on the FKG inequality which will also be used in several occasions later in this section. Thus for concreteness we elaborate on how our setting enables the application. The statement of the inequality says, for a product measure on a countable product of the real line (i.e.\ for a collection of independent real valued random variables), increasing functions are positively correlated (see e.g.\ \cite[Lemma 2.1]{Kes03}); i.e., if there are two functions $f,g$ on the product space that are increasing (i.e.\ $f(\omega_1)\geq f(\omega_2)$ and $g(\omega_1)\geq g(\omega_2)$) whenever two sample points satisfy $\omega_1\geq \omega_2$ co-ordinate wise, we have $\E [fg] \ge \E [f]\E [g]$. Now consider the above set-up. Observe that for a fixed configuration $\omega_{\theta}$ on the vertices outside $R_{\theta}$, $X_{n}$ and $X_{r}$ are non-decreasing functions on the product space indexed by the vertices of $R_{\theta}.$ It follows that the conditional expectation $\E[X_{n}X_{r}\mid \cg_{r,\theta}]$ evaluated at the configuration $\omega_{\theta}$ satisfies 
$$ \E [X_{n} X_{r} \mid \cg_{r,\theta}] (\omega_{\theta}) \geq  \E [X_{n}\mid \cg_{r,\theta}](\omega_\theta) \E[X_{r} \mid \cg_{r,\theta}](\omega_{\theta}).$$
Observe that for $\omega_{\theta}\in \ce,$ Proposition \ref{prop:event} gives a sharper estimate which we shall use in the following proof. Observe also that as functions of $\omega_{\theta}$, $\E [X_{n}\mid \cg_{r,\theta}](\omega_{\theta})$ and  $\E[X_{r} \mid \cg_{r,\theta}](\omega_{\theta})$ are also increasing and hence we have 
$$\E\left( \E [X_{n}\mid \cg_{r,\theta}] \E[X_{r} \mid \cg_{r,\theta}]\right) \geq \E\left( \E [X_{n}\mid \cg_{r,\theta}]\right) \E\left( \E[X_{r} \mid \cg_{r,\theta}]\right)=\E[X_n]\E[X_{r}].$$
We are now ready to complete the proof of Theorem \ref{t:lower}.

\begin{proof}[Proof of Theorem \ref{t:lower}]
Let $\cE$ be the event satisfying the properties listed in Proposition \ref{prop:event}. Applying the FKG inequality twice, as explained above, together with Proposition \ref{prop:event} then implies

\begin{eqnarray*}
\E X_{n} X_{r} &=& \E \left( \E [X_{n} X_{r} \mid \cg_{r,\theta}]\right)\\
&=& \int_{\ce}  \E [X_{n} X_{r} \mid \cg_{r,\theta}] d\omega + \int_{\ce ^c}  \E [X_{n} X_{r} \mid \cg_{r,\theta}] d\omega\\
&\geq &  \int_{\ce}  \E [X_{n}\mid \cg_{r,\theta}] \E[X_{r} \mid \cg_{r,\theta}] d\omega  + \beta r^{4/3}n^{-2/3} + \int_{\ce ^c}  \E [X_{n}\mid \cg_{r,\theta}] \E[X_{r} \mid \cg_{r,\theta}] d\omega \\
&{=} & \E\left( \E [X_{n}\mid \cg_{r,\theta}] \E[X_{r} \mid \cg_{r,\theta}]\right)+\beta r^{4/3}n^{-2/3}\\
&\geq & \E[X_n]\E[X_{r}]+\beta r^{4/3}n^{-2/3};
\end{eqnarray*}
i.e., $\Cov(X_r,X_n)\geq \beta (r/n)^{4/3}n^{2/3}$, as desired. 
\end{proof}

\subsection{Construction of the event $\cE$}
\label{s:ce}
We construct the event $\cE$ in this subsection and show that it satisfies the required covariance lower bound. The probability lower bound on $\cE$ is established in the next subsection. There are two main components of the construction of $\cE$: two independent events $\cE_{\mathrm{dec}}$ ({\rm{dec}} stands for decay) and $\cE_{\mathrm{bar}}$ ({\rm{bar}} stands for barrier); the first of these is measurable with respect to the $\sigma$-algebra $\cf_r$ and the second depends only the weights below the line $\L_r$. We give precise definitions below.

\begin{figure}[htbp!]
\includegraphics[width=.4\textwidth]{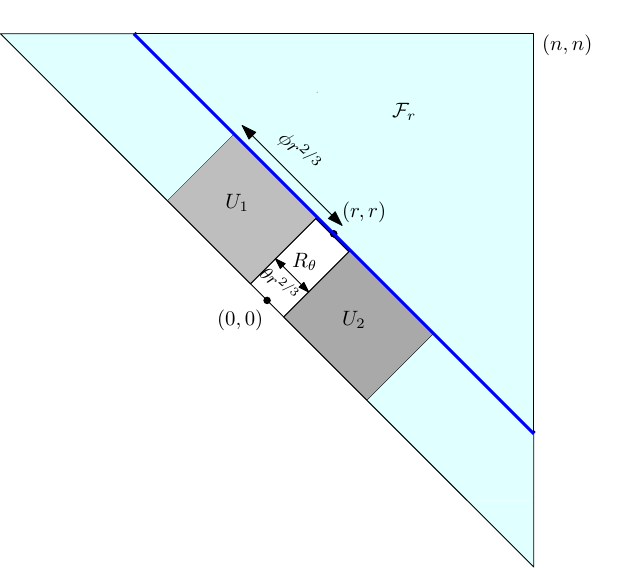} 
\caption{Conditioning on the complement of $R_{\theta}$; a thin rectangle of width $\theta r^{2/3}$ around the straight line joining $\mathbf{0}$ and $\mathbf{r}$ below and \emph{excluding} the line $x+y=2r.$ We condition on the line to point profile $\{T_{v,\mathbf{n}}\}_{v\in \mathbb{L}_r}$ to have its maximum close to $v=\mathbf{r}$ and have characteristic decay away from the maxima (this is the event $\ce_{\mathrm{dec}}$). We also condition on the region below $\mathbb{L}_r$ except for the region $R_\theta$ (the region marked in blue and gray) and in particular ask that the regions $U_1$ and $U_2$ act like barriers; any path going across this region is heavily penalised (this is the event $\ce_{\mathrm{bar}}$). The event $\ce$ is a subset of $\ce_{\mathrm{dec}}\cap \ce_{\mathrm{bar}}$. $\cf_r$ denotes the sigma algebra generated by all vertices \emph{including} and above the line $x+y=2r.$}
\label{f:defn}
\end{figure}

\noindent
\textbf{Definition of $\cE_{\mathrm{dec}}:$} {As in the setting of Proposition \ref{p:maxdecay}, we take $\tau = \frac{1}{4}$, and take $\alpha$ (a universal constant) as given by Proposition \ref{prop:lbae}}. Let $H'_0$ be the event where
\begin{equation}
\max_{|m| < r^{2/3}} T_{\br+(m,-m), \bn} =\max_{|m| < \theta r^{2/3}} T_{\br+(m,-m), \bn} {< T_{\br, \bn} + 2\alpha^{-1}r^{1/3}}. 
\end{equation}

For each $j\in \Z_+$, let $H'_j=H'_j(n,r)$ denote the event where
\begin{equation}
\max_{2^{j-1} r^{2/3} \leq |m| < 2^{j} r^{2/3}} T_{\br+(m,-m), \bn} < \max_{|m| < r^{2/3}} T_{\br+(m,-m), \bn} - 2\alpha \cdot 2^{j\left(\frac{1}{2}-\tau\right)}r^{1/3}.
\end{equation}
{We define 
\begin{equation}
\label{e:defdec}
\cE_{\mathrm{dec}}:=\bigcap_{j\in\Z_{\geq 0}}H'_j;
\end{equation} 
clearly $\cE_{\mathrm{dec}}$ is measurable with respect to $\cf_r$ and is a translate of the event $\cE_{n-r,r}$ defined in Section \ref{s:finite}.}
Note that the above constraints in the definition implies that the profile is maximized in the interval around $\br$ of size $\theta r^{2/3}$ with the value at $\br$ comparable to the value of the maximum along with an almost diffusive decay away from the interval of size $r^{2/3}.$

\noindent
\textbf{Definition of $\cE_{\mathrm{bar}}:$}
This event will depend on two absolute constants $\phi$ and $L$ which will be chosen sufficiently large later (depending on $\theta$). Let 
$$U_1=\{(u_1,u_2)\in \Z^2: 0\leq u_1+u_2<2r: \theta r^{2/3}< u_1-u_2\leq \phi r^{2/3}\};$$
$$U_2=\{(u_1,u_2)\in \Z^2: 0\leq u_1+u_2<2r: \theta r^{2/3}<u_2-u_1\leq \phi r^{2/3}\};$$
(see Figure \ref{f:defn}).
For any point $u=(u_1,u_2)\in \Z^2$, let 
\begin{equation}\label{sumcoordinates}
d(u):=u_1+u_2.
\end{equation}
 Also, for any region $U$, and points $u,v\in U$, let us denote, by $T_{u,v}^{U}$ to be the length of the longest path from $u$ to $v$ that does not exit $U$. We define $\cE_{\mathrm{bar}}$ to be the intersection of the events 
$$\cE_{1,\mathrm{bar}}:= \left\{ T^{U_1}_{u,u'}-\E T_{u,u'} \leq -Lr^{1/3} ~\forall u,u' \in  U_1 ~\text{with}~ |d(u)-d(u')|\geq \frac{r}{L}\right\}$$
and 
$$\cE_{2,\mathrm{bar}}:=\left\{T^{U_2}_{u,u'}-\E T_{u,u'} \leq -Lr^{1/3} ~\forall u,u' \in  U_2 ~\text{with}~ |d(u)-d(u')|\geq \frac{r}{L}\right\}.$$

We are now ready to define the event $\cE$. In what follows, the constants $C_1$ and $c_1$ will be chosen appropriately large and small respectively later (independent of $\theta$). By an abuse of notation we shall also denote by 
$X_{2\theta}$ ($X_{\theta}$, respectively) the weight of the best path {from $\mathbf{0}$ to $\mathbf{r}$ that does not exit $R_{2\theta}$ ($R_{\theta}$, respectively)}. This  local usage with the specific value of $\theta$ and $2\theta$, should not create any confusion with objects such as $X_r$ and $X_n$ defined earlier and used throughout the article including in this section. Now we define $\cE \subseteq \cE_{\mathrm{bar}}\cap \cE_{\mathrm{dec}}$ to be the event such that for all $\omega_{\theta}\in \cE$ we have 

\begin{enumerate}
\item[(i)] $\E[(X_r-X_{2\theta})^2\mid \omega_{\theta}] \leq 10r^{2/3}$.
\item[(ii)] $\E[(X_n-X_n^r-X_{2\theta})^2\mid \omega_{\theta}] \leq C_1r^{2/3}$.
\item[(iii)] $\Var (X_{2\theta}\mid \omega_{\theta})\geq c_1\theta^{-1/2}r^{2/3}$.
\end{enumerate}

Next we show that for $\cE$ defined as above we have the required covariance lower bound for all $\omega_{\theta}\in \cE$. 

\begin{proposition}
\label{p:covlb}
There exist choices of parameters $C_1$ and $c_1$ independent of $\theta$ such that for all $\theta$ sufficiently small and appropriate choices of parameters $\phi$ and $L$ depending on $\theta$ we have the following: for $\cE$ defined above and for all $\omega_{\theta}\in \cE$ that $\Cov(X_r,X_n\mid \omega_{\theta})>r^{2/3}$. 
\end{proposition}

\begin{proof}
Let us fix $\omega_{\theta}\in \ce$. Observe that $X_n^r$ is a deterministic function of $\omega_{\theta}$. Using this together with linearity of covariance and Cauchy-Schwarz inequality, we have for each $\omega_{\theta} \in \ce$
\begin{eqnarray*}
{\rm {Cov}} (X_r, X_n \mid \omega_{\theta}) &=& {\rm {Cov}} (X_r, X_n-X_n^r \mid \omega_{\theta})={\rm {Cov}} (X_r-X_{2\theta}+X_{2\theta}, X_n-X_n^r-X_{2\theta}+X_{2\theta} \mid \omega_{\theta})\\
&=& {\rm {Var}} (X_{2\theta}\mid \omega_{\theta})+ {\rm {Cov}}(X_{2\theta}, X_n-X_n^r-X_{2\theta} \mid \omega_{\theta})\\
&+& {\rm {Cov}}(X_r-X_{2\theta},X_{2\theta}\mid \omega_{\theta})+ {\rm {Cov}} (X_r-X_{2\theta},X_n-X_n^r-X_{2\theta}\mid \omega_{\theta})\\
&\geq & {\rm {Var}} (X_{2\theta}\mid \omega_{\theta})\\
&-&\sqrt{{\rm {Var}}(X_{2\theta}\mid \omega_{\theta})}\left(\sqrt{{\rm {Var}}(X_r-X_{2\theta}\mid \omega_{\theta})}+\sqrt{{\rm {Var}}(X_n-X_n^r-X_{2\theta}\mid \omega_{\theta})}\right)\\
&-& \sqrt{{\rm {Var}}(X_r-X_{2\theta}\mid \omega_{\theta})}\sqrt{{\rm {Var}}(X_n-X_n^r-X_{2\theta}\mid \omega_{\theta})}.
\end{eqnarray*}
Clearly the last term in RHS is further lower bounded by $-\sqrt{10 C_1}r^{2/3}$ (by definition of $\cE$), and using the definition of $\cE$ and choosing $\theta$ sufficiently small one can make the total contribution of the other terms lower bounded by $(\sqrt{10C_1}+1)r^{2/3}$, implying ${\rm {Cov}} (X_r, X_n \mid \omega_{\theta})>r^{2/3}$, as desired. 
\end{proof}

\subsection{Lower Bounding the Probability of $\cE$}
To complete the proof of Theorem \ref{t:lower}, it remains now to prove an appropriate lower bound on $\P(\cE)$, which is done in the next proposition. 

\begin{proposition}
\label{p:elb}
In the set-up of Proposition \ref{prop:event}, for $\cE$ as defined in the previous section, there exist choices of parameters $C_1$ and $c_1$ independent of $\theta$ such that for all $\theta$ sufficiently small and appropriate choices of parameters $\phi$ and $L$ depending on $\theta$ such that the following holds: there exists $\beta>0$ depending on all the parameters such that $\P(\cE)\geq \beta r^{2/3}n^{-2/3}$.
\end{proposition}

It is obvious that Proposition \ref{p:covlb} and \ref{p:elb} imply Proposition \ref{prop:event}. There are two key steps to the proof of Proposition \ref{p:elb}. First we show that $\cE_{\mathrm{bar}}\cap \cE_{\mathrm{dec}}$ has probability bounded below by $cr^{2/3}n^{-2/3}$ for some $c>0$ (depending on $\delta$ and $\theta$, but not on $r$ or $n$).  By the definition of $\cE$, the key step is to show that, for appropriate choices of the parameters and sufficiently small $\theta$, conditional on the event $\cE_{\mathrm{bar}}\cap \cE_{\mathrm{dec}}$, the three conditions defining $\cE$ are likely to occur. In particular, we shall show the following corresponding statements for suitable absolute constants $C_1$ and $c_1$:

\begin{equation}
\label{e:2nd1}
\P[\E[(X_r-X_{2\theta})^2\mid \omega_{\theta}] \leq 10r^{2/3} \mid \cE_{\mathrm{bar}}\cap \cE_{\mathrm{dec}}] > 0.9.
\end{equation}

\begin{equation}
\label{e:2nd2}
\P[\E[(X_n-X_n^r-X_{2\theta})^2\mid \omega_{\theta}] \leq C_1r^{2/3} \mid \cE_{\mathrm{bar}}\cap \cE_{\mathrm{dec}}] > 0.9.
\end{equation}

\begin{equation}
\label{e:varlb}
\P[\Var (X_{2\theta}\mid \omega_{\theta})\geq c_1\theta^{-1/2}r^{2/3} \mid \cE_{\mathrm{bar}}\cap \cE_{\mathrm{dec}}] > 0.9.
\end{equation}

We prove \eqref{e:2nd1} and \eqref{e:2nd2} in Lemma \ref{l:12proof} below, whereas the proof of \eqref{e:varlb} is contained in Lemma \ref{l:realvar} and Lemma \ref{l:e8}.

\subsubsection{Probability Lower Bounds for $\cE_{\mathrm{bar}}$ and  $\cE_{\mathrm{dec}}$}
We start with lower bounding $\P(\cE_{\mathrm{bar}}\cap \cE_{\mathrm{dec}})$. Observe that these two events depend on disjoint sets of vertex weights and hence are independent. So it suffices to lower bound $\P(\cE_{\mathrm{dec}})$ and $\P(\cE_{\mathrm{bar}})$ separately. We start with lower bound of $\P(\cE_{\mathrm{dec}})$. 

\begin{lemma}
\label{l:declb}
In the set-up of Proposition \ref{prop:event}, there exists an absolute constant $c>0,$ {such that for all $\theta$ sufficiently small}, $\P(\cE_{\mathrm{dec}})\geq c\theta r^{2/3}n^{-2/3}$.
\end{lemma}

\begin{proof}
Observe that in the setting of Proposition \ref{prop:event} we have $\delta n<r< \frac{n}{2}$ and hence $\delta(n-r) <r< n-r$. Observing that $\ce_{\mathrm{dec}}$ is a translate of the event $\ce_{n-r,r}$ as defined in Proposition \ref{p:maxdecay}, and using the same proposition we get 
$\P(\ce_{\mathrm{dec}})\geq c_0\theta r^{2/3}(n-r)^{-2/3}\geq c\theta r^{2/3}n^{-2/3}$, as desired.
\end{proof}

\begin{lemma} \label{l:barlb}
There exists a constant $\kappa=\kappa(L,\phi)>0$ (independent of $r$) such that $\P(\ce_{\mathrm{bar}})>\kappa$ for all $r$ sufficiently large. 
\end{lemma}

\begin{proof}
Recall that $\cE_{\mathrm{bar}}$ was defined to be the intersection of the events $\cE_{1,\mathrm{bar}}$ and $\cE_{2,\mathrm{bar}}$
where $\cE_{i,\mathrm{bar}}$ was the event that asked the passage times across the rectangle $U_{i}$ to be small. 
Observe that $\cE_{1,\mathrm{bar}}$ and $\cE_{2,\mathrm{bar}}$ are independent and by obvious symmetry these events have equal probabilities. Hence it suffices to only lower bound $\P(\cE_{1,\mathrm{bar}})$.  Observe that $\cE_{1,\mathrm{bar}}$ is the translate of the event $\mathcal{B}$ from Lemma \ref{l:barbasic} with $\Delta=\phi-\theta$. Hence it follows from translation invariance of the underlying field and Lemma \ref{l:barbasic} that 
$\P(\cE_{1,\mathrm{bar}})\geq \sqrt{\kappa}$ for some $\kappa=\kappa(L,\phi)>0$. This completes the proof of the lemma.
\end{proof}

\subsubsection{Lower Bound for $\P(\cE)$}
In this subsection we prove \eqref{e:2nd1} and \eqref{e:2nd2}, \eqref{e:varlb} and use those to complete the proof of Proposition \ref{p:elb}. The broad strategy for the proofs of \eqref{e:2nd1} and \eqref{e:2nd2}  is the same. We shall prove the following bounds on the  conditional expectations: $$\E[(X_r-X_{2\theta})^2\mid \cE_{\mathrm{bar}}\cap \cE_{\mathrm{dec}}], \text{and, } \E\left[(X_{n}-X_{n}^r-X_{2\theta})^2  \mid \cE_{\mathrm{bar}}\cap \cE_{\mathrm{dec}}\right] \leq C r^{2/3};$$ \eqref{e:2nd1} and \eqref{e:2nd2} will then follow from Markov's inequality. We first state these results. 

Observe first that $X_r-X_{2\theta}$ is independent of $\cf_r$ and hence $\cE_{\mathrm{dec}}$, so it suffices to control $$\E[(X_r-X_{2\theta})^2\mid \ce_{\mathrm{bar}}].$$ In fact we shall prove a slightly stronger result. Let 
\begin{equation}\label{x'}
X'=\sup_{u\in {\L_{r,\frac{1}{4}\phi r^{2/3}}} }X_{u},
\end{equation} that is the weight of the maximum weight point to line path from $\L_0$ with the other endpoint being on the line segment in $\L_{r}$ centered at $\mathbf{r}$ with length {$\frac{1}{2}\phi r^{2/3}$}.

\begin{lemma}
\label{l:2nd1}
For $\theta$ sufficiently small and appropriate choices of parameters $\phi$ and $L$ we have for all $r$ sufficiently large: $\E[(X'-X_{2\theta})^2\mid \ce_{\mathrm{bar}}]\leq r^{2/3}$. 
\end{lemma}

\begin{lemma}
\label{l:2nd2}
There exists and absolute constant $C>0$ such that for all $\theta$ sufficiently small and appropriate choices of parameters $\phi$ and $L$ we have for all $r$ sufficiently large: $$\E\left[(X_{n}-X_{n}^r-X_{2\theta})^2  \mid \cE_{\mathrm{bar}}\cap \cE_{\mathrm{dec}}\right] \leq C r^{2/3}.$$ 
\end{lemma}

Let us postpone the proofs of these two lemmas and first show how these results lead to simple proofs of \eqref{e:2nd1} and \eqref{e:2nd2} which are recorded in the next lemma. 

\begin{lemma}
\label{l:12proof}
There exist an absolute constant $C_1$ (independent of $\theta$) such that for all $\theta$ sufficiently small and appropriate choices of parameters $\phi$ and $L$ depending on $\theta$ we have 
\begin{enumerate}
\item[(i)] $\P[\E[(X_r-X_{2\theta})^2\mid \omega_{\theta}] \leq 10r^{2/3} \mid \cE_{\mathrm{bar}}\cap \cE_{\mathrm{dec}}] > 0.9$,
\item[(ii)] $\P[\E[(X_n-X_n^r-X_{2\theta})^2\mid \omega_{\theta}] \leq C_1r^{2/3} \mid \cE_{\mathrm{bar}}\cap \cE_{\mathrm{dec}}] > 0.9$.
\end{enumerate}
\end{lemma}

\begin{proof}
Note that (i) follows from Lemma \ref{l:2nd1} by observing that $X'\geq X_r \geq X_{2\theta}$; and $X_{2\theta}$ and $X_r$ and $\cE_{\mathrm{bar}}$ are all independent of $\cE_{\mathrm{dec}}$ and using Markov's inequality. Similarly, (ii) follows from Lemma \ref{l:2nd2} and Markov's inequality by setting $C_1=10C$ where $C$ is as in the statement of Lemma \ref{l:2nd2}.
\end{proof}

We next move towards  proving \eqref{e:varlb}. The strategy is to divide the rectangles $R_{\theta}$ (and $R_{2\theta}$) by parallel lines $\L_{i\theta^{3/2}r}=\{x+y=2i\theta^{3/2}r\}$. The rectangle formed by the intersection of $R_{\theta}$ (resp.\ $R_{2\theta}$) and the lines $\L_{i\theta^{3/2}r}$ and $\L_{(i+1)\theta^{3/2}r}$ will be denoted $R_{\theta}^{i}$ (resp.\ $R_{2\theta}^{i}$). We want to show that conditional on some positive probability subset of $\cE$, a linear fraction of these rectangles has a contribution of the order of $\theta^{1/2}r^{1/3}$ to the variance of $X_{2\theta}$.
For any $\omega_{\theta}$, weight configuration outside $R_{\theta}$, an index $i\in \{1,2,\ldots , \frac{1}{5}\theta^{-3/2}\}$ is called \textbf{good} if 
$$\P\left(\sup_{u\in R_{2\theta}^{5i}, v\in R_{2\theta}^{5i+4}} T_{u,v} - 2|d(u)-d(v)|\leq  C_2\theta^{1/2}r^{1/3}\mid \omega_{\theta}\right) \geq 0.99$$
for some fixed large constant $C_2$. Let $\ce_{*}$ denote the subset of all $\omega_{\theta}\in \ce_{\mathrm{dec}}\cap \ce_{\mathrm{bar}}$ such that the fraction of good indices is at least a half. We have the following results which together immediately imply \eqref{e:varlb}.

\begin{lemma}
\label{l:realvar}
For each $C_2>0$ sufficiently large, there exists $c_1>0$, such that for all $\theta$ sufficiently small and for $\omega_{\theta}\in \cE_{*}$ we have
$$\Var (X_{2\theta}\mid \omega_{\theta})\geq c_1\theta^{-1/2}r^{2/3}$$
for all $r$ sufficiently large (depending on $\theta$).
\end{lemma}

Proof of Lemma \ref{l:realvar} is postponed. The next lemma deals with the conditional probability of $\ce_{*}$.
\begin{lemma}
\label{l:e8}
If $C_2$ is sufficiently large, then $\P(\ce_*\mid \ce_{\mathrm{dec}}\cap \ce_{\mathrm{bar}})\geq 0.9$, {for all $\theta$ sufficiently small.}
\end{lemma}

\begin{proof} 
Let $N$ denote the fraction of indices $i$ for which we have 
$$\sup_{u\in R_{2\theta}^{5i}, v\in R_{2\theta}^{5i+4}} T_{u,v} - 2|d(u)-d(v)|  \geq C_2\theta^{1/2}r^{1/3}.$$
It follows from \eqref{e:mean} and Theorem \ref{t:supinf} (ii) that $\E[N]\leq 0.001$ for $C_2$ sufficiently large. Notice that $N$ is increasing in the configuration of weights outside $R_{\theta}$, and hence by the FKG inequality we have $$\E[N\mid \ce_{\mathrm{dec}}\cap \ce_{\mathrm{bar}}]\leq 0.001.$$ The conclusion follows from an application of Markov inequality.
\end{proof}

Let us now complete the proof of Proposition \ref{p:elb} using Lemmas \ref{l:12proof}, \ref{l:realvar} and \ref{l:e8}. 

\begin{proof}[Proof of Proposition \ref{p:elb}]
Observe first that Lemma \ref{l:realvar} and Lemma \ref{l:e8} implies \eqref{e:varlb}. This, together with \eqref{e:2nd1} and \eqref{e:2nd2} (established in Lemma \ref{l:12proof} for appropriate choices of parameters), a union bound and the definition of $\cE$ implies that $\P(\ce \mid \ce_{\mathrm{dec}}\cap \ce_{\mathrm{bar}})\geq 0.7$. This completes the proof by invoking Lemma \ref{l:declb} and Lemma \ref{l:barlb} and observing the by independence we get $$\P(\ce_{\mathrm{dec}}\cap \ce_{\mathrm{bar}})\geq cr^{2/3}n^{-2/3},$$ for some $c>0$ depending on the parameters.
\end{proof}

We shall prove Lemmas \ref{l:2nd1}, \ref{l:2nd2} and \ref{l:realvar} over the next three subsections, but before proceeding with these we need to specify the parameter choices.

\noindent
\textbf{Choice of Parameters:} {Throughout this section the constants $c, C_1,C_2$ etc.\ will be absolute constants independent of $\theta$. We shall finally choose $\theta$ sufficiently small depending on these and choose $\phi$ to be a large negative power of $\theta$, and $L$ to be a large  power of $\phi$. {For concreteness we take $\phi = \theta^{-30}$, and $L = \phi^{30}$ for the remaining of the text.} {The scales $r$ and $n$ are taken large enough depending on $\theta,\phi$ and $L$ and other constants.}}

\subsection{Proof of Lemma \ref{l:2nd1}}
Let us first explain the strategy for proving Lemma \ref{l:2nd1}. As any path with non-negligible intersection with the regions $U_1$ and $U_2$ (recall Figure \ref{f:defn})is heavily penalized under the event $\ce_{\mathrm{bar}}$, we shall show that conditional on $\ce_{\mathrm{bar}}$ the path attaining $X'$ will with probability close to 1 lie in $R_{2\theta}$ between the anti-diagonal lines $x+y=\theta^{3/2}r$ and $x+y=2r-\theta^{3/2}r$. This, in turn, would imply that on this large probability event, conditional on $\ce_{\mathrm{bar}}$, $X'-X_{2\theta}$ has small conditional second moment.  

We need the following result which shows that with large probability the best path between almost all pairs of points in $R_{\theta}$ restricted to not exit $R_{\theta}$ cannot be too small and none of the vertices in $R_{\theta}$ has very large weight (the later is needed to take care of certain boundary issues). Recall that for $u,v\in \R_{\theta}$, $T_{u,v}^{R_{\theta}}$ denotes the length of the best path from $u$ to $v$ constrained to not exit $R_{\theta}$. We have the following result. 

\begin{lemma}
\label{l:prep3-rtheta}
Let $\cA_{\theta, \mathrm{reg}}$ denote the event that for all $u,v\in R_{\theta}$  with $|d(u)-d(v)|\geq \frac{r}{L}$ we have $T_{u,v}^{R_{\theta}}\geq 2|d(u)-d(v)|-\theta^{-2}r^{1/3}$ and $\omega_v\leq r^{1/3}$ for each $v\in R_{\theta}$ ({recall that $\omega_v$ is the random environment}). Then for all $\theta$ sufficiently small and $r$ sufficiently large we have $$\P(\cA^{c}_{\theta, \mathrm{reg}})\leq {e^{-c\theta^{-1}}}$$
for some $c>0$. 
\end{lemma}

\begin{proof}
Recall the $\theta^{3/2}r\times \theta r^{2/3}$ rectangles $R_{\theta}^i$ that $R_{\theta}$ is divided into. Let $A_{i}$ denote the event that for all $u,v\in R_{\theta}^{i}\cup R_{\theta}^{(i+1)}$ with $|d(u)-d(v)|\geq \frac{r}{L}$ we have $T_{u,v}^{R_{\theta}}\geq 2|d(u)-d(v)|-\theta^{-1/2}r^{1/3}$. Now notice that for every $u,v\in R_{\theta}$  with $|d(u)-d(v)|\geq \frac{r}{L}$ there is is a sequence of points $u=w_0,w_1,\ldots w_{k}=v$ with $k\leq \theta^{-3/2}$ and such that {$w_i,w_{i+1}\in R_{\theta}^{j}\cup R_{\theta}^{(j+1)}$ for some $j$} and also $|d(w_{i+1})-d(w_i)|\geq \frac{r}{L}$. It follows that on $\bigcap A_{i}$ we have $T_{u,v}^{R_{\theta}}\geq 2|d(u)-d(v)|-\theta^{-2}r^{1/3}$. It follows from Theorem \ref{t:supinf}(iii) that {$\P(A_{i}^c)\leq Ce^{-c\theta^{-1}}$} for each $i$. Finally notice that the probability that $\omega_{v}\geq r^{1/3}$ for some $v\in R_{\theta}$ is at most $r^{5/3}e^{-r^{1/3}}$ and the result follows by a union bound and taking $\theta$ sufficiently small and $r$ sufficiently large depending on $\theta$.
\end{proof}

The next result will be the key to the proof of Lemma \ref{l:2nd1}. It shows that because of the existence of the barrier region, the best path from $\L_0$ to $\L_{r, {\frac{1}{4}\phi r^{2/3}}}$ (i.e., the path attaining weight $X'$) is extremely unlikely to exit $R_{2\theta}$ except near its endpoints.
\begin{lemma}
\label{l:key1}
Let $\Gamma'$ denote the path from $\L_0$ to $\L_{r, {\frac{1}{4}\phi r^{2/3}}}$ that attains the weight $X'$. Let $\ce'$ denote the event that there exist points $v_0, v_1\in R_{\theta} \cap \Gamma'$ such that $d(v_0), 2r-d(v_1) \in [(\frac{\theta}{2})^{3/2}r, (\frac{\theta}{2})^{3/2}r + \frac{r}{L}]$, and such that the restrictions of $\Gamma'$ between $v_0$ and $v_1$ is contained in $R_{2\theta}$. Then for all $\theta$ sufficiently small and $r$ sufficiently large, we get $$\P((\ce')^c\mid \ce_{\mathrm{bar}})\leq {{e^{-\theta^{-2/5}}}}.$$
\end{lemma}

\begin{proof}
This proof is somewhat long and involved, so let us begin by presenting a roadmap of the proof. We shall first show by using Proposition \ref{l: prep1-tf} that a path from $\L_0$ to $\L_{r, {\frac{1}{4}\phi r^{2/3}}}$ that exits $R_{\phi}$ is unlikely to attain $X'$. The next step would be to show that the barrier event makes it unlikely for the path to have a long excursion outside $R_{\theta}$ (the penalty of passing through the barrier region is higher than that of passing constrained within $R_{\theta}$), in particular we show that $\Gamma'$ is likely to intersect $R_{\theta}$ in every interval (in the diagonal direction) of length $\frac{r}{L}$. The final step involves showing that if $\Gamma'$ only has short excursions outside $R_{\theta}$ then it is unlikely to exit $R_{2\theta}$ (except possibly near the endpoints of $\Gamma'$) as it would create a large transversal fluctuation. Now we proceed to formalize each of the above steps.

\medskip

\begin{figure}[htbp!]
\includegraphics[width=.9\textwidth]{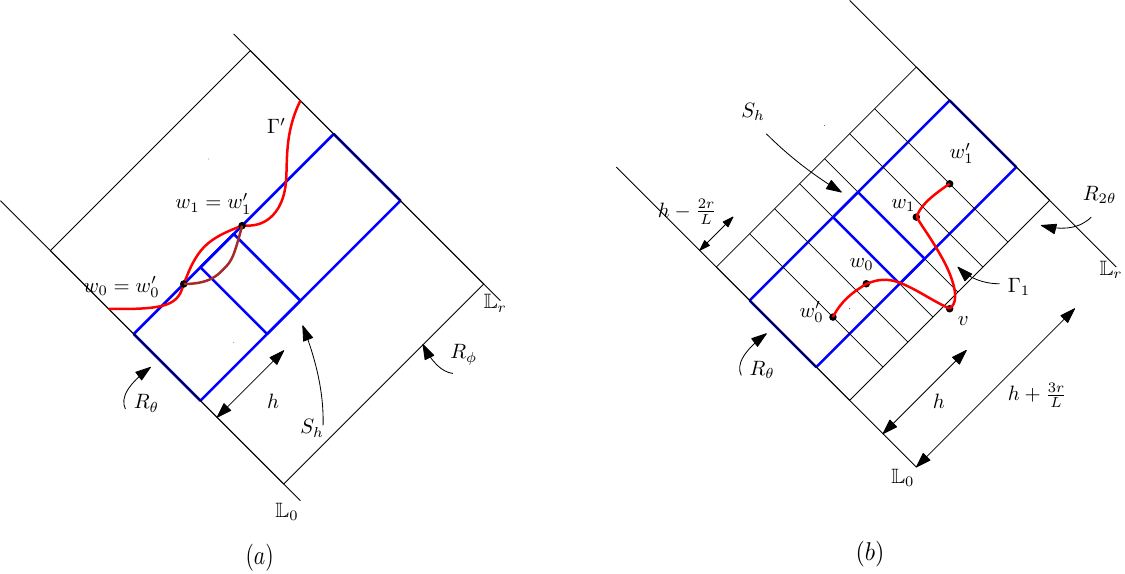} 
\caption{The figure illustrates the two key proof steps outlined above. (a) corresponds to the argument ruling out long excursions of $\Gamma'$ outside $R_{\theta}$.  Using (a) as input, (b) indicates the argument which shows that it is unlikely for the path to exit $R_{2\theta}$ except near the end points as the event creates uncharacteristic transversal fluctuations which can be ruled out with high probability. Note that both the figures use the same letters to denote different objects. However this will not create any confusion as further references to (a) and (b) made later in the relevant parts of the proof will clarify the context.}
\label{f:key0}
\end{figure}

Denote $\mathrm{TF}(\Gamma')$ to be the event where $\Gamma'$ exits $R_{\phi}$. Observe that on $\cA_{\theta, \mathrm{reg}}$ we have $\ell(\Gamma')\geq 4r-\theta^{-2}r^{1/3}$. Recall the event $\mathrm{LargeTF}$ defined in Proposition \ref{l: prep1-tf}. Using that {$\phi\gg \theta^{-1}$}, if $\theta$ is sufficiently small (and translation invariance) the probability that there exist a path from any $2r^{2/3}$ length sub-interval of $\L_{r, {\frac{1}{4}\phi r^{2/3}}}$ to $\L_0$ that exits $R_{\phi}$ and has length larger than $4r-\theta^{-2}r^{1/3}$ can be bounded by $\P(\mathrm{LargeTF}(\phi/4,r))$. Thus by Proposition \ref{l: prep1-tf}, and the FKG inequality (for $\ce_{\mathrm{bar}}$ is a negative event, while translations of $\mathrm{LargeTF}(\phi/4,r)$ are positive events, on the configuration on $\Z^2\setminus R_{\theta}$), we have
$$
\P(\mathrm{TF}(\Gamma')\cap \cA_{\theta, \mathrm{reg}}\mid \ce_{\mathrm{bar}}) \leq \lceil \phi/4 \rceil e^{-c_2(\phi/4)^3}.
$$
For each $h \in [0, 2r-\frac{r}{L}]$, let $$S_h:= \{v \in R_{\theta} : d(v) \in [h, h + \frac{r}{L}] \}.$$ 
Now we show that, on $\mathrm{TF}(\Gamma')^c\cap \cA_{\theta, \mathrm{reg}} \cap \ce_{\mathrm{bar}}$, we must have $\Gamma' \cap S_h \neq \emptyset$, for any $h \in [0, 2r-\frac{r}{L}]$.
Indeed, otherwise, we let $w_0$ be the vertex with the largest $d(w_0)$, such that $w_0 \in \Gamma'\cap R_{\theta}$, $d(w_0) < h$;
and if no such vertex exists, let $w_0$ be the intersection of $\Gamma'$ with $\L_0$.
Similarly, let $w_1$ be the vertex with the smallest $d(w_1)$, such that $w_1 \in \Gamma'\cap R_{\theta}$, $d(w_1) > h + \frac{r}{L}$; and if no such vertex exists, let $w_1$ be the intersection of $\Gamma'$ with $\L_r$.
Then on $\mathrm{TF}(\Gamma')^c$, the geodesic $\Gamma_{w_0, w_1}$, {which is a sub-segment of $\Gamma'$}, except possibly for the end points, lies entirely in $U_1 \cup U_2$ (recall the sets from the definition of $\ce_{\mathrm{bar}}$).
We take $w_0' = w_0$ if $w_0 \in R_{\theta}$, otherwise, take $w_0' \in \L_0 \cap R_{\theta}$ to be the {nearest} vertex to $w_0$ (i.e., it is the corner vertex of $R_{\theta}$ on $\L_0$ closer to $w_0$); and similarly, take $w_1' = w_1$ if $w_1 \in R_{\theta}$, otherwise, take $w_1' \in \L_r \cap R_{\theta}$ to be the nearest vertex to $w_1$, (see Figure \ref{f:key0} (a)). 
Then obviously, $d(w_0)=d(w_0')$, $d(w_1)=d(w_1')$. We also have, by definition of $\Gamma'$, that $$T_{w_0, w_1} \geq T_{w_0', w_1'} \geq T_{w_0', w_1'}^{R_{\theta}}.$$
By the definition of $\ce_{\mathrm{bar}}$ and $\cA_{\theta, \mathrm{reg}}$, we have $T_{w_0, w_1} \leq \E T_{w_0, w_1} - (L-2)r^{1/3}$ where we have used $\cA_{\theta, \mathrm{reg}}$ to take care of the endpoints.
However, on $\cA_{\theta, \mathrm{reg}}$, 
$$T_{w_0', w_1'}^{R_{\theta}} \geq 2|d(w_1')-d(w_0')| - \theta^{-2}r^{1/3}.$$
causing us to arrive at a contradiction when $\theta$ is small enough (we also use that $\E T_{w_0, w_1} \leq 2|d(w_1)-d(w_0)|+Cr^{1/3}=2|d(w_1')-d(w_0')|+Cr^{1/3}$ from \eqref{e:mean}).

For each $h \in [0, 2r-\frac{r}{L}]$, we denote $\mathrm{TF}_h$ to be the event that there exists some $v \in \Gamma'$, such that $d(v) \in [h, h+\frac{r}{L}]$, and $v \not\in R_{2\theta}$.
We then have
\begin{equation}\label{intersectionevents}
\mathrm{TF}(\Gamma')^c\bigcap \cA_{\theta, \mathrm{reg}} \bigcap \ce_{\mathrm{bar}} \bigcap (\ce')^c \subset \bigcup_{i=2}^{\lfloor 2L \rfloor - 3} \mathrm{TF}_{\frac{ir}{L}}.
\end{equation}
We next bound the following probability:
$$\P(\mathrm{TF}_h \cap \mathrm{TF}(\Gamma')^c\cap \cA_{\theta, \mathrm{reg}}\mid \ce_{\mathrm{bar}}),$$ for any $h \in [\frac{2r}{L}, 2r - \frac{3r}{L}]$. Under this event, we first take any $w_0, w_1 \in \Gamma'$, such that $w_0 \in S_{h-\frac{r}{L}}$, and $w_1 \in S_{h+\frac{r}{L}}$.
We further take $w_0' \in \L_{\lfloor (h-\frac{2r}{L})/2 \rfloor} \cap R_{\theta}$, and $w_1' \in \L_{\lceil (h+\frac{3r}{L})/2 \rceil} \cap R_{\theta}$, such that $w_0'-w_0$, $w_1'-w_1$ are parallel to $(1, 1)$ (see Figure \ref{f:key0} (b)).
Let $\Gamma_1$ be the concatenation of $\Gamma_{w_0', w_0}$, $\Gamma_{w_0, w_1}$, and $\Gamma_{w_1, w_1'}$. Noting that $d(w_0)-d(w'_0), d(w_1)-d(w_0), d(w'_1)-d(w_1)\geq \frac{r}{L}$, by $\cA_{\theta, \mathrm{reg}}$ we know that $\ell(\Gamma_1) \geq 2 (d(w_1')-d(w_0')) - 3\theta^{-2}r^{1/3}$.
Also, from its construction, $\Gamma_1$ exits $R_{2\theta}$.
Further denote $r_1:= \lceil (h+\frac{3r}{L})/2 \rceil - \lfloor (h-\frac{2r}{L})/2 \rfloor$. The event that such $\Gamma_1$ exists is covered by $\theta r^{2/3}r_1^{-2/3}$ many translations of the event $\mathrm{LargeTF}(\theta r^{2/3}r_1^{-2/3}, r_1)$, when $\theta$ is small enough.
By Proposition \ref{l: prep1-tf}, and the FKG inequality (for $\ce_{\mathrm{bar}}$ is a negative event, while translations of $\mathrm{LargeTF}(\theta r^{2/3}r_1^{-2/3}, r_1)$ are positive events, on the configuration on $\Z^2\setminus R_{\theta}$), we get
$$
\P(\mathrm{TF}_h \cap \mathrm{TF}(\Gamma')^c\cap \cA_{\theta, \mathrm{reg}} \mid \ce_{\mathrm{bar}}) \leq \theta r^{2/3}r_1^{-2/3} e^{-c_2 \theta r^{2/3}r_1^{-2/3}}.
$$
Note that since $\frac{2rL}{5r+7L} < rr_1^{-1} \leq \frac{2L}{5}$, we have by \eqref{intersectionevents} and Lemma \ref{l:prep3-rtheta}
\begin{eqnarray*}
\P((\ce')^c\mid \ce_{\mathrm{bar}}) &\leq &
\P(\cA_{\theta, \mathrm{reg}}^c)
+
\P(\mathrm{TF}(\Gamma')\cap \cA_{\theta, \mathrm{reg}}\mid \ce_{\mathrm{bar}})
+
\P\left((\cup_{i=2}^{\lfloor 2L \rfloor - 3} \mathrm{TF}_{\frac{ir}{L}}) \cap \mathrm{TF}(\Gamma')^c\cap \cA_{\theta, \mathrm{reg}} \mid \ce_{\mathrm{bar}}\right)
\\
&\leq &
{e^{-c\theta^{-1}}} + \lceil \phi/4\rceil e^{-c_2(\phi/4)^3} + 2L \theta \left(\frac{2L}{5}\right)^{2/3}e^{-c_2\theta\left(\frac{2rL}{5r+7L}\right)^{2/3}},
\end{eqnarray*}
so our conclusion follows for small enough $\theta$ and large enough $r$.
\end{proof}

We can now complete the proof of Lemma \ref{l:2nd1}. 
\begin{proof}[Proof of Lemma \ref{l:2nd1}]
Let $\ce'$ denote the event in Lemma \ref{l:key1}. We can write
$$\E[(X'-X_{2\theta})^2\mid \ce_{\mathrm{bar}}]=\E[(X'-X_{2\theta})^2\don_{\ce'}\mid \ce_{\mathrm{bar}}]+\E[(X'-X_{2\theta})^2\don_{(\ce')^c}\mid \ce_{\mathrm{bar}}].$$ 
We shall control the two terms separately. Notice that by Cauchy-Schwarz inequality, the second term is upper bounded by
$$\P((\ce')^c \mid \ce_{\mathrm{bar}})^{1/2}\E[(X'-X_{2\theta})^4\mid \ce_{\mathrm{bar}}]^{1/2}.$$
Using the fact that $X'\geq X_{2\theta}\geq X_{\theta}$ and the FKG inequality (and the fact that $\ce_{\mathrm{bar}}$ is a negative event on the configuration on $\Z^2\setminus R_{\theta}$) we get that 
$$\E[(X'-X_{2\theta})^4\mid \ce_{\mathrm{bar}}]\leq \E[X'-X_{\theta}]^4\leq 16( \E[(X'-4r)_{+}^4]+\E[(4r-X_{\theta})_{+}^4])\leq C\theta^{-4}r^{4/3}$$
where the final inequality follows from Lemma \ref{l:l2lbd} and Proposition \ref{p:constrained}. These together with Lemma \ref{l:key1} imply that 
$$\E[(X'-X_{2\theta})^2\don_{(\ce')^c}\mid \ce_{\mathrm{bar}}] \leq \frac{1}{2}r^{2/3}$$
for $\theta$ sufficiently small. 
For the first term notice that by definition of $\ce'$ we have on this event 
$$0\leq X'-X_{2\theta}\leq  \sup_{u\in U'} (X_{u}-T_{\mathbf{0},u}^{R_{\theta}})+ \sup_{u\in U''} (X^*_{u}-T_{u,\mathbf{r}}^{R_{\theta}})$$
where $U'$ (resp.\ $U''$) denotes the set of all points $u\in R_{\theta}$ with $$d(u)\in {\left[(\theta/2)^{3/2}r, (\theta/2)^{3/2}r + \frac{r}{L}\right]},  \left(\text{resp. }\ 2r-d(u)\in {\left[(\theta/2)^{3/2}r, (\theta/2)^{3/2}r + \frac{r}{L}\right]}\right)$$ and $X^*_{u}$ denotes the weight of the maximum weight path from $u$ to $\L_{r}$. Using the FKG inequality again and Lemma \ref{l:sideregularity} we get from the above that
$$\E[(X'-X_{2\theta})^2 \don_{\ce'} \mid \ce_{\mathrm{bar}}] = O(\theta r^{2/3})\leq \frac{1}{2}r^{2/3}$$
for $\theta$ sufficiently small. Putting together the two bounds completes the proof of the lemma.
\end{proof}

\subsection{Proof of Lemma \ref{l:2nd2}}
We shall prove Lemma \ref{l:2nd2} in this subsection. The strategy of the proof will be as follows: on a high probability event (conditional on $\ce_{\mathrm{dec}}\cap \ce_{\mathrm{bar}}$) we shall show that $X_{n}-X^r_{n}-X_{2\theta}$ has small conditional second moment whereas the other term is controlled by Cauchy-Schwarz inequality and an a priori control of the fourth moment as in the proof of Lemma \ref{l:2nd1}. We first start with the control on the fourth moment. 

\begin{lemma}  \label{l:covboundw}
There is absolute constant $C>0$ such that
$$
\E\left[(X_{n}-X_{n}^r-X_{2\theta})_+^4  \mid \cE_{\mathrm{bar}}\cap \cE_{\mathrm{dec}}\right] \leq C \theta^{-4}r^{4/3}.
$$
\end{lemma}
\begin{proof}
First we write
$$\E\left[(X_{n}-X_{n}^r-X_{2\theta})_+^4  \mid \cE_{\mathrm{bar}}\cap \cE_{\mathrm{dec}}\right]
\leq
8\E\left[(X_{n}-X_{n}^r-4r)_+^4  \mid \cE_{\mathrm{bar}}\cap \cE_{\mathrm{dec}}\right]
+
8\E\left[(X_{2\theta} - 4r)^4  \mid \cE_{\mathrm{bar}}\cap \cE_{\mathrm{dec}}\right].$$
Noticing that $X_{2\theta}-4r$ is independent of $\cE_{\mathrm{dec}}$, we show, as in the proof of Lemma \ref{l:2nd1} that the second term above is $O(\theta^{-4}r^{4/3})$. So it suffices to provide a similar upper bound for the first term.  By the FKG inequality, we get 
$$\E\left[(X_{n}-X_{n}^r-4r)_+^4  \mid \cE_{\mathrm{bar}}\cap \cE_{\mathrm{dec}}\right]
\leq
\E\left[(X_{n}-X_{n}^r-4r)_+^4  \mid  \cE_{\mathrm{dec}}\right].$$

To control this, we write, using the definition of $\cE_{\mathrm{dec}}$,
\begin{eqnarray*}
\P\biggl(X_{n}-X_{n}^r-4r &\geq & xr^{1/3}\mid \cE_{\mathrm{dec}}\biggr) \leq  \P\biggl(\max_{|m|<r^{2/3}} X_{\br+(m,-m)}-4r\geq xr^{1/3}\biggr)\\
&+& \sum_{j=1}^{\infty} \P\left( \max_{2^{j-1}r^{2/3}\leq |m| < 2^jr^{2/3}} X_{\br+(m,-m)}-4r -2\alpha\cdot 2^{j\left(\frac{1}{2}-\tau\right)}r^{1/3} \geq xr^{1/3}\right).
\end{eqnarray*}

Apply Proposition \ref{p:l2l} to $\L_{r, r^{2/3}} + (i, -i)$, for $i \in \lfloor 2r^{2/3} \rfloor\Z$; we have for each $x > 0$ 
$$
\P\left[X_{\br+(m,-m)}-4r \geq xr^{1/3}\right] \leq C'\exp(-c'x),
$$
and for each $j \in \Z_+$,
\begin{multline*}
\P\left[\max_{2^{j-1}r^{2/3}\leq |m| < 2^jr^{2/3}} X_{\br+(m,-m)}-4r -2\alpha\cdot 2^{j\left(\frac{1}{2}-\tau\right)}r^{1/3} \geq xr^{1/3}\right] \leq 2^jC'\exp\left(-c'\left(x+2\alpha\cdot 2^{j\left(\frac{1}{2}-\tau\right)}\right)\right),
\end{multline*}
where $C', c' > 0$ are absolute constants. 
These imply that, for some absolute constants $C,c>0$, we have 
$$\P\biggl(X_{n}-X_{n}^r-4r \geq  xr^{1/3}\mid \cE_{\mathrm{dec}}\biggr)\leq Ce^{-cx},$$ 
for all $x>0$ and it follows in turn that 
$$\E\left[(X_{n}-X_{n}^r-4r)_+^4  \mid \cE_{\mathrm{bar}}\cap \cE_{\mathrm{dec}}\right]\leq C'r^{4/3},$$
for some $C'>0$ as required, completing the proof of the lemma. 
\end{proof}

Our next objective is to show that on a large probability event (conditional on $\cE_{\mathrm{bar}}\cap \cE_{\mathrm{dec}}$) we have 
$$X_{n}-X_{n}^r-X_{2\theta}\leq X'-X_{2\theta}.$$ By definition of $X'$, for this, it suffices to construct an event $A$ such that on $A\cap \cE_{\mathrm{bar}}\cap \cE_{\mathrm{dec}}$, for any point $u=(u_1,u_2)$ on $\L_{r}$ with $|u_1-u_2|\geq \phi r^{2/3}$ we have 
$$T_{u,\bn}+X_{u}\leq T_{\br,\bn}+ X_{2\theta}.$$
{As we have, on $\cE_{\mathrm{dec}}$, $X_{n}^r \leq T_{\br,\bn}+$}${2\alpha^{-1}r^{1/3}}$, this is ensured by the following two events (after observing that $X_{2\theta}\geq X_{\theta}$):

\begin{enumerate}
\item[(i)] Let $\ce_1:=\{X_{\theta} \geq 4r-(\theta^{-3}-{2\alpha^{-1}})r^{1/3}\}$.
\item[(ii)] For all $j$ such that $2^{j-1}\geq \frac{1}{4}\phi$, let 
$$\ce_{2,j}:=\left\{ \max_{2^{j-1}r^{2/3}|m|<2^{j}r^{2/3}} X_{\br+(m,-m)} \leq 4r -\theta^{-3}r^{1/3} + \alpha\cdot 2^{j\left(\frac{1}{2}-\tau\right)}r^{1/3}  \right\}$$ 
and let $\displaystyle{\ce_2:=\bigcap_{j} \ce_{2,j}}$ where the intersection is over all $j$ as above. 
\end{enumerate}

Let us set $\mathfrak{A}:=\ce_1\cap \ce_2$. The following lemma is an obvious consequence of the above discussion together with definition of $\ce_{\mathrm{dec}}$.

\begin{lemma}  
\label{l:condition}
On the event $\mathfrak{A}\bigcap \cE_{\mathrm{bar}}\bigcap \cE_{\mathrm{dec}}$, we have $(X_{n}-X_{n}^r-X_{2\theta})_+ \leq (X' - X_{2\theta})$. 
\end{lemma}

The next lemma shows that $\mathfrak{A}$ has large conditional probability given $\cE_{\mathrm{bar}}\cap \cE_{\mathrm{dec}}$.

\begin{lemma} 
\label{l:goodevent}
When $\theta$ is small enough,
we have $\P[\mathfrak{A}^c|\cE_{\mathrm{bar}}\cap \cE_{\mathrm{dec}}] <{ \exp(-\theta^{-2/5})}$
for all sufficiently large $r$.
\end{lemma}

\begin{proof}
Observe first that both $\ce_1$ and $\ce_2$ are independent of $\cE_{\mathrm{dec}}$, and $\ce_1$ is also independent of $\cE_{\mathrm{bar}}$. Observe further that by the FKG inequality we have $\P(\ce^c_2\mid \cE_{\mathrm{bar}})\leq \P(\ce_2^c)$. It therefore suffices to upper bound
$\P(\ce_1^c)$ and $\P(\ce_2^c)$ separately. It follows {from Lemma \ref{l:prep3-rtheta} (or the proof of Proposition \ref{p:constrained}) that $\P(\ce^c_1)\leq e^{-c\theta^{-1}}$} for small enough $\theta$. Observing that $2^{j-1} \geq \frac{\phi}{4} \gg \theta^{-10}$ by our choice of parameters, it follows that for $\theta$ sufficiently small $\alpha\cdot 2^{j\left(\frac{1}{2}-\tau\right)} \geq 2\theta^{-3}$. Using Proposition \ref{p:l2l} as in the proof of Lemma \ref{l:covboundw} we get 
$$\P(\ce_{2,j}^c)\leq 2^j e^{-c2^{j(1/2-\tau)}}$$
for some $c>0$.  By summing over all $j$ (and using the bound on the minimum value of $j$) we get $\P(\ce_2^c)\leq e^{-1/\theta^{1/2}}$ for all $\theta$ small enough. We get the desired result by a union bound. 
\end{proof}

We are now ready to prove Lemma \ref{l:2nd2}.
\begin{proof}[Proof of Lemma \ref{l:2nd2}]
Clearly, it suffices to show, separately, that $\E[(X_{n}-X_{n}^r-X_{2\theta})_{-}^2\mid \cE_{\mathrm{bar}}\bigcap \cE_{\mathrm{dec}}]=O(r^{2/3})$ and $\E[(X_{n}-X_{n}^r-X_{2\theta})_{+}^2\mid \cE_{\mathrm{bar}}\bigcap \cE_{\mathrm{dec}}]=O(r^{2/3})$. {Notice that by definition, on $\cE_{\mathrm{dec}}$ we have $X_{n}\geq T_{\br, \bn}+X_{2\theta}$ which implies $X_{n}-X_{n}^r-X_{2\theta}\geq -{2\alpha^{-1}}r^{1/3}$. This gives the first desired inequality.}

To bound $\E\left[(X_{n}-X_{n}^r-X_{2\theta})_+^2  \mid \cE_{\mathrm{bar}}\bigcap \cE_{\mathrm{dec}}\right]$, we write 
\begin{eqnarray*}
\E\left[(X_{n}-X_{n}^r-X_{2\theta})_+^2  \mid \cE_{\mathrm{bar}}\bigcap \cE_{\mathrm{dec}}\right] &=&\E\left[\don_{\mathfrak{A}}(X_{n}-X_{n}^r-X_{2\theta})_+^2  \mid \cE_{\mathrm{bar}}\bigcap \cE_{\mathrm{dec}}\right]\\ &+&\E\left[\don_{\mathfrak{A}^c}(X_{n}-X_{n}^r-X_{2\theta})_+^2  \mid \cE_{\mathrm{bar}}\bigcap \cE_{\mathrm{dec}}\right].
\end{eqnarray*}
As in the proof of Lemma \ref{l:2nd1}, we bound the second term by 
$$\P(\mathfrak{A}^c\mid \cE_{\mathrm{bar}}\bigcap \cE_{\mathrm{dec}})^{1/2} \E\left[(X_{n}-X_{n}^r-X_{2\theta})_+^4  \mid \cE_{\mathrm{bar}}\bigcap \cE_{\mathrm{dec}}\right]^{1/2} \leq r^{2/3}$$
for $\theta$ sufficiently small where the last inequality follows from Lemma \ref{l:covboundw} and Lemma \ref{l:goodevent}. Using Lemma \ref{l:condition}, the first term is bounded by 
$$\E[(X'-X_{2\theta})^{2}\mid \cE_{\mathrm{bar}}] \leq r^{2/3}$$
where the inequality follows from Lemma \ref{l:2nd1}. The proof is completed combining all these. 
\end{proof}

\subsection{Proof of Lemma \ref{l:realvar}}
Now we move towards the proof of Lemma \ref{l:realvar} which will be done by a standard Doob Martingale variance decomposition. Recall the definition of $\ce_{*}$. Let us first fix $\omega_{\theta}\in \ce_*$. Let $J=\{i_1<i_2<\cdots < i_{|J|}\}$ an enumeration of good indices (notice that these are deterministic given $\omega_{\theta}$); i.e., for each $j\in J$
$$\P\left(\sup_{u\in R_{2\theta}^{5j}, v\in R_{2\theta}^{5j+4}} T_{u,v}- 2|d(u)-d(v)| \leq  C_2\theta^{1/2}r^{1/3}\mid \omega_{\theta}\right) \geq 0.99$$ 
By definition of $\ce_{*}$, we know that $|J|\geq \frac{1}{10}\theta^{-3/2}$. Let us now define a sequence of $\sigma$-fields, $\ch_0 \subseteq \ch_1 \subseteq \cdots \subseteq \ch_{|J|}$ where $\ch_0$ is generated by the configuration $\omega_{\theta}$ on $\Z^2\setminus R_{\theta}$ together with the configuration on $R_{\theta}^{i}$ for all $i$ not of the form $5j+2$ for some $j\in J$, and for $j\geq 0$, $\ch_{j+1}$ is the sigma algebra generated by $\ch_{j}$ and the configuration on $R_{\theta}^{5i_{j+1}+2}$. 
We shall consider the Doob Martingale $M_j:=\E[X_{2\theta}\mid \ch_{j}]$. Observe that by the standard variance decomposition of a Doob Martingale it follows that 
\begin{equation}
\label{l:vardecom}
{\rm Var}~(X_{2\theta}\mid \omega_{\theta})\geq \sum_{j=1}^{J} \E [(M_j-M_{j-1})^2].
\end{equation}
Lemma \ref{l:realvar} is immediately implied by \eqref{l:vardecom} together with the following lemma.
\begin{lemma}
\label{l:eachgood}
There exists $c>0$ such that given any sufficiently small $\theta,$ for all sufficiently large $r$ the following holds: for each $i_j$ in $J$, there is a $\ch_{j-1}$ measurable subset $G_j$ with probability at least $1/2$ such that we have
$$\don_{G_j} \E[(M_j-M_{j-1})^2\mid \ch_{j-1}] \geq c\theta r^{2/3}.$$ 
\end{lemma}

Before starting to prove this lemma we need to introduce some more notation. Fix $i_j\in J$. Let $e_j$ and $f_j$ denote the midpoints of two shorter sides of $R_{\theta}^{5i_j+2}$ (see Figure \ref{f:resampling}). Let $\omega_1$ denote a configuration on $R_{\theta}^{5i_j+2}$ drawn from the i.i.d.\ exponential distribution. 
Let $F$ be the event where
$$T^{R_{\theta}}_{e_j,f_j}\geq 4\theta^{3/2}r+ 2C_2\theta^{1/2}r^{1/3}.$$
Let $\tilde{\omega}_1$ denote a configuration on $R_{\theta}^{5i_j+2}$ drawn from the i.i.d.\ exponential distribution conditional on $F$. Observe that $F$ is an increasing event in the configuration on $R_{\theta}^{5i_{j}+2}$. As discussed before, the FKG inequality implies that $\don_{F}$ is positively correlated with every increasing function in the configuration on $R_{\theta}^{5i_{j}+2}$. Which further implies that the conditional distribution of the weights on $R_{\theta}^{5i_{j}+2}$ given $F$ (i.e., the law of $\tilde{\omega}_1$) is stochastically larger than the law of $\tilde{\omega}$. Strassen's Theorem (see e.g.\ \cite{L99}) therefore implies that there exists a coupling of $(\omega_1,\tilde{\omega}_1)$ such that $\tilde{\omega}_1 \geq \omega_1$ point wise. Let $\mu$ denote such a coupling. Let $\omega_0$ denote a configuration on all the vertices that are revealed in $\ch_{j-1}$ (in particular, the restriction of $\omega_0$ to the co-ordinates of $\Z^2\setminus R_{\theta}$ is $\omega_{\theta}$). Let $\omega_2$ denote a configuration on $\cup_{\ell >j} R_{\theta}^{5i_{\ell}+2}$. Let $\underline{\omega}= (\omega_0, \omega_1, \omega_2)$ and $\underline{\tilde{\omega}}= (\omega_0, \tilde{\omega}_1, \omega_2)$ denote the two environments and let $X^*(\underline{\omega})$ (resp.\ $X^*(\underline{\tilde{\omega}})$) denote the value of the statistic $X_{2\theta}$ computed in the environment $\underline{\omega}$ (resp.\ $\underline{\tilde{\omega}}$). Observe now that for the coupling $\mu$ described above we have 

\begin{equation}
\label{e:keystepvar0}
(\E[M_j\mid F, {\ch_{j-1}}]-M_{j-1})(\omega_0)=\int (X^*(\underline{\tilde{\omega}})-X^*(\underline{\omega}))~ d\mu~d\omega_2.
\end{equation}
Observe also that $F$ is independent of $\ch_{j-1}$ (and $M_{j-1}$), hence we get 
\begin{equation}
\label{e:keystepvar}
\begin{split}
\E[(M_j-M_{j-1})^2\mid \ch_{j-1}] &\geq 
\P(F)\E[(M_j-M_{j-1})^2\mid F, \ch_{j-1}]
\\
&=
\P(F)(\E[M_j^2\mid F, \ch_{j-1}] - 2\E[M_j\mid F, \ch_{j-1}]M_{j-1} + M_{j-1}^2)
\\
&\geq
\P(F)(\E[M_j\mid F, \ch_{j-1}]^2 - 2\E[M_j\mid F, \ch_{j-1}]M_{j-1} + M_{j-1}^2)
\\
&=\P(F)(\E[M_j\mid F, \ch_{j-1}]-M_{j-1})^{2}.
\end{split}
\end{equation}

\begin{figure}[htbp!]
\includegraphics[width=.5\textwidth]{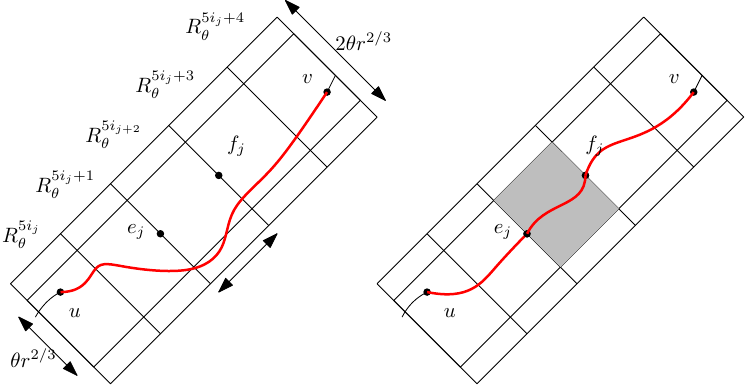} 
\caption{Figure illustrating the proof of Lemma \ref{l:eachgood} by construction of the Doob Martingale $M_j.$ The figure shows a particular instance where high values of the exponential variables in the rectangle $R_{\theta}^{5i_j+2}$ leads to a quantitative lower bound of the $j^{th}$ term in the RHS of \eqref{l:vardecom}. This is formulated precisely in Lemma \ref{l:keystepvar}.}
\label{f:resampling}
\end{figure}

\begin{lemma}
\label{l:keystepvar}
Given $C_2$ sufficiently large, for all $\theta$ sufficiently small, for all $r$ sufficiently large, and $\omega_{\theta}\in \ce_{*}$, there is a set $G_j$ of configurations $\omega_0$ with probability at least $0.9$ and a set $G^2_j$ of configurations $\omega_2$ with probability at least $0.9$ such that for each $\omega_0\in G_j$  there exists a set $G^3_j(\omega_0)$ of configurations $(\omega_1,\tilde{\omega}_1)$ with $\mu$-probability at least $0.9$ with the following property: for $\omega_0\in G_j$,  $(\omega_1,\tilde{\omega}_1)\in G^3_j(\omega_0)$ and $\omega_2\in G^2_j$ we have
$X^*(\underline{\tilde{\omega}})-X^*(\underline{\omega})\geq \frac{C_2}{10}\theta^{1/2}r^{1/3}$.
\end{lemma}

Let us postpone the proof of Lemma \ref{l:keystepvar} for the moment and complete the proof of Lemma \ref{l:eachgood}.

\begin{proof}[Proof of Lemma \ref{l:eachgood}]
For $G_j$ as in Lemma \ref{l:keystepvar} and $\omega_0\in G_j$ we have from \eqref{e:keystepvar0} and \eqref{e:keystepvar}  and the fact that $X^*(\underline{\tilde{\omega}})\geq X^*(\underline{\omega})$ point wise that
$$\E[(M_j-M_{j-1})^2\mid \ch_{j-1}](\omega_0) \geq  \P(F)\left( \int_{G_j^2} \int_{G^3_j(\omega_0,\omega_2)} (X^*(\underline{\tilde{\omega}})-X^*(\underline{\omega}))~ d\mu~d\omega_2\right)^2.$$
It follows from Lemma \ref{l:uppertail} that $\P(F)\geq \tilde c$ for some $\tilde c>0$ and using the properties of $G^2_j$ and $G^3_j$ from Lemma \ref{l:keystepvar} we see that the right hand side above is lower bounded by $\frac{\tilde cC_2^2}{1000}\theta r^{2/3}$ and the lemma follows.
\end{proof}

It remains to complete the proof of Lemma \ref{l:keystepvar}.

\begin{proof}[Proof of Lemma \ref{l:keystepvar}]
Let $\omega_{\theta}\in \ce_{*}$ be fixed and let $j\in J$. Let us first define $G_j$ and $G^2_j$. Notice that $\omega_0$ and $\omega_2$ both refer to certain different  rectangles $R_{\theta}^{j}$. Let us call these disjoint sets of indices $J_1$ and $J_2$. Let  $\tilde{G}^1_j$ denote the event that for each $k\in J_1$ we have 
$$ T_{u,v}^{R^k_{\theta}} \geq \E T_{u,v}- \log ^{10}(\theta^{-1}) \theta^{1/2} r^{1/3}$$
for all $(u,v)\in S(R_{\theta}^{k})$. Further, we also ask for $k=5i_j, 5i_j+1, 5i_j+3, 5i_j+4$
$$ T_{u,v}^{R^k_{\theta}} \geq \E T_{u,v}-\frac{C_2}{4}\theta^{1/2} r^{1/3}$$
for all $(u,v)\in S(R_{\theta}^{k})$.
Clearly for $C_2$ sufficiently large we have $\P(\tilde{G}^1_j)\geq 0.999$ using Theorem \ref{t:supinf}. Let $\tilde{G}_j^4$ denote the set of all configurations $\omega_1$ such that $$ T_{u,v}^{R^{5i_j+2}_{\theta}} \geq \E T_{u,v}- \log ^{10}(\theta^{-1}) \theta^{1/2} r^{1/3}$$ for all $(u,v)\in S(R_{\theta}^{k})$.  {It follows by Fubini's theorem and the fact that $j$ is a good index that there is a subset $G_j$ of $\tilde{G}^1_j$ with probability at least $0.9$ such that for each $\omega_0\in G_j$, there exists a subset $\tilde{G}_j^3(\omega_0)$ of $\tilde{G}_j^4$ such that for all $\omega_1\in \tilde{G}_j^3(\omega_0)$ we have $T_{u,v}(\underline{\omega})\leq 2|d(u)-d(v)|+C_2\theta^{1/2} r^{1/3}$ for each $u\in R_{2\theta}^{5i_j}$ and $v\in R_{2\theta}^{5i_j+4}$} (notice that this event does not depend on $\omega_2$). Let $G_j^3(\omega_0)$ denote the set of all configurations $(\omega_1, \tilde{\omega}_1)$ under the coupling $\mu$ such that $\omega_1\in \tilde{G}_j^3(\omega_0)$. Clearly the sets $G_j$ and $G_j^3(\omega_0)$ satisfy the required probability bounds. Let $G_j^2$ denote the event that for each $k\in J_2$ we have 
$$ T_{u,v}^{R^k_{\theta}} \geq \E T_{u,v}- \log ^{10}(\theta^{-1}) \theta^{1/2} r^{1/3}$$
for all $(u,v)\in S(R_{\theta}^{k})$. Again, by Theorem \ref{t:supinf} we have $\P(G_j^2)\geq 0.9$, as required. 
 
It remains to prove that for $\underline{\omega}, \underline{\tilde{\omega}}$ constructed as above from the events $G_j, G_j^2$ and $G_j^3$ as described in the statement of the lemma, we have 
$X^*(\underline{\tilde{\omega}})-X^*(\underline{\omega})\geq \frac{C_2}{2}\theta^{1/2}r^{1/3}$. Let $\gamma_1$ denote the path attaining $X_{2\theta}$ in the environment $\underline{\omega}$. Observe that, by using the definition of the events we conclude as in the proof of Lemma \ref{l:key1} that $\gamma_1$ must intersect $R_{\theta}^{5i_j}$ and $R_{\theta}^{5i_j+4}$, say at points $u$ and $v$. Observe that by definition 
$T_{u,v}^{R_{2\theta}}(\underline{\omega})\leq 2|d(u)-d(v)|+C_2\theta^{1/2} r^{1/3}$. On the other hand by definition of $\tilde{\omega}_1$ we have
$$T^{R_{\theta}}_{u,e_j}(\underline{\tilde{\omega}})+T^{R_{\theta}}_{e_j,f_j}(\underline{\tilde{\omega}})+T^{R_{\theta}}_{f_j,v}(\underline{\tilde{\omega}}) \geq 2|d(u)-d(v)|+\frac{11C_2}{10}\theta^{1/2} r^{1/3}$$
concluding the proof. 
\end{proof}

\bibliography{aging-flat}
\bibliographystyle{plain}

\appendix
\section{Brownian Calculations}
\label{s:appa}
As promised before, we provide the remaining details of some of the straightforward computations using Brownian motions that were omitted from the main text. 

\begin{proof}[Proof of \eqref{bcomp}]
Recall that we need to show
$$\P\left[\max_{x \in I} W(x) > \max_{x \in [-2M, 2M]} W(x)-\sqrt{\varepsilon}\right] \leq C_1 \varepsilon.
$$
{Let the end points of $I$ be $-M \leq x_1 < x_2 \leq M$,
and let $m:=x_2-x_1=|I|$.
We have
\begin{align*}
&\P\left[\max_{x \in I} W(x) > \max_{x \in [-2M, 2M]} W(x)-\sqrt{\varepsilon}\right]
\\
&\leq
\iint_{[0,\infty)^2} \P\left[ \max_{x \in [x_1, x_2]} W(x) - W(x_1) = h_1, \max_{x \in [x_1, x_2]} W(x) - W(x_2) = h_2\right]
\\
&\times
\P \left[ \max_{x \in [x_1-M, x_1]} W(x) - W(x_1) < h_1 + \sqrt{\varepsilon}\right]
\P \left[ \max_{x \in [x_2, x_2+M]} W(x) - W(x_2) < h_2 + \sqrt{\varepsilon}\right]
dh_1 dh_2.
\end{align*}
(Above, the first term in the second line denotes the probability density)}.
Using reflection principle, this equals to
\begin{align*}
&\iint_{[0,\infty)^2}
\frac{(h_1+h_2)\exp(-(h_1+h_2)^2/4m)}{\sqrt{4\pi}m^{3/2}}
\\
&\times
 2 \P\left[|W(x_1-M) - W(x_1)| < h_1 + \sqrt{\varepsilon}\right] \cdot 2 \P\left[|W(x_2+M) - W(x_2)| < h_2 + \sqrt{\varepsilon}\right]
dh_1 dh_2
\\
&\leq
\iint_{[0,\infty)^2}
\frac{2(h_1+h_2)\exp(-(h_1+h_2)^2/4m)}{\sqrt{4\pi}m^{3/2}}
\cdot
\frac{(h_1 + \sqrt{\varepsilon})(h_2 + \sqrt{\varepsilon})}{4\pi M}
dh_1 dh_2 
\\
&=
\iint_{[0,\infty)^2}
\frac{2(h_1+h_2)\exp(-(h_1+h_2)^2/4m)}{\sqrt{4\pi}m^{3/2}}
\cdot
\frac{h_1h_2 + \sqrt{\varepsilon}(h_1+h_2) + \varepsilon}{4\pi M}
dh_1 dh_2 
\end{align*}
We note that (by change of variables)
\begin{equation*}
\begin{split}
\iint_{[0,\infty)^2}
\frac{h_1h_2(h_1+h_2)\exp(-(h_1+h_2)^2/4m)}{m^{5/2}}
dh_1 dh_2, \\
\iint_{[0,\infty)^2}
\frac{(h_1+h_2)^2\exp(-(h_1+h_2)^2/4m)}{m^{2}}
dh_1 dh_2, \\
\iint_{[0,\infty)^2}
\frac{(h_1+h_2)\exp(-(h_1+h_2)^2/4m)}{m^{3/2}}
dh_1 dh_2,
\end{split}
\end{equation*}
are finite and independent of $m$.
This implies that for some constant $C_2$ depending on $M$,
$$
\P\left[\max_{x \in I} W(x) > \max_{x \in [-2M, 2M]} W(x)-\sqrt{\varepsilon}\right]
\leq
C_2(m + \sqrt{\varepsilon}\sqrt{m} + {\varepsilon}),
$$
and our conclusion follows since $m \leq \varepsilon$.
\end{proof}

\begin{proof}[Completion of Proof of Lemma \ref{l:error}] 
Denote
\begin{align*}
\mathscr{C}_*' &:= \left\{\sup_{|x|\leq \lambda} W(x)= \sup_{|x|\leq \sqrt{2\mathscr{M}}} W(x) \leq \sup_{2^{k-1}\lambda'\leq |x|\leq 2^{k}\lambda'} W(x)+2^{k(\frac{1}{2} - \tau)} \alpha\sqrt{\lambda'}~\text{for some}~k\leq k_* \right\}, \text{ and },\\
\mathscr{C}_{\#}' &:= \left\{\sup_{|x|\leq \lambda} W(x)= \sup_{|x|\leq \sqrt{2\mathscr{M}}} W(x) \geq W(0) + \alpha^{-1}\sqrt{\si'} \right\}.
\end{align*}
Recall the event $\mathscr{C}':= \mathscr{C}_*'\cup \mathscr{C}_{\#}'$,
and that it was left to show that  $\lim_{\alpha \to 0} \lambda^{-1}\P(\mathscr{C}')= 0$
uniformly in $0<\lambda<\lambda'<1$. 
{We first study $\mathscr{C}_*'$.
Take any $k\in \N$, $k \leq k_*$.
Let $y_1:=\sup_{-\sqrt{2\mathscr{M}} \leq x \leq -\si}W(x)-W(-\si)$,
$y_2:=\sup_{|x| < \si}W(x)-W(-\si)$,
and $y_3:=\sup_{|x| < \si}W(x)-W(\si)$.
Then we have
\begin{multline}   \label{eq:lbae:cla1}
\P\left[\sup_{x: |x| \leq \si} W(x) = \sup_{x: |x| \leq \sqrt{2\mathscr{M}}} W(x) <
\sup_{2^{k-1}\si'\leq x\leq 2^{k}\si' } W(x) + 2^{k(\frac{1}{2} - \tau)} \alpha\sqrt{\si'}\right]
\\
=
\int \P[y_1 < y_2, y_3=t]\P\left[ \sup_{0\leq x \leq \sqrt{2\mathscr{M}}-\si} W(x)<t, 
\sup_{2^{k-1}\si'-\si\leq x\leq 2^{k}\si'-\si} W(x)
> t - 2^{k(\frac{1}{2} - \tau)} \alpha\sqrt{\si'}\right] dt.
\end{multline}
For any $t_1, t_2, t>0$, by reflection principle we have
\begin{align*}
\P[y_1=t_1] = \frac{2\exp(-t_1^2/(4(\sqrt{2\mathscr{M}}-\si)))}{\sqrt{4\pi(\sqrt{2\mathscr{M}}-\si)}};
\P[y_2=t_2, y_3=t] = \frac{(t_2+t)\exp(-(t_2+t)^2/8\si)}{\sqrt{8\pi}\si^{3/2}}, \text{and thus}
\end{align*}

\begin{align*}
\P[y_1 < y_2, y_3=t]
&=
\int_{0<t_1<t_2}
\frac{2\exp(-t_1^2/(4(\sqrt{2\mathscr{M}}-\si)))}{\sqrt{4\pi(\sqrt{2\mathscr{M}}-\si)}}
\cdot
\frac{(t_2+t)\exp(-(t_2+t)^2/8\si)}{\sqrt{8\pi}\si^{3/2}} dt_1 dt_2
\\
&=
\int_{0<t_1} \frac{2\exp(-t_1^2/(4(\sqrt{2\mathscr{M}}-\si)))}{\sqrt{4\pi(\sqrt{2\mathscr{M}}-\si)}}
\cdot
\frac{4\exp(-(t_1+t)^2/8\si)}{\sqrt{8\pi\si}} dt_1
\\
&\leq
\frac{8}{\sqrt{4\pi(\sqrt{2\mathscr{M}}-\si)}}
\int_{0<t_1} 
\cdot
\frac{\exp(-(t_1^2+t^2)/8\si)}{\sqrt{8\pi\si}} dt_1
=
\frac{4\exp(-t^2/8\si)}{\sqrt{4\pi(\sqrt{2\mathscr{M}}-\si)}}.
\end{align*}
For the other factor of the integrand in the RHS of \eqref{eq:lbae:cla1}, consider the stopping time:
$$
x_* := \inf\{x \in [2^{k-1}\si'-\si, 2^k \si' - \si]: W(x)>t - 2^{k(\frac{1}{2} - \tau)} \alpha\sqrt{\si'}\}\bigcup\{\sqrt{2\mathscr{M}}-\si\}.
$$
Then we have that
\begin{multline*}
\P\left[ \sup_{0\leq x \leq \sqrt{2\mathscr{M}}-\si} W(x)<t, 
\sup_{2^{k-1}\si'-\si\leq x\leq 2^{k}\si'-\si} W(x)
> t - 2^{k(\frac{1}{2} - \tau)} \alpha\sqrt{\si'}\right]
=
\P\left[ \sup_{0\leq x \leq \sqrt{2\mathscr{M}}-\si} W(x)<t, 
x_* < 8\right].
\end{multline*}
Using the fact that $x\mapsto W(x+x_*)-W(x_*)$ is again a Brownian motion and has the same law as $W$, we bound this by
\begin{multline*}
\P\left[ \sup_{0 \leq x \leq 2^{k-1}\si'-\si)} W(x)<t \right]
\P\left[ \sup_{0 \leq x \leq \sqrt{2\mathscr{M}}-\si-8} W(x)<2^{k(\frac{1}{2} - \tau)} \alpha\sqrt{\si'} \right]
\\
\leq
\frac{2t}{\sqrt{4\pi\cdot 2^{k-1}\si'}} \cdot \frac{2\cdot 2^{k(\frac{1}{2} - \tau)}\alpha\sqrt{\si'}}{\sqrt{4\pi\cdot (\sqrt{2\mathscr{M}}-\si-8)}}
\leq 4t\alpha 2^{-k\tau},
\end{multline*}
where in the last inequality we assume that $\mathscr{M}$ is large enough.
In conclusion, we have
\begin{align}   \label{eq:lbae:cla5}
\P[\mathscr{C}_*']\leq 2\sum_{k=1}^{k_*}\int_{t>0}
\frac{4\exp(-t^2/8\si)}{\sqrt{4\pi(\sqrt{2\mathscr{M}}-\si)}} \cdot 4t\alpha 2^{-k\tau} dt
&\leq
\int_{t>0}
\frac{8\exp(-t^2/8\si)}{\sqrt{4\pi(\sqrt{2\mathscr{M}}-\si)}} \cdot \frac{4t\alpha}{1-2^{-\tau}} dt
\\
\nonumber
&=
\frac{128\si\alpha}{\sqrt{4\pi(\sqrt{2\mathscr{M}}-\si)}(1-2^{-\tau})},
\end{align}
thus $\si^{-1}\P[\mathscr{C}_*'] \rightarrow 0$ as $\alpha \rightarrow 0$, uniformly for $\si < \si' \in (0, 1)$.}
{We next consider $\mathscr{C}_{\#}'$.
We let $z_1:= \sup_{-\sqrt{2\mathscr{M}}\leq x \leq 0}W(x) - W(0)$,
$z_2:=\sup_{0< x \leq \si}W(x) - W(0)$,
and $z_3:=\sup_{0< x \leq \si}W(x) - W(\si)$,
$z_4:=\sup_{\si< x \leq \sqrt{2\mathscr{M}}}W(x) - W(\si)$.
By symmetry, we have
\begin{align}
\label{eq:lbae:cla6}
&\P\left[\sup_{|x|\leq \lambda} W(x)= \sup_{|x|\leq \sqrt{2\mathscr{M}}} W(x) \geq W(0) + \alpha^{-1}\sqrt{\si'}\right]\\
\nonumber
&=
2
\P\left[\sup_{0\leq x\leq \lambda} W(x)= \sup_{|x|\leq \sqrt{2\mathscr{M}}} W(x) \geq W(0) + \alpha^{-1}\sqrt{\si'}\right]
=2\P[z_1 < z_2, z_2 > \alpha^{-1}\sqrt{\si'}, z_3 > z_4]
\\
\nonumber
&=2\iint_{t_2 > \alpha^{-1}\sqrt{\si'}, t_3 > 0} \P[z_1 < t_2]\P[z_4 < t_3]\P[z_2 = t_2, z_3 = t_3]  dt_2 dt_3
\end{align}
For any $t_2, t_3 > 0$, by reflection principle we have
\begin{align*}
\P[z_1 < t_2] &= \int_{0 < t_1 < t_2} \frac{2\exp(-t_1^2/(4\sqrt{2\mathscr{M}}))}{\sqrt{4\pi\sqrt{2\mathscr{M}}}} dt_1
\leq
\frac{2t_2}{\sqrt{4\pi\sqrt{2\mathscr{M}}}},\\
\P[z_4 < t_3] &= \int_{0 < t_4 < t_3} \frac{2\exp(-t_1^2/(4(\sqrt{2\mathscr{M}}-\si)))}{\sqrt{4\pi(\sqrt{2\mathscr{M}}-\si)}} dt_4
\leq
\frac{2t_3}{\sqrt{4\pi(\sqrt{2\mathscr{M}}-\si)}},\\
\P[z_2=t_2, z_3=t_3] &= \frac{(t_2+t_3)\exp(-(t_2+t_3)^2/4\si)}{\sqrt{\pi}\si^{3/2}},
\end{align*}
so \eqref{eq:lbae:cla6} can be bounded by
\begin{align*}
&2\iint_{t_2 > \alpha^{-1}\sqrt{\si'}, t_3 > 0} \P[z_1 < t_2]\P[z_4 < t_3]\P[z_2 = t_2, z_3 = t_3]  dt_2 dt_3
\\
&\leq
2\iint_{t_2 > \alpha^{-1}\sqrt{\si'}, t_3 > 0}
\frac{2t_2}{\sqrt{4\pi\sqrt{2\mathscr{M}}}}
\cdot
\frac{2t_3}{\sqrt{4\pi(\sqrt{2\mathscr{M}}-\si)}}
\cdot
\frac{(t_2+t_3)\exp(-(t_2+t_3)^2/4\si)}{\sqrt{\pi}\si^{3/2}}
dt_2 dt_3
\\
&=
\frac{8\si}{\sqrt{4\pi\sqrt{2\mathscr{M}}}\sqrt{4\pi(\sqrt{2\mathscr{M}}-\si)}}\iint_{t_2 > \alpha^{-1}\sqrt{\si'/\si}, t_3 > 0}
\frac{t_2t_3(t_2+t_3)\exp(-(t_2+t_3)^2/4)}{\sqrt{\pi}}
dt_2 dt_3
\\
&\leq
\frac{8\si}{\sqrt{4\pi\sqrt{2\mathscr{M}}}\sqrt{4\pi(\sqrt{2\mathscr{M}}-\si)}}\iint_{t_2 > \alpha^{-1}, t_3 > 0}
\frac{t_2t_3(t_2+t_3)\exp(-(t_2+t_3)^2/4)}{\sqrt{\pi}}
dt_2 dt_3.
\end{align*}
Note that the integral is independent of $\si, \si'$, and converges to $0$ as $\alpha \rightarrow 0$,
so $\si^{-1}\P[\mathscr{C}_{\#}']\rightarrow 0$ as $\alpha \rightarrow 0$, uniformly for $\si < \si' \in (0, 1)$.
Thus our conclusion follows.
}
\end{proof}

\section{Convergence to Airy\texorpdfstring{$_2$}{2} Process and Consequences}
\label{s:appb}
We proceed to providing the previously omitted  proofs of Theorem \ref{t:lpptoairy}, Proposition \ref{prop:bdi} and Proposition \ref{p:maxdecay}.  We start with the latter two results assuming the first one. 

\begin{proof}[Proof of Proposition \ref{prop:bdi}]
Take $N \in \N$ such that $N\iota > 2$.
We consider a finite collection of intervals:
$$
\Theta := \left\{ \left[\frac{N_1}{N}, \frac{N_2}{N}\right] : -MN \leq N_1 < N_2 \leq MN, N_1, N_2 \in \Z \right\},
$$
and a finite collection of values:
$$
\Upsilon := \left\{ 0 < \varepsilon \leq 1 : \varepsilon N \in \Z \right\} .
$$
By Theorem \ref{t:lpptoairy}, for any $\mathcal{I} \in \Theta$ and $\varepsilon \in \Upsilon$, 
\begin{multline*}
\lim_{n\rightarrow \infty}\P\left[\max_{u\in n^{2/3}\mathcal{I}\cap \Z} T_{(u,-u),\mathbf{n}} > \max_{u\in \llbracket-2Mn^{2/3}, 2Mn^{2/3}\rrbracket} T_{(u,-u),\mathbf{n}}-\sqrt{\varepsilon}n^{1/3}\right]
\\
=
\P\left[\sup_{x\in 2^{-2/3}\mathcal{I}} \sL(x) > \sup_{x\in [-2^{-2/3}\cdot 2M, 2^{-2/3}\cdot 2M]} \sL(x) - 2^{-4/3}\sqrt{\varepsilon}\right].
\end{multline*}
If $4|\mathcal{I}| < \varepsilon < \frac{1}{2}$, by Proposition \ref{p:twopeaks} the right hand side is bounded by $C_1 \varepsilon\exp\left(C_1|\log (\varepsilon)|^{5/6}\right)$,
for some $C_1$ depending only on $M$.
Hence by taking $n_0(\iota, M)$ large, we have that
$$
\P\left[\max_{u\in n^{2/3}\mathcal{I}\cap \Z} T_{(u,-u),\mathbf{n}} > \max_{u\in \llbracket-2Mn^{2/3}, 2Mn^{2/3}\rrbracket} T_{(u,-u),\mathbf{n}}-\sqrt{\varepsilon}n^{1/3}\right]
\leq 
C_1 \varepsilon\exp\left(C_1|\log (\varepsilon)|^{5/6}\right),
$$
for any $n\geq n_0(\iota, M)$, and $\mathcal{I}\in \Theta$, $\varepsilon \in \Upsilon$ with $4|\mathcal{I}| < \varepsilon$.
Now for any $I \subset \llbracket-Mn^{2/3},Mn^{2/3}\rrbracket$ and $\varepsilon \in \left(\iota, \frac{1}{18}\right)$, with $\iota n^{2/3}\leq |I|\leq \varepsilon n^{2/3}$, we can find $\mathcal{I} \in \Theta$, such that
$$I \subset n^{2/3}\mathcal{I},\,\, |\mathcal{I}| \leq n^{-2/3}|I| + \frac{2}{N};
\text{ and, } \varepsilon' \in \Upsilon,$$ such that
$8\varepsilon < \varepsilon' \leq  8\varepsilon + \frac{1}{N}$.
Then as $\frac{2}{N}<\iota \leq n^{-2/3}|I|$, we have that
$|\mathcal{I}| \leq 2n^{-2/3}|I|$
and $\varepsilon' \leq 9 \varepsilon$,
so
$4|\mathcal{I}| \leq 8n^{-2/3}|I| \leq 8\varepsilon < \varepsilon' < \frac{1}{2}$,
and
\begin{multline*}
\P\left[\max_{u\in I} T_{(u,-u),\mathbf{n}} > \max_{u\in \llbracket-2Mn^{2/3}, 2Mn^{2/3}\rrbracket} T_{(u,-u),\mathbf{n}}-\sqrt{\varepsilon}n^{1/3}\right]
\\
\leq
\P\left[\max_{u\in n^{2/3}\mathcal{I}\cap \Z} T_{(u,-u),\mathbf{n}} > \max_{u\in \llbracket-2Mn^{2/3}, 2Mn^{2/3}\rrbracket} T_{(u,-u),\mathbf{n}}-\sqrt{\varepsilon}n^{1/3}\right]
\leq 
C_1 \varepsilon'\exp\left(C_1|\log (\varepsilon')|^{5/6}\right),
\end{multline*}
and our conclusion follows by taking $C = 9C_1$.
\end{proof}

\begin{proof}[Proof of Proposition \ref{p:maxdecay}]
Recall the event 
$$
H_j=\left\{\max_{2^{j-1} r^{2/3} \leq |u| < 2^{j} r^{2/3}} T_{(u,-u), \bn} < \max_{|u| < r^{2/3}} T_{(u,-u), \bn} - 2\alpha\cdot 2^{j\left(\frac{1}{2}-\tau\right)}r^{1/3}.\right\},
$$
and let $C_{\delta, \theta} \in \N$ be chosen depending $\delta, \theta$ (to be chosen appropriately large later).
We first claim that, for any $n, r$ with $\delta n < r < n$, and $n$ large enough, we have
\begin{multline}  \label{eq:lbe:pf1}
\P\left[\bigcap_{j=0}^{C_{\delta, \theta}}H_j\right] 
> \frac{1}{2}\P\biggl[ 
\sup_{|x|<r^{2/3}(2n)^{-2/3}} \sL(x) = \sup_{|x|<\theta r^{2/3}(2n)^{-2/3}} \sL(x){ < \sL(0) + 2\alpha^{-1}r^{1/3}\cdot 2^{-4/3}n^{-1/3},}
\\
\sup_{2^{j-1} r^{2/3}(2n)^{-2/3} \leq |x| < 2^{j} r^{2/3}(2n)^{-2/3}} \sL(x)
< \sup_{|x|<r^{2/3}(2n)^{-2/3}} \sL(x)- 2\alpha \cdot 2^{j\left(\frac{1}{2}-\tau\right)}r^{1/3}\cdot 2^{-4/3}n^{-1/3}
\\
\forall 1 \leq j \leq C_{\delta, \theta}
\biggr].
\end{multline}
We argue by contradiction.
Assume otherwise and hence there are sequences of integers $\{n_k\}_{k=1}^{\infty}$ and $\{r_k\}_{k=1}^{\infty}$, with $\lim_{k\rightarrow \infty}n_k = \infty$, such that for each $k$, $\delta n_k < r_k < n_k$, $n_k > n_0(\delta,\theta)$, and 
\eqref{eq:lbe:pf1} does not hold for each $n=n_k$, $r=r_k$.
By taking a subsequence, we can assume that $\iota:=\lim_{k\rightarrow \infty}\frac{r_k}{n_k}$ exists.
By Theorem \ref{t:lpptoairy},
we have that by taking $n=n_k$, $r=r_k$ and $k\rightarrow \infty$, the left hand side of \eqref{eq:lbe:pf1} converges to
\begin{multline*}
\P\left[
\sup_{|x|<2^{-2/3}\iota^{2/3}} \sL(x) = \sup_{|x|<2^{-2/3}\theta \iota^{2/3}} \sL(x) {< \sL(0) + 2\alpha^{-1}\cdot 2^{-4/3}\iota^{1/3},}
\right.
\\
\left.
\forall 1 \leq j \leq C_{\delta, \theta},\;\;
\sup_{2^{j-1} 2^{-2/3}\iota^{2/3} \leq |x| < 2^{j} 2^{-2/3}\iota^{2/3}} \sL(x)
< \sup_{|x|<2^{-2/3}\iota^{2/3}} \sL(x)- 2\alpha\cdot 2^{j\left(\frac{1}{2}-\tau\right)} 2^{-4/3}\iota^{1/3}
\right],
\end{multline*}
while the right hand side of \eqref{eq:lbe:pf1} converges to half of this. By Proposition \ref{prop:lbae}, the right hand side of \eqref{eq:lbe:pf1} is at least $\alpha \theta r^{2/3}(2n)^{-2/3}> 2^{-2/3}\alpha \theta \delta^{2/3}>0$ and thus we get a contradiction.

Next we consider $H_j$ for $j > C_{\delta, \theta}$.
If $H_j$ does not hold, clearly either
\begin{equation}   \label{eq:lbe:pf5}
\max_{2^{j-1} r^{2/3} \leq |u| < 2^{j} r^{2/3}} T_{(u,-u), \bn} \geq 4n-2^jr^{1/3}
\end{equation}
or
\begin{equation}   \label{eq:lbe:pf6}
\max_{|u| < r^{2/3}} T_{(u,-u), \bn} \leq 4n + 2\alpha\cdot 2^{j\left(\frac{1}{2}-\tau\right)}r^{1/3} -2^jr^{1/3}
\end{equation}
holds.

By Theorem \ref{t:onepoint}, and by lower bounding  $\max_{|u| < r^{2/3}} T_{(u,-u), \bn}$ by $T_{\mathbf{0}, \bn}$, we have that the event \eqref{eq:lbe:pf6} has probability at most $C \exp(-c 2^jr^{1/3}n^{-1/3})$, (in fact it provides a stronger probability bound which we do not use)
for some universal constants $c, C$.
For the event \eqref{eq:lbe:pf5}, we can divide the interval $2^{j-1} r^{2/3} \leq |u| < 2^{j} r^{2/3}$ into $2^{j}r^{2/3}n^{-2/3}$ sub-intervals of length $n^{2/3}$. Provided $2^{j}r^{2/3} \leq n/2$, 
for any $2^{j-1} r^{2/3} \leq |u| < 2^{j} r^{2/3}$ we have using \eqref{e:mean} that
$$
\E T_{(u, -u), \bn} \leq 4n - 2^{2(j-1)}r^{4/3}n^{-1}  + C'n^{1/3}
< 4n - 2\cdot 2^jr^{1/3},
$$
where the second inequality holds by choosing $C_{\delta, \theta}$ large enough.
Then Theorem \ref{t:supinf} applies again for each of the subintervals and we get an upper bound of $C 2^jr^{2/3}n^{-2/3} \exp(-c 2^jr^{1/3}n^{-1/3})$,
for some universal constants $c, C$.
On the other hand if $2^{j}r^{2/3}\geq n/2,$ the crude bound in \eqref{e:steep} and an union bound over all points yield an upper bound of $C2^{j}r^{2/3}\exp(-c 2^jr^{1/3}n^{-1/2})$.
Combining all of these and using $r>\delta n$ we get that 
$$
\sum_{j > C_{\delta, \theta}} \P[H_j^c] < \sum_{j > C_{\delta, \theta}}
2C 2^j \exp(-c 2^j\delta^{1/3})+ \sum_{j:2^{j}r^{2/3}\geq n/2} C2^{j}r^{2/3}\exp(-c 2^jr^{1/3}n^{-1/2}) < \alpha \theta \delta^{2/3}/2,
$$
where the last inequality holds by taking $C_{\delta, \theta}$ large enough and $n$ sufficiently large.
Thus our conclusion follows by letting $c_0 := (2^{-2/3}-2^{-1})\alpha$.
\end{proof}

We end with a discussion of the proof of Theorem \ref{t:lpptoairy}. As this has appeared in the literature before, we shall not provide a complete proof, instead sketching how to obtain the finite dimensional convergence, and then the necessary equi-continuity to establish uniform convergence. 

\begin{proof}[Proof of Theorem \ref{t:lpptoairy}] 
Finite dimensional convergence of the (appropriately scaled) TASEP height functions is well known (see e.g.\ \cite{BF08}). In the language of exponential LPP this translates to the following: for any $x_1, \cdots x_k, h_1, \cdots, h_k \in \R$, we have that as $n\rightarrow \infty$,
\begin{equation*}
\P\left[
T_{(\lfloor x(2n)^{2/3}\rfloor+\lfloor 2^{-2/3}n^{1/3}h \rfloor, -\lfloor x(2n)^{2/3}\rfloor+\lfloor 2^{-2/3}n^{1/3}h \rfloor) , \bn} < 4n,\;
\forall 1 \leq i \leq k \right]
\rightarrow \P\left[\sL(x_i) < h_i, \forall 1 \leq i \leq k \right] .
\end{equation*}
It is also standard that the above equation, using the phenomenon of so-called \emph{slow decorrelation} can be used to establish the finite dimensional convergence of $\sL_n$ to $\sL$ (see, e.g.\ \cite{FO17}). Indeed, it can be proved that
for any fixed $h, x \in \R$, as $n\rightarrow \infty$,
\begin{equation}  
\label{e:lpptoairy1}
\left|T_{(\lfloor x(2n)^{2/3}\rfloor+\lfloor 2^{-2/3}n^{1/3}h \rfloor, -\lfloor x(2n)^{2/3}\rfloor+\lfloor 2^{-2/3}n^{1/3}h \rfloor) , \bn} + 2^{4/3}n^{1/3}h
- T_{(\lfloor x(2n)^{2/3}\rfloor, -\lfloor x(2n)^{2/3}\rfloor), \bn}  \right| {\rightarrow} 0
\end{equation}
in probability. Clearly this suffices for the finite dimensional convergence. We shall omit the proof of \eqref{e:lpptoairy1}.

To upgrade to weak convergence in the uniform convergence topology, it remains show equicontinuity of $\sL_n$, i.e.
given any $M, \varepsilon, \lambda>0$, there is $\delta>0$ \footnote{This is a local use of the symbol $\delta$ and should not be confused with the same symbol used in the main theorem statements.} and $n_0 \in \Z_+$, such that
\begin{equation} \label{e:lpptoairy2}
\P\left[\sup_{|x_1|, |x_2| < M, |x_1-x_2|<\delta} |\sL_n(x_1)-\sL_n(x_2)| > \lambda\right] < \varepsilon
\end{equation}
for any $n>n_0$. 
To prove this we rely on the Brownian type fluctuation upper bounds of the weight profile in exponential LPP  proved in \cite{BG18}. To proceed, we divide $[-M, M]$ into intervals of length $c'\delta^3$.
For each such interval $I$, by \cite[Theorem 3]{BG18}, we have
$$
\P\left[\sup_{x_1, x_2\in I} |\sL_n(x_1)-\sL_n(x_2)| > \lambda\right] \leq C'\exp(-c'\lambda^{4/9}\delta^{-2/9}).
$$
Note that the above tail bounds are sub-optimal and one does expect Gaussian tail behavior as has been established for the pre-limiting model of Brownian LPP in \cite{HHJ+}. However the above bound suffices for our purpose, since the total number of such intervals $I$ is $\lceil M\delta^{-3} \rceil$, the left hand side of \eqref{e:lpptoairy2} can be bounded by $C'\lceil M\delta^{-3} \rceil\exp(-c'\lambda^{4/9}\delta^{-2/9})$, which is made less than $\varepsilon$ by taking $\delta$ small enough. The above equicontinuity, and by standard results (see e.g.  \cite[Theorem 7.1, 7.3]{billingsley1999convergence}), the desired conclusion follows.
\end{proof}

\section{Passage Times across Parallelograms and Transversal Fluctuation} \label{s:appc}
As indicated before, in this appendix we provide the proofs of the estimates on last passage times across parallelograms (Theorem \ref{t:supinf}), and the proof of the fact that paths with large transversal fluctuation are likely to have significantly smaller weights than geodesics (Proposition \ref{l: prep1-tf}). As pointed out before, a version of Theorem \ref{t:supinf} for Poissonian LPP was obtained in \cite{BSS14} where \cite[Proposition 10.1, 10.5, 12.2]{BSS14} are the analogous versions of Theorem \ref{t:supinf} (i), (ii), and (iii) respectively. The proofs there use moderate deviation estimates for the passage time (a weaker version of Theorem \ref{t:onepoint} for Poissonian LPP) and we essentially repeat the arguments in the context of exponential LPP.  However, we have tightened up the calculations therein using the optimal estimates in Theorem \ref{t:onepoint} and hence we get better exponents in our results (optimal ones for parts (i) and (ii)). Proposition \ref{l: prep1-tf} has not appeared before in the form stated, but the proof uses the same idea as in \cite[Proposition 11.1]{BSS14} involving Theorem \ref{t:supinf} (ii) and a chaining argument. 

The basic structure of the section is as follows: In Section \ref{s.thm4.2(i)} we prove Theorem \ref{t:supinf} (i), while Section \ref{s.thm4.2(ii)} handles Theorem \ref{t:supinf} (ii). Using the latter and Theorem \ref{t:onepoint}, the proof of Proposition \ref{l: prep1-tf} is completed in Section \ref{s.ppn4.7}, and finally Theorem \ref{t:supinf} (iii) is proved in Section \ref{s.thm4.2(iii)}.

\subsection{Minimum passage time in a parallelogram}\label{s.thm4.2(i)}
In the course of proving Theorem \ref{t:supinf}(i), we shall first prove a weaker version. We start with a few definitions first. For $m,h$ such that $|m|+h<\psi r^{1/3}$, let $U_0$ be the parallelogram whose one pair of opposite sides are parallel to the line $x+y=0$, have midpoints $(mr^{2/3},-mr^{2/3})$ and $\br$ respectively and length $2hr^{2/3}$. Let $\hat{U}$ be the sub-parallelogram of $U_0$ restricted to the strip $\{0\leq x+y\leq \frac{r}{16}\}$. The main technical work goes into proving the following result. 
\begin{lemma}
\label{l:p2sidethick}
For each $\psi<1$ and $h>0$, there exists $C,c>0$ depending only on $\psi, h$ such that for $\hat{U}$ as above with $|m|+h\leq \psi r^{1/3}$, we have for all $x>0$ and all $r\geq 1$
$$\P\left( \inf_{u\in \hat{U}}  (T_{u,\br}-\E T_{u,\br}) \leq -xr^{1/3}\right)\leq Ce^{-cx^3}.$$
\end{lemma}
We shall come back to the proof of Lemma \ref{l:p2sidethick} at the end of this subsection; first let us show how this leads to Theorem \ref{t:supinf}(i). 
Let $U_*, U_{*,1}$ and $U_{*,2}$ be the sub-parallelograms of $U_0$ restricted to the strips $\{0\leq x+y\leq \frac{9r}{5}\},$ $\{0\leq x+y\leq \frac{9r}{10}\},$ and $\{\frac{11r}{10}\leq x+y\leq 2r\}$ respectively.
We upgrade Lemma \ref{l:p2sidethick} to the following one.
\begin{lemma}
\label{l:p2sidethick-upgraded}
For each $\psi<1$ and $h>0$, there exists $C,c>0$ depending only on $\psi, h$ such that for $U_*$ as above with $|m|+{10}h\leq \psi r^{1/3}$, we have for all $x>0$ and all $r\geq 1$
$$\P\left( \inf_{u\in U_*}  (T_{u,\br}-\E T_{u,\br}) \leq -xr^{1/3}\right)\leq Ce^{-cx^3}.$$
\end{lemma}
The constraint  $|m|+{10}h\leq \psi r^{1/3}$ will allow us to apply Lemma \ref{l:p2sidethick}.
\begin{proof}
For each $0\le i \le 359,$ we denote
$\hat{U}_i$ as the subparallelogram of $U_0$ restricted to the strip $\{\frac{ir}{200} \le x+y \le \frac{(i+1)r}{200}\}$ (the choice of $360$ is guided by the fact that we are tiling $U_*$ by parallelograms of height $\frac{r}{200},$ and $360=(9/5)200,$ where the $9/5$ factor appears in the definition of $U_*,$ while the latter choice of $200$ is somewhat arbitrary and can be replaced by all large enough numbers).

Now to apply Lemma \ref{l:p2sidethick} we define $U_i$ as the intersection of $U_0$ and the region $\{x+y\ge\frac{ir}{200}\}$. Then $\hat{U}_i$ is a subset of $U_i$, sharing the bottom face and having height at most $1/32$th of the latter. Thus we could apply Lemma \ref{l:p2sidethick}, with $U_i$ in the place of $U_0$ and $\hat{U}_i$ in the place of $\hat{U}$. Also note that the necessary slope condition is satisfied by the hypothesis $|m|+10h\le \psi r^{1/3}$, since one pair of opposite sides of $U_i$ have length $2hr^{2/3}$ and midpoints  $\frac{i}{400}\br +\frac{(400-i)}{400}  (mr^{2/3},-mr^{2/3})$ and $\br$ respectively. The proof is now completed by taking a union bound over $i$.
\end{proof}

Applying Lemma \ref{l:p2sidethick-upgraded} twice we get the following result.
\begin{lemma}
\label{l:side2sidethick}
For each $\psi<1$ and $h>0$, there exists $C,c>0$ depending only on $\psi, h$ such that for $U_{*,1}, U_{*,2}$ as above with $|m|+20h\leq \psi r^{1/3}$, we have for all $x>0$ and $r\geq 1$
$$\P\left( \inf_{u\in U_{*,1}, v\in U_{*,2}} (T_{u,v}-\E T_{u,v}) \leq -xr^{1/3}\right)\leq Ce^{-cx^3}.$$
\end{lemma}
\begin{proof}
Let $w=(w_1,w_2)=\frac{\br+(mr^{2/3}, -mr^{2/3})}{2}$, which is the center of $U_0$.
Consider the events 
$$\cA_1=\left\{\inf_{u\in U_{*,1}} (T_{u,w}-\E T_{u,w}) \geq -\frac{x}{4}r^{1/3}\right\};$$
$$\cA_2=\left\{\inf_{v\in U_{*,2}} (T_{w,v}-\E T_{w,v}) \geq -\frac{x}{4}r^{1/3}\right\}.$$

Simple algebra now shows that for all $u=(u_1,u_2)\in U_{*,1}$ and $v=(v_1,v_2)\in U_{*,2}$ we have 
$$0 \leq (\sqrt{v_1-u_1}+\sqrt{v_2-u_2})^2 - (\sqrt{w_1-u_1}+\sqrt{w_2-u_2})^2 - (\sqrt{v_1-w_1}+\sqrt{v_2-w_2})^2 \leq C_1 r^{1/3}$$
for some constant $C_1$ depending on $\psi, h$. It follows from \eqref{e:mean} that for $x$ sufficiently large we have for all $u\in U_{*,1}$ and $v\in U_{*,2}$
$$|\E T_{u,w} + \E T_{w,v} - \E T_{u,v}|\leq \frac{x}{2}r^{1/3}$$
and hence on $\cA_1\cap \cA_2$ we have $\inf_{u\in U_{*,1}, v\in U_{*,2}} (T_{u,v}-\E T_{u,v}) \geq -xr^{1/3}$. The proof is completed by using a union bound and Lemma \ref{l:p2sidethick-upgraded} to upper bound $\P(\cA_1^c\cup \cA_2^{c})$.
\end{proof}

Using Lemma \ref{l:side2sidethick}, we can now complete the proof of Theorem \ref{t:supinf}(i). Recall the parallelogram $U$ from the statement of the theorem whose one pair of sides lie on $\L_0$ and $\L_r$ with length $2r^{2/3}$ and midpoints $(mr^{2/3},-mr^{2/3})$ and $\br$ respectively. Let $U'$ be the parallelogram whose one pair of sides also lie on $\L_0$ and $\L_r$ with midpoints $(mr^{2/3},-mr^{2/3})$ and $\br$ respectively, but the side length is now $20r^{2/3}$. For notational convenience, let us first define the following. Consider a parallelogram which is a translate (in $\Z^2$) of a parallelogram which has, for some $r\in \N$, a pair of opposite sides along $\L_0$ and $\L_r$ with length at most $2r^{2/3}$ and midpoints $(mr^{2/3},-mr^{2/3})$ and $\br$ respectively. If $|m|+20<\psi r^{1/3}$ for some $\psi<1$ we say that such a parallelogram satisfied the $\psi$-slope condition. 

Without loss of generality we shall assume that $r$ and $r/L$ are sufficiently large powers of $2$, and omit the floor and ceiling signs in the following argument. For $i\in \llbracket 0, 2r \rrbracket$,
let $A_{i}$ denote the line segment of the intersection of the line $x+y=i$ and $U'$.
For $j\in \N$, and $i_1, i_2\in \llbracket 0, 2r \rrbracket$, with $i_1<i_2$,
let us divide the lines segments $A_{i_1}$ and $A_{i_2}$ into $10\cdot2^{2j/3}$ equal sub-segments of length $2(2^{-j}r)^{2/3}$. Consider all possible parallelograms formed by taking one pair of opposite sides from these sub-segments, one on $A_{i_1}$ the other on $A_{i_2}$. Call this family of parallelograms $\mathcal{P}_{i_1,i_2,j}$. We next record the following  straightforward deterministic facts.

\begin{observation}
\label{o:para1}
We have the following:
\begin{enumerate}
\item[(i)] If $|m|<\psi r^{1/3}$ for some $\psi<1$ then for any $\psi<\psi'<1,$ and each $j\in \N$ and each $r$ sufficiently large (compared to $2^{j}$ and $\psi$),
and each $i_1, i_2 \in \llbracket 0, 2r \rrbracket$ with $i_2-i_1 \geq 2^{-j}(2r)$, 
every parallelogram in $\mathcal{P}_{i_1,i_2,j}$ satisfies the $\psi'$-slope condition.
\item[(ii)] For $j\in \N$, and $i_1,i_2\in \llbracket 0, 2r \rrbracket$ with $i_1<i_2$, let $u,v\in U$ be such that $i_1\leq d(u)\leq i_1+\frac{9(i_2-i_1)}{20}$ and $i_1+\frac{11 (i_2-i_1)}{20}\leq d(v)\leq i_2$. Then for $r$ sufficiently large there exists a parallelogram in $\mathcal{P}_{i_1,i_2,j}$ containing both $u$ and $v$.
\end{enumerate}
\end{observation}

While the first observation is rather simple to verify, 
the second observation follows from the fact that the straight line passing through such $u,v$ intersects $A_{i_1}$ and $A_{i_2}$.

We can now complete the proof of Theorem \ref{t:supinf}(i).
\begin{proof}[Proof of Theorem \ref{t:supinf}(i)]
For $j=1,2,\ldots  \log_2 L + 1$, let us consider the family of parallelograms $\mathcal{P}_{i_1,i_2,j}$ for $i_1, i_2\in \{0, 2^{-j}(2r), 2\times  2^{-j}(2r), 3\times 2^{-j}(2r), \ldots, 2r\}$,
satisfying $i_2-i_1 \in \{3\times  2^{-j}(2r), 4\times  2^{-j}(2r), 5\times  2^{-j}(2r)\}$. For a fixed $j$, let $\cB_{j}$ denote the event that for any parallelogram $U_0$ as above we have 
$$\inf_{u\in U_{*,1},v\in U_{*,2}} (T_{u,v}-\E T_{u,v}) \geq -xr^{1/3}$$
where $U_{*,1}$ and $U_{*,2}$ are defined as follows: suppose $U_0\in \mathcal{P}_{i_1,i_2,j}$ for some $i_1,i_2$, then $U_{*,1}$ and $U_{*,2}$ are sub-parallelograms of $U_0$ restricted to the strips $\{i_1\le x+y \le \frac{11i_1+9i_2}{20}\}$ and $\{\frac{9i_1+11i_2}{20}\le x+y \le i_2\}$, respectively.
By Observation \ref{o:para1} (i), for $r$ sufficiently large depending on $L$, Lemma \ref{l:side2sidethick} applies; and since $|\mathcal{P}_{i_1,i_2,j}|=O(2^{4j/3})$ it follows that  $\P(\cB_j^c)=O(2^{7j/3}e^{-c2^{j}x^3})$.

Now we consider any $u, v \in U$ with $d(v)-d(u)\geq \frac{r}{L}$.
Suppose that $2^{-j}(2r)\leq d(v)-d(u)<2^{-j+1}(2r)$ for some $j\leq \log_2 L$.
Then we can find some $i_1,i_2 \in \{0, 2^{-j-1}(2r), 2\times  2^{-j-1}(2r), 3\times 2^{-j-1}(2r), \ldots, 2r\}$, with
$i_2-i_1 \in \{3\times 2^{-j-1}(2r), 4\times 2^{-j-1}(2r), 5\times 2^{-j-1}(2r)\}$, 
such that $i_1\le d(u) < i_1+ 2^{-j-1}(2r) < i_2- 2^{-j-1}(2r) \le d(v) < i_2$. 
By Observation \ref{o:para1} (ii), there exists a parallelogram {in $U_0\in \mathcal{P}_{i_1,i_2,j}$ such that $u\in U_{*,1}$ and $v\in U_{*,2}$}. It follows that on $\cap_{j}\cB_j $ we have 
$$\inf_{d(v)-d(u)\geq \frac{r}{L}} (T_{u,v}-\E T_{u,v}) \geq -xr^{1/3}.$$ 
The proof is completed by taking a union bound over $\cB_j^c$. 
\end{proof}

\subsubsection{Proof of Lemma \ref{l:p2sidethick}}
As already noted, the proof of Lemma \ref{l:p2sidethick} is rather technical. For ease of exposition we first present a weaker version.
For $|m|+h<\psi r^{1/3}$, let $A'$ denote the line segment of length $2hr^{2/3}$ on $\L_0$ with midpoint at $(mr^{2/3},-mr^{2/3})$ For $|m|+h<\psi r^{1/3}$. Then we have the following.

\begin{lemma}
\label{l:infpointtoside}
For each $\psi<1$ and $h>0$, there exists $C,c>0$ depending only on $\psi, h$ such that for $A'$ as above, we have for all $x>0$ and all $r\geq 1$
$$\P\left( \inf_{u\in A'}  (T_{u,\br}-\E T_{u,\br}) \leq -xr^{1/3}\right)\leq Ce^{-cx^3}.$$
\end{lemma}

We are going to present a full proof of Lemma \ref{l:p2sidethick} and hence, for notational convenience, we shall write the proof of Lemma \ref{l:infpointtoside} only for the special case $h=1$ and $m=0$. The reader will notice that the same proof will apply to the general case with {minor adjustments}. Before proceeding with the proof, let us present the basic idea. We shall construct a tree $\mathcal{T}$ whose vertices are a subset of vertices of $\Z^2$; in particular root of $\mathcal{T}$ will be the vertex $\br$ and the leaves of $\mathcal{T}$ are close to vertices on $A'$. The tree will be constructed such that if $T_{u,v}-\E T_{u,v}$ is not too small for every edge $(u,v)\in \mathcal{T}$ then we shall have 
$\inf_{u\in A'}  (T_{u,\br}-\E T_{u,\br}) \geq -xr^{1/3}$. Taking a union bound of the complements of the above events over all edges of $\mathcal{T}$ will then yield the result.  

Let us now formally construct the tree $\mathcal{T}$. Let $r$ be sufficiently large so that there exists $J$ such that $r^{1/4}< 8^{-J}(2r)\leq r^{1/3}$.
For smaller $r$ the lemma follows by taking $C$ large and $c$ small enough.
We shall be ignoring the rounding issues for notational convenience. For $j=0,1,2,\ldots, J$, there will be $4^{j}$ vertices of $\mathcal{T}$ at level $j$ (let us denote this set by $\mathcal{T}_j$) on the line $x+y=8^{-j}(2r)$, such that these $4^{j}$ vertices divide the line joining $8^{-j}\br +(-r^{2/3},r^{2/3})$ and $8^{-j}\br -(-r^{2/3},r^{2/3})$ into $4^{J}+1$ equal length intervals. Notice that, for each $j$, the vertices in $\mathcal{T}_j$ are ordered naturally from left to right. The vertex set of $\mathcal{T}$ is $\cup_{0\leq j\leq J} \mathcal{T}_j$, and the $k$-th vertex at level $j$ from the left is connected to the four vertices in level $(j+1)$ which are labelled $4k-3,4k-2,4k-1$ and $4k$ from the left. 

Recall that for any $u=(u_x,u_{y})\in \Z^2$, we denote $d(u)=u_{x}+u_{y}$. It would be also convenient to  let  ${\rm{ad}}(u)=u_{x}-u_{y}$ (where $\rm{ad}$ is used to denote the anti-diagonal deviation of $u$).
By the construction of $\mathcal{T}$ we have for each $j\leq J$ 
\begin{align}\label{treeprop11}
{d(u_{j})-d(u_{j+1})}&=\frac{14r}{8^{j+1}},\\
\label{treeprop12}
| {\rm{ad}}(u_{j+1})-{\rm{ad}}(u_j)|&\leq C_1\frac{r^{2/3}}{4^{j}} \text{ for some }C_1>0.
\end{align}

Noticing that it suffices to prove Lemma \ref{l:infpointtoside} for $x$ sufficiently large,
let $\cA_j$ denote the event that for all $u\in \mathcal{T}_j$ and for all $v\in \mathcal{T}_{j+1}$ such that the edge $(u,v)$ is present in $\mathcal{T},$ we have $$T_{v,u}\geq \E T_{v,u}- x2^{-(9j/10+10)}r^{1/3}.$$ We first have the following lemma. 
\begin{lemma}
\label{l:infprobest}
In the above set-up, there exists $C,c>0$ such that for all $x$ sufficiently large
$$\P(\cup_{j}\cA_j^c) \leq Ce^{-cx^3}.$$
\end{lemma}
\begin{proof}
Notice that, by our construction of $\mathcal{T}$, for each edge between a vertex  $u\in \mathcal{T}_j$ and a vertex $v\in \mathcal{T}_{j+1}$, Theorem \ref{t:onepoint} applies to $T_{v,u}$ (where the slope condition is satisfied for all large $r$ by \eqref{treeprop11}, \eqref{treeprop12}) and hence we have that 
$$\P(T_{v,u}-\E T_{v,u}\leq -y (8^{-j}r)^{1/3})\leq Ce^{-cy^3}$$
for some $C,c>0$ and all $y>0$. Applying this with $y=2^{j/10-10}x$ we obtain that 
$$\P(T_{v,u}-\E T_{v,u}\leq -x{2^{-(9j/10+10)}}r^{1/3})\leq Ce^{-cx^32^{3j/10}}$$
for some $C,c>0$. Now taking a union bound over all $4^{j+1}$ such edges gives that
$$\P(\cA_j^{c})\leq Ce^{-cx^32^{j/10}}$$
for $C,c>0$ and all $j=0,1,2,\ldots, J-1$. Taking another union bound over $j$ completes the proof of the lemma. 
\end{proof}

The proof of Lemma \ref{l:infpointtoside} is completed by using the next lemma. 
\begin{lemma}
\label{l:ptosidelb}
On $\cap_{0\leq j\leq J} \cA_j$, we have $\inf_{u\in A'}  (T_{u,\br}-\E T_{u,\br}) \geq -xr^{1/3}$ for all $x$ sufficiently large. 
\end{lemma}
\begin{proof}
Let us fix $u\in A'$ and let $u_J$ be the vertex in $\mathcal{T}_J$ such that the difference between ${\rm{ad}}(u)$ and ${\rm{ad}}(u_J)$ is smallest. Let $u_J, u_{J-1},\ldots, u_0=\br$ denote the path to $\br$ in $\mathcal{T}$. By our construction of $\mathcal{T}$, $u$ is coordinate-wise smaller than $u_0$ and hence we have 
$$T_{u,\br} \geq \sum_{j=0}^{J-1} T_{u_{j+1},u_j}.$$
By definition, we have that on $\cap_{0\leq j\leq J} \cA_j$ 
$$\sum_{j=0}^{J-1} T_{u_{j+1},u_j}- \E T_{u_{j+1},u_j} \geq -\frac{x}{2} r^{1/3},$$
for $x$ sufficiently large. Observe also that by our definition of $\mathcal{T}$ and \eqref{e:mean} we have 
\begin{equation}\label{treeapprox}\E T_{u,u_{J}}\leq \frac{x}{4}r^{1/3},
\end{equation}
and hence it suffices to show that 
$$ \sum_{j=0}^{J} \E T_{u_{j+1},u_j} \geq \E T_{u,\br}-\frac{x}{4}r^{1/3},$$
where we write $u=u_{J+1}$ for convenience of writing. 

Recalling \eqref{treeprop11}, \eqref{treeprop12}, and  \eqref{e:mean} (observe again that it applies to each $T_{u_{j+1},u_j}$) we get

$$\E T_{u_{j+1},u_{j}} \geq 2(d(u_j)-d(u_{j+1}))-C_22^{-j}r^{1/3}$$
for each $j\leq J$. Summing over $j$ from $0$ to $J$, along with the  bound $\E T_{u,\br}\le  4r+O(r^{1/3})$  (which has been used several times already and in particular follows from \eqref{e:mean}), and \eqref{treeapprox} we get that 
$$ \sum_{j=0}^{J} \E T_{u_{j+1},u_j} \geq \E T_{u,\br}-\frac{x}{4}r^{1/3}$$
for $x$ sufficiently large, as required. 
This completes the proof.
\end{proof}

In the above tree construction, each level of the tree made a deterministic progress along the diagonal direction with each vertex splitting into four offsprings spread in the anti-diagonal direction.
To strengthen Lemma \ref{l:infpointtoside} to Lemma \ref{l:p2sidethick}, we similarly construct a tree, which, in addition to branching in the anti-diagonal direction, also branches in the diagonal direction.
For this tree, its leaves are dense in the parallelogram $\hat U$, allowing us to simultaneously lower bound the passage time from each vertex in $\hat U$ to $\br$.
\begin{proof}[Proof of Lemma \ref{l:p2sidethick}]
We assume that $r$ and $x$ are large enough (depending on $\psi, h$), since otherwise the result follows by taking $C$ large and $c$ small enough.
In this proof we use $C$ and $c$ to denote large and small constants depending on $\psi$ and $h$, and the specific values can change from line to line.
Without loss of generality we also assume $m\geq 0$.

Similarly to the proof of Lemma \ref{l:infpointtoside}, we construct a tree $\mathcal{T}$ as follows.
Take $J$ such that $r^{1/4}< 8^{-J}(2r)\leq r^{1/3}$.
Let $\mathcal{T}_0=\{\br\}$.
For each $j=1,2,\ldots, J$, there are $32^{j}$ vertices of $\mathcal{T}$ at level $j$ (denoted as $\mathcal{T}_j$), given as follows.
For each $i=1,2,\ldots, 8^j$, consider the intersection of the line $x+y=\frac{2i+1}{32}8^{-j}(2r)$ with $U_0$;
this is also the line segment joining
$$\frac{2i+1}{32}8^{-j}\br +(m(1-\frac{2i+1}{32}8^{-j})+h)(r^{2/3},-r^{2/3})$$
and
$$\frac{2i+1}{32}8^{-j}\br +(m(1-\frac{2i+1}{32}8^{-j})-h)(r^{2/3},-r^{2/3}).$$
On this line segment there are $4^j$ level $j$ vertices, which divide this line segment into $4^{j}+1$ equal length intervals.
From this construction we see that the sets $\mathcal{T}_0,\ldots,\mathcal{T}_J$ are mutually disjoint, since the lines $x+y=\frac{2i+1}{32}8^{-j}(2r)$ are mutually different for different $i, j$. 
The vertex set of $\mathcal{T}$ is $\cup_{0\leq j\leq J} \mathcal{T}_j$.
We can label the vertices in $\mathcal{T}_{j}$ using $\{(i,k):1\leq i \leq 8^j, 1\leq k \leq 4^j\}$,
for $i$ indexing the lines and $k$ indexing vertices in a line from left to right. 
For $0\leq j < J$, the vertex in $\mathcal{T}_{j}$ labelled $(i,k)$ is connected to $32$ vertices in level $(j+1)$, which are labelled
$\{(8i-i',4k-k'):0\leq i' \leq 7, 0\leq k' \leq 3\}$.
Then each vertex in $\mathcal{T}_{j+1}$ is connected to exactly one vertex in $\mathcal{T}_j$, and the graph we construct is a tree.

We record now statements analogous to \eqref{treeprop11}, \eqref{treeprop12}. 
Let us set $\rho=mr^{-1/3}$ which will parametrize the role of the slope in the subsequent calculations. For each $j< J$ 
\begin{align}
\label{treeprop21}
d(u_j)-d(u_{j+1}) = 2r 8^{-j-1}(7+2i')/32, &\text{ for some } i'=0,..., 7\\
\label{treeprop22}
|({\rm{ad}}(u_{j})-{\rm{ad}}(u_{j+1})) + \rho(d(u_{j})-d(u_{j+1}))|&\leq C\frac{r^{2/3}}{4^{j}}.
\end{align} 

Since by choice $2r^{2/3}\le 8^{J}\le 2r^{3/4},$
as a consequence, we see that the slope of each tree edge is bounded away from $0$ and $\infty$ uniformly in $m,$ for all large enough $r$ which will allow us to apply Theorem \ref{t:onepoint}.

Again, following the same strategy as in the proof of Lemma \ref{l:infpointtoside}, we let $\cA_j$ denote the event that for all $u\in \mathcal{T}_j$ and for all $v\in \mathcal{T}_{j+1}$ such that the edge $(u,v)$ is present in $\mathcal{T}$, we have $T_{v,u}\geq \E T_{v,u}- x2^{-(9j/10+10)}r^{1/3}$.
As in Lemma \ref{l:infprobest}, we show that 
$$\P(\cup_{j}\cA_j^c) \leq Ce^{-cx^3}.$$
Indeed, by our construction of $\mathcal{T}$, for each edge between a vertex  $u\in \mathcal{T}_j$ and a vertex $v\in \mathcal{T}_{j+1}$, by Theorem \ref{t:onepoint} applied to $T_{v,u}$ we have for all $y>0$,
$$\P(T_{v,u}-\E T_{v,u}\leq -y (8^{-j}r)^{1/3})\leq Ce^{-cy^3}.$$
Using this with $y=2^{j/10-10}x$ we obtain that 
$$\P(T_{v,u}-\E T_{v,u}\leq -x{2^{-(9j/10+10)}}r^{1/3})\leq Ce^{-cx^32^{3j/10}}.$$
Now taking a union bound over all $32^{j+1}$ such edges, and then over all $j=0,1,2,\ldots, J-1$, we get $\P(\cup_{j}\cA_j^c) \leq Ce^{-cx^3}$.

It remains to show that on the event $\cap_{0\leq j\leq J} \cA_j$, we must have $\inf_{u\in \hat{U}}  (T_{u,\br}-\E T_{u,\br}) \geq -xr^{1/3}$.
Now let's take any $u\in \hat{U}$,
and take $u_J$ such that 
$$
d(u_J) - \frac{3\cdot 8^{-J}(2r)}{32} \leq d(u) \leq d(u_J) - \frac{8^{-J}(2r)}{32}
$$
and
$$| ({\rm{ad}}(u_J)-{\rm{ad}}(u)) + \rho(d(u_{J})-d(u))|\leq C\frac{r^{2/3}}{4^{J}}.$$
Such $u_J$ exists by our construction of the tree. It is evident from the above two displays that for $r$ large enough, $|{\rm{ad}}(u_J)-{\rm{ad}}(u)| < d(u_{J})-d(u)$, so $u_J$ is coordinate-wise greater than $u$.
Let $u_J, u_{J-1},\ldots,u_0=\br$ be the unique path from $u_J$ to $\br$ in $\mathcal{T}$.
We have that
$T_{u,\br} \geq \sum_{j=0}^{J-1} T_{u_{j+1},u_j}$.
On $\cap_{0\leq j\leq J} \cA_j$ we also have that
$$\sum_{j=0}^{J-1} T_{u_{j+1},u_j}- \E T_{u_{j+1},u_j} \geq -\frac{x}{2} r^{1/3}.$$
So we have $T_{u,\br}\geq \sum_{j=0}^{J-1} \E T_{u_{j+1},u_j} -\frac{x}{2}r^{1/3}$, and we just need to show that
\begin{equation}   \label{eq:p2sidethick:pf1}
\sum_{j=0}^{J-1} \E T_{u_{j+1},u_j} \geq \E T_{u,\br}-\frac{x}{2}r^{1/3}.
\end{equation}
As in the the proof of Lemma \ref{l:ptosidelb}, we use \eqref{e:mean} to estimate each expectation to establish \eqref{eq:p2sidethick:pf1}. In particular, we will use the following  bound.
$$
\E T_{u_{j+1},u_{j}} \geq 
(d(u_j)-d(u_{j+1}))(1+\sqrt{1-\rho^2})+{\frac{\rho( ({\rm{ad}}(u_{j})-{\rm{ad}}(u_{j+1})) + \rho(d(u_{j})-d(u_{j+1})) )}{\sqrt{1-\rho^2}}}-C 2^{-j}r^{1/3}.
$$

This follows from \eqref{e:mean} and the following calculation:
For each $v\in \Z^2$ we have $v=\frac{1}{2}(d(v)+{\rm{ad}}(v), d(v)-{\rm{ad}}(v))$.
When $d(v)>0$, and $\frac{|{\rm{ad}}(v)|}{d(v)}<(\psi+1)/2<1$,
we have
\begin{equation}  \label{eq:p2sidethick:pf2}
\begin{split}
\frac{1}{2} \left(\sqrt{d(v)+{\rm{ad}}(v)}+\sqrt{d(v)-{\rm{ad}}(v)} \right)^2 &= d(v) + \sqrt{d(v)^2-{\rm{ad}}^2(v)}
\\
&=d(v)(1+\sqrt{1-\rho^2}) + \frac{\rho({\rm{ad}}(v) +\rho d(v) )}{\sqrt{1-\rho^2}} + E(v),
\end{split}    
\end{equation}
where $E(v)$ is an error term satisfying $|E(v)|<C\frac{({\rm{ad}}(v)+\rho d(v) )^2}{d(v)}$.
For {$v=u_{j}-u_{j+1}$,} and all large $r$, by \eqref{treeprop21} and \eqref{treeprop22} we get the bound on the error term which gives the sought lower bound. 

Summing over $j$ from $0$ to $J-1$, we get 
$$\sum_{j=0}^{J-1} \E T_{u_{j+1},u_j} \geq 
(2r-d(u_J))(1+\sqrt{1-\rho^2})+\frac{\rho( -{\rm{ad}}(u_{J}) + \rho(2r-d(u_J)))}{\sqrt{1-\rho^2}} -
C r^{1/3}.$$

On the other hand, using \eqref{e:mean} and \eqref{eq:p2sidethick:pf2} for $T_{u,\br}$, we have
$$
\E T_{u,\br} \leq 
(2r-d(u))(1+\sqrt{1-\rho^2})+\frac{\rho (-{\rm{ad}}(u)+ \rho(2r-d(u)) )}{\sqrt{1-\rho^2}}+C r^{1/3}.
$$
By our choice of $u_J$ we have $d(u_J)-d(u)\leq Cr^{1/3}$ and 
$| ({\rm{ad}}(u_J)-{\rm{ad}}(u)) + \rho(d(u_{J})-d(u))|\leq C\frac{r^{2/3}}{4^{J}}$.
Thus we get \eqref{eq:p2sidethick:pf1} as $x$ is large enough which completes the proof.
\end{proof}

\subsection{Maximum Passage time in a parallelogram}\label{s.thm4.2(ii)}
\begin{proof}[Proof of Theorem \ref{t:supinf}(ii)]
Observe first that it suffices to prove the result for $x$ sufficiently large. 
Denote $u_- = -\br + 2(mr^{2/3},-mr^{2/3})$ and $u_+ = 2\br - (mr^{2/3},-mr^{2/3})$, {i.e., these are the points where the straight line joining $\br$ and $(mr^{2/3},-mr^{2/3})$ intersects $\L_{-r}$ and $\L_{2r}$ respectively}.
Consider the following events:
$$\cA_1=\{\inf_{u\in U_1} T_{u_-,u}-\E T_{u_-,u} \geq -\frac{xr^{1/3}}{10}\};$$
$$\cA_2=\{\inf_{v\in U_2} T_{v,u_+}-\E T_{v,u_+} \geq -\frac{xr^{1/3}}{10}\};$$
$$\cA_3=\{\sup_{u\in U_1,v\in U_2}  T_{u,v}-\E T_{u,v} \geq xr^{1/3}\}.$$
It follows from \eqref{e:mean} that for $x$ sufficiently large we have for any $u\in U_1$ and $v\in U_2$
\begin{equation}
\label{e:meancomp}
\E T_{u_-,u}+\E T_{u,v}+ \E T_{v,u_+} \geq \E T_{u_-,u_+} -\frac{xr^{1/3}}{10}.
\end{equation}
It therefore follows that 
$\cA\supseteq \cA_1\cap \cA_2 \cap \cA_3$ where 
$$\cA=\{T_{u_-,u_+}-\E T_{u_-,u_+} \geq \frac{xr^{1/3}}{2}\}.$$
Since $\cA_1,\cA_2,\cA_3$ are all increasing events, it follows by the FKG inequality that 
$$\P(\cA)\geq \P(\cA_1\cap \cA_2\cap \cA_3)\geq \P(\cA_1)\P(\cA_2)\P(\cA_3)$$
by the FKG inequality. The result follows by noting that we have $\P(\cA_1), \P(\cA_2)\geq \frac{1}{2}$ for $x$ sufficiently large by Lemma \ref{l:p2sidethick-upgraded}, and $\P(\cA)\leq Ce^{-c\min\{x^{3/2},xr^{1/3}\}}$ by Theorem \ref{t:onepoint}.
\end{proof}

\subsection{Transversal Fluctuation Estimates}\label{s.ppn4.7}
Using Theorem \ref{t:onepoint} and Theorem \ref{t:supinf}(ii) one can show that paths with large transversal fluctuations are likely to have significantly smaller weights than geodesics and hence geodesics are unlikely to have large transversal fluctuation.

We state the following result, which is a stronger variant Proposition \ref{l: prep1-tf} dealing with general slopes.
Let $A'$ be a segment of length $2r^{2/3}$ on $\L_0$ with midpoint $(mr^{2/3},-mr^{2/3})$.
Let $U_{m,\phi}$ be the parallelogram whose one pair of opposite sides of length $\phi r^{2/3}$ lie on the lines $\L_0$ and $\L_r$ respectively with respective midpoints $(mr^{2/3},-mr^{2/3})$ and $\br$. 
Denoted by $\mathrm{LargeTF}{(m, \phi, r)}$ the event where there exists a path $\gamma$ from $u\in A'$ to $v \in \L_{r,r^{2/3}}$ which exits $U_{m,\phi}$ and has $\ell(\gamma)\ge \E T_{u,v}-c_1\phi^2 r^{1/3}$.
\begin{proposition}
\label{l: prep1-tf-s}
For each $\psi<1$, there exist constants $c_1,c_2>0$ such that for all $|m|\leq \psi r^{1/3}$, sufficiently large $\phi$, and all $r\geq 1$,
$$\P(\mathrm{LargeTF}{(m, \phi, r)})\leq e^{-c_2\phi^3}.$$ 
\end{proposition}
From this proposition together with Theorem \ref{t:onepoint}, we can immediately deduce the following result, which is essentially \cite[Proposition 11.1]{BSS14} and will be used in proving Theorem \ref{t:supinf}(iii).
\begin{proposition}
\label{p:tf}
Let $\cA_\phi$ denote the event that the geodesic from $(mr^{2/3},-mr^{2/3})$ to $\br$ exits $U_{m,\phi}$.
For each $\psi<1$, there exist $C,c>0$ such that for all $|m|\leq \psi r^{1/3}$ and $\phi>0$, $r\geq 1$,
$$\P(\cA_{\phi})\leq Ce^{-c\phi^3}.$$
\end{proposition}
\begin{proof}
It is immediate that for $c_1$ as in Proposition \ref{l: prep1-tf-s} we have
$$ \P(\cA_{\phi})\leq \P(\mathrm{LargeTF}{(m, \phi, r)})+ \P\left(T_{(mr^{2/3},-mr^{2/3}), \br} \leq \E T_{(mr^{2/3},-mr^{2/3}), \br}-c_1\phi^2 r^{1/3}\right).$$
The result is then immediate from Proposition \ref{l: prep1-tf-s} and Theorem \ref{t:onepoint}.
\end{proof}
We now prove Proposition \ref{l: prep1-tf-s} using a chaining argument introduced in \cite{BSS14}.
\begin{proof}[Proof of Proposition \ref{l: prep1-tf-s}] 
We shall assume that $r$ is large (depending on $\psi$), since otherwise the conclusion obviously holds. 
In this proof we use $c$ to denote a small constant depending on $\psi$, and the value can change from line to line.

Let $\{a_j\}_{j\geq 2}$ be a sequence of positive real constants going to $0$ satisfying $\prod_{j=2}^{\infty} (1+a_j) < \infty$, and to be chosen appropriately later.
Let $h_1= \frac{1}{2\prod_{j=2}^{\infty} (1+a_j)}$, and for $j>1$ let us set $h_j=h_{j-1}(1+a_{j})$.

For $j\geq 1$, and for $\ell=1,2,\ldots, 2^{j}-1$, let $\cB_{\ell,j}$ denote the event that there exists a path $\gamma$ from $u \in A'$ to $v \in \L_{r,r^{2/3}}$ with 
$\ell(\gamma)\geq \E T_{u,v} -c_1\phi^2 r^{1/3}$ which intersects line $x+y= \ell 2^{-j} (2r)$ outside the parallelogram $U_{m,h_j\phi}$. Let us set $\cB_{j}=\cup_{\ell=1}^{2^j-1} \cB_{\ell,j}$. 
Let $j_0:=\log_{2}(r^{1/3}).$ To avoid rounding issues we will assume $r$ is a power of $8$. To treat a general $8^j\le r< 8^{j+1},$ one can instead use the result for $r'=8^{j+2},$ along with the event that no geodesic from $\br'$ to a point in $\L_{r,r^{2/3}}$ has a weight deficit from mean of order $\phi^{2}r^{1/3}$, which occurs with probability $1-e^{-c\phi^6}$ by Theorem \ref{t:supinf} (i). It can be checked that the intersection of the latter event and $\mathrm{LargeTF}{(m, \phi, r)}$  with $c_1$ in the definition of the event replaced by $c_1/64$, which are independent, implies $\mathrm{LargeTF}{(m', \frac{\phi}{16}, r')}$ where $m'=m(\frac{r}{r'})^{2/3}$, allowing the upper bound on the latter to yield a similar upper bound for $\mathrm{LargeTF}{(m, \phi, r)}$ with $c_1/64$ in place of $c_1$. Above $r'$ is chosen to be $8^{j+2}$ and not $8^{j+1}$ to ensure enough room between $\br$ and $\br'$. Thus throughout the remaining proof, $r$ will be a power of $8.$

It is clear from our definition of $a_j$'s that $h_{j_0}\leq \frac{1}{2}$.
By the directed nature of the paths, for $\phi$ large, on the event $\cap_{j=1}^{j_0}\cB^c_{j}$ every path $\gamma$ from $u\in A'$ to $v\in \L_{r,r^{2/3}}$  that exits $U_{m,\phi}$ satisfies $\ell(\gamma)< \E T_{u,v} -c_1\phi^2 r^{1/3}$; i.e. we have $\mathrm{LargeTF}{(m, \phi, r)} \subset \cup_{j=1}^{j_0}\cB_{j}.$
It is now immediate that Proposition \ref{l: prep1-tf-s} will follow from the next two lemmas.
\end{proof}

\begin{lemma}
\label{l:0}
In the above set-up  we have $\P(\cB_1)\leq e^{-c\phi^3}$ for all $\phi$ sufficiently large. 
\end{lemma}

\begin{lemma}
\label{l:j}
In the above set up (for some appropriate choice of $a_j$), for $j_0\geq j\geq 2$, and $1\leq \ell < 2^j$, we have $\P(\cB_{\ell,j}\cap \cB_{j-1}^{c})\leq 4^{-j}e^{-c\phi^3}$ when $\phi$ is sufficiently large.
\end{lemma}

Both Lemma \ref{l:0} and Lemma \ref{l:j} follow from the following general technical result. Let $A_-$ and $A_+$ be segments in $\L_{-r}, \L_r$ with length $2r^{2/3}$, and centers $-\br + (mr^{2/3}, -mr^{2/3})$ and $\br + (-mr^{2/3}, mr^{2/3})$. Let $B=\L_0\setminus \{(u_1,-u_1): |u_1|<tr^{2/3}\}$. We have the following lemma.

\begin{lemma}
\label{l:transversal-general}
For each $\psi<1$
there exist $C,c>0$ such that when $|m|\leq \psi r^{1/3}$, $t>0$ and $r\geq 1$ we have 
$$\P(\sup_{u\in A_-, v\in B, w\in A_+} T_{u,v} + T_{v,w} - \E T_{u,w} \geq -ct^{2}r^{1/3})\leq Ce^{-ct^3}.$$
\end{lemma}
For the proof we would need to assume a certain largeness of $t,$ which then can made into any $t$ by choosing $c$ small enough. 
\begin{proof}
Let $B_{+}$ denote the part of $B$ contained in the fourth quadrant. By the directed nature of the model and symmetry, it suffices to prove $$\P(\sup_{u\in A_-, v\in B_+, w\in A_+} T_{u,v} + T_{v,w} - \E T_{u,w} \geq -ct^{2}r^{1/3})\leq Ce^{-ct^3}.$$

We will rely on Theorem \ref{t:supinf}(ii). To this end, we first prove the following bounds  for the corresponding expectations.

For $j\geq 0$, let $B_{j}$ denote the line segment joining $((t+j)r^{2/3}, -(t+j)r^{2/3})$ and $((t+j+1)r^{2/3}, -(t+j+1)r^{2/3})$. It follows from \eqref{e:mean} and \eqref{e:steep} that for some $c'>0$ depending on $\psi$, we have
$$\sup _{u\in A_-, v\in B_j, w\in A_+} \E T_{u,v} + \E T_{v,w} - \E T_{u,w} \leq - c'(t+j)^2r^{1/3},$$
for all $j\geq 0$ with $t+j+|m|\leq \frac{(1+\psi)}{2} r^{1/3}$.

Note that for $j$ such that $\frac{(1+\psi)}{2} r^{1/3} < t+j+|m| \leq r^{1/3}+1,$ Theorem \ref{t:supinf}(ii) is not directly applicable, since the needed slope condition is violated by  points in $A_-$ and $B_j$. While, given the slack, there are several ways to address this, what we do is simply  translate the points in $A_{-}$ and $A_+$ by $\mp c_0\br$  for a small $\psi$ dependent constant $c_0.$  Note that this can only increase the passage times and their expectations. Now, while $c_0$ is large enough so that the points in $A_- - {c_0\br}$ and $B_j$ (and similarly for $A_+ +{c_0\br}$) satisfy the slope condition, the following bound on their expectation still holds,

$$\sup _{u\in A_-, v\in B_j, w\in A_+} \E T_{u-{{c_0\br}},v} + \E T_{v,w+{c_0\br}} - \E T_{u,w} \leq - c'(t+j)^2r^{1/3}.$$

Then it follows from Theorem \ref{t:supinf}(ii) (applied to $A_-,B_j$ and $B_j,A_+$, or $A_- - {c_0\br},B_j$ and $B_j,A_+ + {c_0\br}$) that for $j\geq 0$ we have
$$\P\left(\sup _{u\in A_-, v\in B_j, w\in A_+} T_{u,v} + T_{v,w} - \E T_{u,w} \geq -c''t^2r^{1/3}\right) \leq e^{-c''(t+j)^3},$$
for some $c''>0$ depending on $\psi.$
Note that above, in the case $t+j+|m|> \frac{(1+\psi)}{2} r^{1/3}$  we use Theorem  \ref{t:supinf}(ii) to $A_- - {c_0\br},B_j$ and $B_j,A_+ + {c_0\br}$ along with the deterministic bound that for all ${u\in A_-, v\in B_j, w\in A_+}$ 
$$T_{u,v} + T_{v,w} \le T_{u-{c_0\br},v} +  {T_{v,w+{c_0\br}}}.
$$ 
Taking a union bound over $j$ gives the desired result.
\end{proof}

{Lemma \ref{l:0} is an immediate consequence of Lemma \ref{l:transversal-general}}. We now complete the proof of Lemma \ref{l:j}.

Recall from before that $A'$ is a segment of length $2r^{2/3}$ on $\L_0$ with midpoint $(mr^{2/3},-mr^{2/3})$, and $U_{m,\phi}$ is the parallelogram whose one pair of opposite sides of length $\phi r^{2/3}$ lie on the lines $\L_0$ and $\L_r$ respectively with respective midpoints $(mr^{2/3},-mr^{2/3})$ and $\br$. 
\begin{proof}[Proof of Lemma \ref{l:j}]
Fix $2\leq j\leq j_0$. 
It is clear that for even $\ell$, $\cB_{\ell,j}\subseteq \cB_{j-1}$ and hence it is only required to prove the lemma for odd $\ell$. Let us assume that $\ell=2s+1$. 
Let $B_1$ and $B_2$ be the line segments given by the intersection of the lines $x+y=2s 2^{-j}(2r)$ and  $x+y=(2s+2) 2^{-j}(2r)$ with the parallelogram $U_{m,h_{j-1}\phi}$, respectively. 
For any $u, v$ with $d(u)<(2s+1) 2^{-j} 2r < d(v)$, let $T^*_{u,v}$ denote the maximum passage time of a path from $u$ to $v$ that intersects the line $x+y=(2s+1) 2^{-j} 2r$ at a point outside the 
parallelogram $U_{m,h_{j}\phi}$.
We will divide the journey from $A'$ to $\L_{r,r^{2/3}}$  into three parts and accordingly we define the following events where instead of centering by the expectation, we use centering by certain approximations of the expectation that will be convenient to work with. 

For a fixed number $K>0$ depending on $\psi$, let 
\begin{align*}
\mathcal{D}_1&:=\{\sup_{u\in A', v\in B_1} T_{u,v} - S_1(v-u) - S_2(v-u) \geq (c_1\phi^2+Kj)r^{1/3}\};\\
\mathcal{D}_2&:=\{\sup_{u\in B_2, v\in \L_{r,r^{2/3}}} T_{u,v} - S_1(v-u) - S_2(v-u) \geq (c_1\phi^2+Kj)r^{1/3}\};\\
\mathcal{D}_3&:= \{\sup_{u\in B_1, v\in B_2} T^*_{u,v} - S_1(v-u) - S_2(v-u) \geq -(3c_1\phi^2+3Kj)r^{1/3}\},
\end{align*}
where $S_1, S_2:\Z^2 \to \R$ are the following functions: for any $u=(u_1,u_2)\in \Z_{\geq 0}^2$ we have $S_1(u)=(\sqrt{u_1}+\sqrt{u_2})^2$; and $S_2(u)=c_1d(u)$ if $|u_1-u_2| \geq (\psi+1)d(u)/2$, and $S_2(u)=0$ otherwise;
for $u \not\in \Z_{\geq 0}^2$ we set $S_1(u)=S_2(u)=0$.
The idea is that in these definitions of $\mathcal{D}_1, \mathcal{D}_2, \mathcal{D}_3$, the $S_1(v-u)$ term is approximately $\E T_{u,v}$ due to \eqref{e:mean}, and $S_2(v-u)$ is the penalty term when the slope of the line connecting $u, v$ is too large or too small.
Observe that for any $u'\in A', u\in B_1, v\in B_2, v'\in\L_{r,r^{2/3}}$, when $r$ is large (depending on $\psi$) the slope of the line connecting $u'$ and $v'$ is between $\frac{1-\psi}{2}$ and $\frac{2}{1-\psi}$. Thus  when $c_1$ is small enough (depending on $\psi$) we have 
$$
S_1(u-u') + S_1(v-u) + S_1(v'-v) + S_2(u-u') + S_2(v-u) + S_2(v'-v) < S_1(v'-u').
$$
This is a simple consequence of super-additivity of passage times and the concavity of the function $(\sqrt{x}+\sqrt{y})^2$ on the line $x+y=1.$

Now, by \eqref{e:mean},  $|S_1(v'-u')-\E (T_{u',v'})|\le C'r^{1/3}$ where $C'$ depends only on $\psi,$ and hence for $\phi, K$ sufficiently large, we have $\cB_{2s+1,j}\cap \cB^c_{j-1}\subset\mathcal{D}_1 \cup \mathcal{D}_2 \cup \mathcal{D}_3$, so
$$\P(\cB_{2s+1,j}\cap \cB^c_{j-1})\leq \P(\mathcal{D}_1)+\P(\mathcal{D}_2)+\P(\mathcal{D}_3).$$

We next bound $\P(\cD_1)$.
When $s=0$ we have $\P(\cD_1)=0$, so below we assume $s\ge 1$.
We divide each of $A'$ and $B_1$ into $O(2^{2j/3})$ and $O(2^{2j/3}\phi)$ many intervals of length $(s2^{-j}r)^{2/3}$.
At this point we again run into a similar issue as in the proof of Lemma 
\ref{l:transversal-general}, namely having to deal with pairs of points who do not satisfy the slope condition in Theorem \ref{t:supinf}(ii). {To address this we use a similar strategy of backing up to perturb the end points slightly to make the theorem applicable, and in the process only increasing the weights by a tolerable amount.}

For each pair of the intervals $I_1\subset A'$ and $I_2\subset B_1$, if the slope of the line connecting the midpoints of $I_1, I_2$ is between $\frac{1-\psi}{4}$ and $\frac{4}{1-\psi}$, Theorem  \ref{t:supinf}(ii) is applicable.
If the slope of the line connecting the midpoints of $I_1, I_2$ is not between $\frac{1-\psi}{4}$ and $\frac{4}{1-\psi}$, we take $I_1'=I_1-cs2^{-j}\br$. 
This makes the slope of any line connecting any point of $I_1$ with any point of $I_2$  between $c_2$ and $c_2^{-1}$, where $c_2$ is a constant depending only on $c,$ allowing us  to apply Theorem \ref{t:supinf}(ii). Now for any $u\in I_1, v\in I_2$, $T_{u-cs2^{-j}\br, v} \geq T_{u,v}$, and $S_1(v-u)+S_2(v-u)>S_1(v-u+cs2^{-j}\br)$ when $c$ is small enough (depending on $\psi, c_1$).
Together we conclude that
$$
\P\left(
\sup_{u\in I_1, v\in I_2} T_{u,v} - S_1(v-u) - S_2(v-u) \geq (c_1\phi^2+Kj)r^{1/3}
\right) \leq e^{-2c\phi^3-2cKj}.
$$
Then by a union bound of all such intervals, we conclude that if $K$ is sufficiently large, $\P(\cD_1)\leq 8^{-j}e^{-c\phi^3}$; and by symmetry we also have $\P(\cD_2)\leq 8^{-j}e^{-c\phi^3}$. 

Fix such a $K$ and it remains to bound $\P(\mathcal{D}_3)$.
Recall that  by our definition of $h_j$, we have $(h_{j}-h_{j-1})\phi r^{2/3}=a_{j}h_{j-1}2^{2j/3} \phi (2^{-j}r)^{2/3}$. Notice that for small enough $c_1$ and large enough $\phi$ (depending on $K$) and setting $a_{j}=2^{-j/3}$,
$$\dfrac{a^2_{j}h^2_{j-1}2^{4j/3}\phi^2}{(3c_1\phi^2+3Kj)2^{j/3}}$$
can be made arbitrarily large.
We divide each of $B_1$ and $B_2$ into $O(2^{2j/3}\phi)$ many intervals of length $(2^{-j}r)^{2/3}$.
Again as above, for a pair of such intervals $I_1\subset B_1, I_2\subset B_2$, if the slope of the {line connecting their midpoints is between $\frac{1-\psi}{4}$ and $\frac{4}{1-\psi}$}, 
we invoke Lemma \ref{l:transversal-general}. If the slope of the line connecting the midpoints of $I_1, I_2$ is not between $\frac{1-\psi}{4}$ and $\frac{4}{1-\psi}$, we take $I_1'=I_1-c2^{-j}r(1,1)$ and $I_2'=I_2+c2^{-j}r(1,1)$ which ensures that the slope for the midpoints of $I_1'$ and $I_2'$ is bounded away from $0$ and $\infty$; and for any $u\in I_1, v\in I_2$, we have $T_{u-c2^{-j}r(1,1), v+c2^{-j}r(1,1)} \geq T_{u,v}$, and $S_1(v-u)+S_2(v-u)>S_1(v-u+2c2^{-j}r(1,1))$ when $c$ is small enough (depending on $\psi, c_1$).
This allows us to again invoke Lemma \ref{l:transversal-general} for $I_1', I_2'$.
In summary, we get that
$$\P\left(\sup_{u\in I_1, v\in I_2}T^*_{u,v} - S_1(v-u)-S_2(v-u)\geq -(3c_1\phi^2+3Kj)2^{j/3} (2^{-j}r)^{1/3}\right) \leq 2^{-4j/3}8^{-j}e^{-2c\phi^3}.$$
Then by a union bound for all such intervals, we have $\P(\cD_3)\leq 8^{-j}e^{-c\phi^3}$ when $\phi$ is large.
This completes the proof by summing up $\P(\mathcal{D}_1)$, $\P(\mathcal{D}_2)$ and $\P(\mathcal{D}_3)$.
\end{proof}

Finally we deduce Proposition \ref{l: prep1-tf} from Proposition \ref{l: prep1-tf-s}.
\begin{proof}[Proof of Proposition \ref{l: prep1-tf}] 
Let $\phi$ be large.
Applying  Proposition \ref{l: prep1-tf-s} and union bound allows us to handle all $\{(u_1,-u_1): |u_1|<\frac{\phi}{8} r^{2/3}\}$.  The remaining cases will be handled if we show that for some small constant $c$,
$$\P\left(\sup_{u\in \L_0\setminus \{(u_1,-u_1): |u_1|<\frac{\phi}{8} r^{2/3}\}, v\in \L_{r,r^{2/3}}} T_{u,v}\geq 4r-c\phi^{2}r^{1/3}\right)\leq e^{-c\phi^3}.$$
For $j\geq 0$, let $B_{j}$ denote the line segment joining $((\frac{\phi}{8}+j)r^{2/3}, -(\frac{\phi}{8}+j)r^{2/3})$ and $((\frac{\phi}{8}+j+1)r^{2/3}, -(\frac{\phi}{8}+j+1)r^{2/3})$. To handle slope issues, as multiple times done, 
when $\frac{\phi}{8}+j> 0.9 r^{1/3}$ we also denote $B_j'=B_j-(c_0r,0)$, for some small enough $c_0$.
It follows from \eqref{e:mean} that for some constant $c'>0$,
$$\sup _{u\in B_j, v\in \L_{r,r^{2/3}}} \E T_{u,v} \leq 4r- c'(\frac{\phi}{8}+j)^2r^{1/3},$$
for all $j\geq 0$ with $\frac{\phi}{8}+j\leq 0.9 r^{1/3}$,
and
$$\sup _{u\in B_j', v\in \L_{r,r^{2/3}}} \E T_{u,v} \leq 4r- c'(\frac{\phi}{8}+j)^2r^{1/3},$$
for all $j$ with $0.9r^{1/3}< \frac{\phi}{8}+j \leq 2r^{1/3}$.
Then it follows from Theorem \ref{t:supinf}(ii) (applied to $B_j, \L_{r,r^{2/3}}$ or $B_j', \L_{r,r^{2/3}}$) that for any $j$ and some constant $c''>0$ we have
$$\P\left(\sup _{u\in B_j, v\in \L_{r,r^{2/3}}} T_{u,v} \geq 4r-c''\phi^2r^{1/3}\right) \leq e^{-c''(\phi+2j)^3}.$$
Taking a union bound over $j$, and using symmetry, we get the desired result.
\end{proof}

\subsection{Paths Constrained to be in a Parallelogram}\label{s.thm4.2(iii)}
Proof of Theorem \ref{t:supinf}(iii) is similar to that of Theorem \ref{t:supinf}(i); the proof proceeds through the same steps except versions of Lemma \ref{l:p2sidethick} and Lemma \ref{l:side2sidethick} involving constrained passage times need to be used, which in turn are established using Proposition \ref{p:tf}. We first need a one point estimate for the constrained passage times which is presented next. 

Let $U_{t}$ denote the parallelogram with one pair of opposite sides of length $2tr^{2/3}$ on the lines $\L_0$ and $\L_r$ respectively with respective midpoints $u_0:=(-mr^{2/3},mr^{2/3})$ and $\br$. We have the following lemma.

\begin{lemma}
\label{l:constonepoint}
For each $\psi<1$ and $t>0$, there exists $C,c>0$ depending on them, such that for all $|m|\leq \psi r^{1/3}$, we have for all $x>0$ and $r\geq 1$
{$$\P\left(T^{U_t}_{u_0,\br}-\E T_{u_0,\br} \leq -xr^{1/3}\right)\leq Ce^{-cx^{3/2}}.$$}
\end{lemma}

A couple of remarks are in order. Although we have not stated explicitly the dependence of $C$ and $c$ on $t,$ the reader might observe that the arguments in Lemma \ref{l:realvar} indicates that for small $t$,  we have $\E T^{U_{t}}_{u_0,\br}=\E T_{u_{0},\br}-\Theta(t^{-1}r^{1/3})$ and $\mbox{Var} T^{U_{t}}_{u_0,\br}=\Theta(t^{-1/2}r^{2/3})$ {(see also \cite[Prop 7.5 arXiv Version 2]{BG18})}.  Further, the tail exponent above, though suffices, is not optimal. Since the initial posting of this paper, optimal exponents for deviations of constrained geodesic weights have been derived by the second named author along with Milind Hegde in \cite{bootstrapping}.
\begin{proof}[Proof of Lemma \ref{l:constonepoint}]
Without loss of generality, let us assume that $x$ is sufficiently large and fix it. For $J$, to be chosen appropriately later, let us set 
$(-mr^{2/3},mr^{2/3})=u_0, u_1, \ldots, u_{J}=\br$ to be $J+1$ equispaced points on the straight line joining $(-mr^{2/3},mr^{2/3})$ and $\br$. By ignoring rounding issues, we also assume that each $u_i\in \Z^2$.  By Theorem \ref{t:onepoint} and Proposition \ref{p:tf} it follows that for each $i$ we have

\begin{eqnarray*}
\P\left(T^{U_{t}}_{u_i,u_{i+1}}-\E T_{u_i,u_{i+1}} \leq -2\frac{x}{J^{2/3}} (r/J)^{1/3} \right) &\leq & \P\left(T_{u_i,u_{i+1}}-\E T_{u_i,u_{i+1}} \leq -2\frac{x}{J^{2/3}} (r/J)^{1/3} \right)\\ &+& \P\left(T^{U_{t}}_{u_i,u_{i+1}}\neq T_{u_i,u_{i+1}}\right)\\
&\leq & Ce^{-cx^3/J^{2}}+Ce^{-cJ^2}\leq Ce^{-cx^3/J^2}=Ce^{-cx^{3/2}} 
\end{eqnarray*}
for some $C,c>0$ depending on $t$, {where the penultimate inequality and the final equality is guaranteed by choosing $J=\delta x^{3/4}$ for some $\delta>0$ small and $x$ sufficiently large.} 
For the term $\P\left(T^{U_{t}}_{u_i,u_{i+1}}\neq T_{u_i,u_{i+1}}\right)$, we invoked Proposition \ref{p:tf} for the parallelogram such that 
the midpoints of one pair of edges are $u_i$, $u_{i+1}$ respectively which are $\frac{J-i}{J} (-mr^{2/3},mr^{2/3}) + \frac{i}{J} \br$, and $\frac{J-i-1}{J} (-mr^{2/3},mr^{2/3}) + \frac{i+1}{J} \br$ and side lengths $2tr^{2/3}$. Namely, this is just $U_t$ between the lines $x+y=2ir/J$ and $x+y=2(i+1)r/J$. 
Thus the $\phi$ in the application of Proposition \ref{p:tf} would be $2tJ^{2/3}$.

By taking a union bound over all $i$ and using that 
$$T^{U_t}_{u_0,\mathbf{r}}\geq  \sum_{i} T^{U_{t}}_{u_i,u_{i+1}}$$
it follows that for $x$ sufficiently large we have 
$$\P\left(T^{U_t}_{u_0,\br}-\sum_{i} \E T_{u_i,u_{i+1}} \leq -2xr^{1/3}\right)\leq Ce^{-cx^{3/2}}.$$
{It follows from \eqref{e:mean}} that $\E T_{u_i,u_{i+1}}\ge J^{-1}\E T_{u_0,\br} - CJ^{-1/3}r^{1/3}$ and hence
$$\sum_{i} \E T_{u_i,u_{i+1}}- \E T_{u_0,\br}\geq -CJ^{2/3}r^{1/3} \geq -xr^{1/3}$$
for $x$ sufficiently large and our choice of $J$. The last two displays, combined together, complete the proof. 
\end{proof}

Using Lemma \ref{l:constonepoint} we can now prove an analogue of Lemma \ref{l:p2sidethick-upgraded} for constrained passage times. Recall the set-up of Lemmas \ref{l:p2sidethick} and \ref{l:p2sidethick-upgraded}: in particular recall that for $m,h$ such that $|m|+h<\psi r^{1/3}$, $U_0$ denotes the parallelogram whose one pair of opposite sides are parallel to the lines $x+y=0$, have midpoints $(mr^{2/3},-mr^{2/3})$ and $\br$ respectively and length $2hr^{2/3}$; and $\hat{U}$ is the sub-parallelogram of $U_0$ restricted to the strip $\{0\leq x+y\leq \frac{r}{16}\}$. Finally $U_*, U_{*,1}$ and $U_{*,2}$ denote the sub-parallelograms of $U_0$ restricted to the strips $\{0\leq x+y\leq \frac{9r}{5}\},$ $\{0\leq x+y\leq \frac{9r}{10}\},$ and $\{\frac{11r}{10}\leq x+y\leq 2r\}$ respectively.

We have the following result. 
\begin{lemma}
\label{l:p2sidethickconst}
For each $\psi<1$ and $h>0$, there exists $C,c>0$ depending only on $\psi, h$ such that for $U_*$ as above with $|m|+{10}h\leq \psi r^{1/3}$, we have for all $x>0$ and all $r\geq 1$
$$\P\left( \inf_{u\in U_*}  (T^{U_0}_{u,\br}-\E T_{u,\br}) \leq -xr^{1/3}\right)\leq Ce^{-cx^{3/2}}.$$
\end{lemma}
This proof will be almost identical to the proof of Lemma \ref{l:p2sidethick-upgraded}, except that we shall use Lemma \ref{l:constonepoint} instead of Theorem \ref{t:onepoint} for the lower tail of constrained passage times. 
\begin{proof}[Proof of Lemma \ref{l:p2sidethickconst}]
We assume that $r$ and $x$ are large enough, since otherwise the result follows by taking $C$ large and $c$ small enough.
In this proof we also use $C$ and $c$ to denote large and small constants depending on $\psi$ and $m$, and the specific values can change from line to line.

We first prove that, when $|m|+h\leq \psi r^{1/3}$, we have
\begin{equation}  \label{eq:p2sidethickconst-weak}
\P\left( \inf_{u\in \hat{U}}  (T^{U_0}_{u,\br}-\E T_{u,\br}) \leq -xr^{1/3}\right)\leq Ce^{-cx^{3/2}}.    
\end{equation}
We consider the same tree $\mathcal{T}$ as in the proof of Lemma \ref{l:p2sidethick}. Let $\cA_j$ now denote the event that for all $u\in \mathcal{T}_j$ and for all $v\in \mathcal{T}_{j+1}$ such that the edge $(u,v)$ is present in $\mathcal{T}$, we have 
$$T^{U_0}_{v,u}\geq \E T_{v,u}- x2^{-(9j/10+10)}r^{1/3}.$$
We claim that there exists $C,c>0$ such that for all $x$ sufficiently large
$$\P(\cup_{j}\cA_j^c) \leq Ce^{-cx^{3/2}}.$$
Notice that we took care in constructing $\mathcal{T}$ so that none of the vertices are too close to the boundary of $U_0$. 
Indeed, by our construction of $\mathcal{T}$, it follows that for all $u_j\in \mathcal{T}_j$ and for all $u_{j+1}\in \mathcal{T}_{j+1}$ such that $(u_j,u_{j+1})$ is present in $\mathcal{T}$, the distance from $u_{j+1}$ and $u_j$ to the boundary of $U_0$ in the anti-diagonal direction is at least {$h(4^{(j+1)}+1)^{-1}r^{2/3}$}. 
Thus, the parallelogram with one pair of parallel sides parallel to $\L_0$ with length {$2h r^{2/3}(4^{(j+1)}+1)^{-1}$} and midpoints $u_{j+1}$ and $u_j$ respectively is contained in $U_0$. Hence Lemma \ref{l:constonepoint} applies (the slope condition is satisfied for all large $r$ by \eqref{treeprop21}, \eqref{treeprop22}) showing that for each $(u_{j+1},u_{j})$ as above we have 
$$\P(T^{U_0}_{u_{j+1},u_{j}}-\E T_{u_{j+1},u_j}\leq -y (8^{-j}r)^{1/3})\leq Ce^{-cy^{3/2}}$$
for some $C,c>0$ all $y>0$ and all $r$ sufficiently large.
Using this with $y=2^{j/10-10}x$ we obtain that 
$$\P(T^{U_0}_{u_{j+1},u_j}-\E T_{u_{j+1},u_{j}}\leq -x{2^{-(9j/10+10)}}r^{1/3})\leq Ce^{-cx^{3/2}2^{3j/20}}.$$
Now taking a union bound over all $32^{j+1}$ such edges, and then over all $j=0,1,2,\ldots, J-1$, we get $\P(\cup_{j}\cA_j^c) \leq Ce^{-cx^{3/2}}$.

To establish \eqref{eq:p2sidethickconst-weak}, it remains to show that on $\cap_{0\leq j\leq J} \cA_j$, $\inf_{u\in \hat{U}}  (T^{U_0}_{u,\br}-\E T_{u,\br}) \geq -xr^{1/3}$.
Let's take any $u\in \hat{U}$. Let $u_{J}$ be as in the proof of Lemma \ref{l:p2sidethick} and $u_J, u_{J-1},\ldots,u_0=\br$ be the path from $u_J$ to $\br$ in $\mathcal{T}$.
We have that
$T^{U_0}_{u,\br} \geq \sum_{j=0}^{J-1} T^{U_0}_{u_{j+1},u_j}$.
On $\cap_{0\leq j\leq J} \cA_j$ we also have that
$$\sum_{j=0}^{J-1} \left( T^{U_0}_{u_{j+1},u_j}- \E T_{u_{j+1},u_j} \right) \geq -\frac{x}{2} r^{1/3}.$$
It was already shown in \eqref{eq:p2sidethick:pf1} that
$$\sum_{j=0}^{J-1} \E T_{u_{j+1},u_j} \geq \E T_{u,\br}-\frac{x}{2}r^{1/3}.$$
Adding these two inequalities up we get \eqref{eq:p2sidethickconst-weak}.

Finally, as in the proof of Lemma \ref{l:p2sidethick-upgraded}, 
for each $0\le i \le 359,$ we  consider the same $U_i$ and 
$\hat{U}_i$ and
then the conclusion follows by applying the above result to each $\hat{U}_i$ and $\br$, and taking a union bound.
\end{proof}

Using the above, we have the constrained version of Lemma \ref{l:side2sidethick}. Recall the set-up of Lemma \ref{l:side2sidethick} and in particular the parallelograms $U_0,  U_{*,1}, U_{*,2}$. We have the following result. 
\begin{lemma}
\label{l:side2sidethickconst}
For each $\psi<1$ and $h>0$, there exists $C,c>0$ depending only on $\psi, h$ such that for $U_0, U_{*,1}, U_{*,2}$ as above with $|m|+20h\leq \psi r^{1/3}$, we have for all $x>0$ and $r\geq 1$
$$\P\left( \inf_{u\in U_{*,1}, v\in U_{*,2}} (T^{U_0}_{u,v}-\E T_{u,v}) \leq -xr^{1/3}\right)\leq Ce^{-cx^{3/2}}.$$
\end{lemma}
\begin{proof}The proof of Lemma \ref{l:side2sidethickconst} is identical to that of Lemma \ref{l:side2sidethick} except that all passage times are now replaced by the passage times constrained in the parallelogram $U_0$ (e.g., in the definition of the events $\cA_1$, $\cA_2$) and the application of Lemma \ref{l:p2sidethick-upgraded} is replaced by that of Lemma \ref{l:p2sidethickconst}. The weaker probability bound of $Ce^{-cx^{3/2}}$ in the input (Lemma \ref{l:p2sidethickconst}) as opposed to $Ce^{-cx^{3}}$ in the former (Lemma \ref{l:p2sidethick-upgraded}) is the reason why we obtain the same weaker bound here compared to Lemma \ref{l:side2sidethick}. \end{proof}

We can now provide the proof of Theorem \ref{t:supinf}(iii). We first need an estimate about constrained passage times between point on the sides of the parallelogram $U$, whose one pair of sides lie on $\L_0$ and $\L_r$ with length $2r^{2/3}$ and midpoints $(mr^{2/3},-mr^{2/3})$ and $\br$ respectively.
Let $A_l, A_r$ be the taller `left and right sides', specifically, let $A_l$ be the collection of all vertices $u$, such that $u\in U$ while $u+(-1,1) \not\in U$;
and let $A_r$ be the collection of all vertices $u$, such that $u\in U$ while $u+(1,-1) \not\in U$.
We have the following result concerning passage times between vertices in $A_l$ or $A_r$.
\begin{lemma}  \label{l:boundary-pass}
For each $\psi<1$, there exists $C,c>0$ depending only on $\psi$ such that for all $x, r$ sufficiently large, $|m|<\psi r^{1/3}$ and $U$ as above we have
$$
\P\left(
\inf_{u,v\in A_*, d(v)>d(u)} (T_{u,v}^U - T_{u,v}) \leq -xr^{1/3}
\right) \leq Ce^{-cx}.
$$
Here $A_*$ is $A_l$ or $A_r$.
\end{lemma}

Postponing the proof of the above, we first finish the proof Theorem \ref{t:supinf}(iii) using Theorem \ref{t:supinf}(i) and Lemma \ref{l:boundary-pass}.

\begin{proof}[Proof of Theorem \ref{t:supinf}(iii)]
For $A_l, A_r$ as in Lemma \ref{l:boundary-pass}, we consider the intersection of the events
$$
\inf_{u,v\in A_l, d(v)>d(u)} (T_{u,v}^U - T_{u,v})> -\frac{xr^{1/3}}{3},\quad
\inf_{u,v\in A_r, d(v)>d(u)} (T_{u,v}^U - T_{u,v})
> -\frac{xr^{1/3}}{3},
$$
and
$$
\inf_{u,v\in U: d(v)-d(u)\geq \frac{r}{L}}  (T_{u,v}-\E T_{u,v}) > -\frac{xr^{1/3}}{3}.
$$
By Theorem \ref{t:supinf}(i) and Lemma \ref{l:boundary-pass}, the probability of the intersection is at least $1-Ce^{-cx}$.

Now on the intersection of these events, we consider any $u, v\in U$ with $d(v)-d(u)\geq \frac{r}{L}$.
Let $\gamma$ be the geodesic from $u$ to $v$.
If $\gamma\cap A_l=\emptyset$ let $\gamma'=\gamma$; otherwise, let $u_{l,-}, u_{l,+} \in \gamma \cap A_{\ell}$ be such that $d(u_{l,-})\leq d(u') \leq d(u_{l,+})$ for any $u'\in\gamma\cap A_l$.
Consider the maximum weight path constrained within $U$ from $u_{l,-}$ to $u_{l,+}$, and 
replace the part of $\gamma$ between $u_{l,-}, u_{l,+}$ by it, and we denote this new path from $u$ to $v$ by $\gamma'$.
Next, from $\gamma'$ we construct $\gamma''$ similarly, this time replacing the excursion outside $A_r$.
Specifically, if $\gamma'\cap A_r=\emptyset$ let $\gamma''=\gamma'$; otherwise, let $u_{r,-}, u_{r,+} \in \gamma'\cap A_{r}$ such that $d(u_{r,-})\leq d(u') \leq d(u_{r,+})$ for any $u'\in\gamma'\cap A_r$.
Consider the maximum weight path constrained within $U$ from $u_{r,-}$ to $u_{r,+}$, and 
replace the part of $\gamma'$ between $u_{r,-}, u_{r,+}$ by it, and we denote this new path from $u$ to $v$ as $\gamma''$.

From this construction we have that $\gamma''$ is constrained in $U$, so we have
\[
\begin{split}
T^U_{u,v} - \E T_{u,v} &\geq T^U_{u,v} - T_{u,v} + T_{u,v} - \E T_{u,v} \\
&= (T^U_{u_{l,-}, u_{l,+}} - T_{u_{l,-}, u_{l,+}}) + (T^U_{u_{r,-}, u_{r,+}} - T_{u_{r,-}, u_{r,+}}) + (T_{u,v} - \E T_{u,v})
\\
& > -xr^{1/3}.
\end{split}
\]
with probability at least $1-Ce^{-cx}$. Thus the conclusion follows.
\end{proof}

We finish with the proof of Lemma \ref{l:boundary-pass}.
\begin{proof}[Proof of Lemma \ref{l:boundary-pass}]
We prove for $A_*=A_l$, and by symmetry the other case also follows. In this proof we let $C,c>0$ be constants depending on $\psi$, and their values can change from line to line.

We start by defining parallelograms ``supported" on $A_l.$
Namely,  for $i_1, i_2\in \llbracket 0, 2r \rrbracket$, $i_1<i_2$, we denote by $P_{i_1,i_2}$ the following parallelogram:
it is contained in $U$, with one pair of sides on $x+y=i_1$, $x+y=i_2$ respectively with  length $2^{1/3}(i_2-i_1)^{2/3}$; and it contains $i_2-i_1+1$ vertices in $A_l$ i.e.,  it is the set of all lattice points inside the $(i_2-i_1+1)$ by $2^{1/3}(i_2-i_1)^{2/3}$ continuous parallelogram supported on an interval of size $(i_2-i_1+1)$ on the taller side of $U_0$ thought of as a real interval (whose lattice version is $A_l$).

We next have a comparison estimate between constrained and unconstrained passage times between points in a family of such parallelograms defined in the following way.

Let $j_0$ be the maximum integer satisfying $2^{j_0}<r^{3/4}$.
For each $j=1,2,\ldots, j_0$, consider parallelograms $P_{i_1,i_2}$ for $i_1, i_2\in \{0, 2^{-j}(2r), 2\times  2^{-j}(2r), 3\times 2^{-j}(2r), \ldots, 2r\}$,
satisfying $i_2-i_1 \in \{3\times  2^{-j}(2r), 4\times  2^{-j}(2r), 5\times  2^{-j}(2r)\}$. Let $\cB_{j}$ denote the event that for a fixed $j$ and any {parallelogram $U_0$} in this family, we have 
$$\inf_{u\in U_{1},v\in U_{2}} (T_{u,v}^U-T_{u,v}) \geq -xr^{1/3}$$
where $U_{1}$ and $U_{2}$ are defined as follows: suppose $U_0= P_{i_1,i_2}$ for some $i_1,i_2$, then $U_{1}$ and $U_{2}$ are $U_0$ restricted to the strips $\{i_1\le x+y \le \frac{2i_1+i_2}{3}\}$ and $\{\frac{i_1+2i_2}{3}\le x+y \le i_2\}$, respectively. Notice that for $u\in U_1, v\in U_2$ we have
$$\left\{(T_{u,v}^U-T_{u,v}) \leq -xr^{1/3}\right\}\subseteq  \left\{(T_{u,v}^U-\E T_{u,v}) \leq -\frac{xr^{1/3}}{2}\right\}\cup \left\{(T_{u,v}-\E T_{u,v}) \geq \frac{xr^{1/3}}{2}\right\}.$$
For $r$ sufficiently large, Theorem \ref{t:supinf}(ii) and Lemma \ref{l:side2sidethickconst} apply (while the slope conditions are satisfied for a slightly larger $\psi$), and it follows that  $\P(\cB_j^c)=O(2^{j}e^{-c2^{j}x})$,
and $\P(\cup_j\cB_j^c)=O(e^{-cx})$.

Given the above, consider any $u, v \in A_l$ with $d(u)<d(v)$. We now have the two following cases:

\begin{itemize}
\item
We first assume that $d(v)-d(u)>2^{-j_0+1}(2r)$.
Suppose that $2^{-j}(2r)\leq d(v)-d(u)<2^{-j+1}(2r)$ for some $j\leq j_0-1$.
Then we can find some $i_1,i_2 \in \{0, 2^{-j-1}(2r), 2\times  2^{-j-1}(2r), 3\times 2^{-j-1}(2r), \ldots, 2r\}$, with
$i_2-i_1 \in \{3\times 2^{-j-1}(2r), 4\times 2^{-j-1}(2r), 5\times 2^{-j-1}(2r)\}$, 
such that $i_1\le d(u) < i_1+ 2^{-j-1}(2r) < i_2- 2^{-j-1}(2r) \le d(v) < i_2$.
Thus we can apply the above estimate to conclude that on $\cap_{j}\cB_j $ we have 
$$\inf_{u,v\in A_l, d(v)-d(u)>2^{-j_0+1}(2r)} (T_{u,v}^U - T_{u,v}) \geq -xr^{1/3}.$$ 
\item
Finally we consider those $u,v \in A_l$ with $0<d(v)-d(u)\leq 2^{-j_0+1}(2r)\leq 8r^{1/4}$.
We have $\P(T_{u,v}^U-T_{u,v} \leq -xr^{1/3})< \P(T_{u,v}>xr^{1/3}) <Ce^{-cxr^{1/4}}$ by Theorem \ref{t:onepoint}.
The number of such $u,v$ is at most  $O(r^{5/4})$ since $A_l$ contains $O(r)$ points. By taking a union bound the conclusion follows.
\end{itemize}
\end{proof}

\end{document}